\documentclass[leqno,11pt]{amsart}
\usepackage{amssymb, amsmath}
\usepackage{amssymb,amsfonts}
\usepackage{amsmath,latexsym}
\usepackage{pdfsync}
\usepackage{bbm}
\usepackage{soul,color, xcolor}
\usepackage{eso-pic,graphicx}

\def\cV{{\mathcal V}}

\def\cO{{\mathcal{O}}}

\def\tan{\mbox{tan}}
\def\e{\varepsilon}
\def\t{\langle t\rangle }
\def\s{\langle s\rangle }

\def\f{\frac}
\def\p{\partial}

\def\CN{\mathcal{N}}

\def\n{\mathbf{n}}
\def\btau{\boldsymbol{\tau}}

\numberwithin{equation}{section}
\newtheorem{theorem}{Theorem}[section]
\newtheorem{proposition}[theorem]{Proposition}
\newtheorem{remark}[theorem]{Remark}
\newtheorem{definition}[theorem]{Definition}
\newtheorem{corollary}[theorem]{Corollary}
\newtheorem{lemma}[theorem]{Lemma}

\newcommand{\andf}{\quad\hbox{and}\quad}
\newcommand{\with}{\quad\hbox{with}\quad}

\newcommand{\beq}{\begin{equation}}
	\newcommand{\eeq}{\end{equation}}
\newcommand{\ben}{\begin{eqnarray}}
	\newcommand{\een}{\end{eqnarray}}
\newcommand{\beno}{\begin{eqnarray*}}
	\newcommand{\eeno}{\end{eqnarray*}}

\providecommand{\R}{\mathbb{R}}

\providecommand{\dive}{\mathrm{div} \,}

\date{\today}

\title[ ]{Global small-time approximate null and Lagrangian controllability of the viscous non-resistive MHD system in a $3D$ domain with Navier type boundary conditions  }
\author[J. Liao]{Jiajiang Liao}
\address[J. Liao]
{School of Mathematical Sciences, Beihang University, 102206 Beijing, China. } \email{jjliao@buaa.edu.cn}
\author[F. Sueur]{Franck Sueur}
\address [F. Sueur]
{ Department of Mathematics, Maison du nombre, 6 avenue de la Fonte,
University of Luxembourg, L-4364 Esch-sur-Alzette, Luxembourg.} \email{Franck.Sueur@uni.lu}
\author[P. Zhang]{Ping Zhang}
\address[P. Zhang]
{State Key Laboratory of Mathematical Sciences, Academy of Mathematics $\&$ Systems Science, The Chinese Academy of
	Sciences, Beijing 100190, China, and School of Mathematical Sciences, University of Chinese Academy of Sciences, Beijing 100049, China. } \email{zp@amss.ac.cn}

\subjclass{Primary 93B05; Secondary 35Q35.}
\keywords{Approximate controllability, magnetohydrodynamics, boundary layers, multi-scales asymptotic expansion, well-prepared dissipation method}

\begin{document}
\maketitle

\begin{abstract}
    We consider the incompressible viscous MHD system without magnetic diffusion in a $3D$ bounded domain with Navier type boundary condition.
    We establish the global small-time approximate null controllability and the Lagrangian controllability of the system, in the class of smooth solutions, by  following the approach initiated in \cite{CMS}  to establish the global small-time  null controllability of the incompressible Navier-Stokes equations  in the class of weak solutions and extended in \cite{LSZ1} to establish the global small-time  null and Lagrangian controllability of the incompressible Navier-Stokes equations  in the class of strong solutions.
This approach makes use of controls with an extra fast scale in time and  some corresponding multi-scale asymptotic expansions of the controlled solution. This expansion is constructed by an iterative process which requires some regularity.
The extra-difficulty here is that the MHD system at stake is  mixed hyperbolic-parabolic,  without any regularizing effect on the magnetic field.
Despite our strategy makes use of a quite precise asymptotic expansion, we succeed to cover the case where the initial velocity belongs the Sobolev space  $H^{24}$ and the initial magnetic field belongs to the Sobolev space  $H^8$.

\end{abstract}

\tableofcontents

\section{Introduction}

\subsection{Setting}

Let $\Omega$ be a bounded simply connected smooth domain in $\mathbb{R}^d$ $ (d=2,3)$ and let $\Gamma$ be a non-empty connected open part of the boundary $\p\Omega$. We additionally require that $\p\Omega$ has only one component and is simply connected when $d=3$. We consider the incompressible viscous magnetohydrodynamic (MHD) system without magnetic diffusion inside $\Omega.$ One may check \cite{Ca70,CP76,LL} for
 the physical background of this system, one may also check \cite{AZ5,CMRR,FMR,FMRR} and the references therein for the well-posedness theory to the Cauchy problem of this system. We denote by $u,p,B$ its velocity, pressure and the magnetic field, respectively. We assume that we can act on $\Gamma$ for both $u$ and $B$, while on the rest of the boundary we suppose that the velocity $u$ satisfies Navier slip-with-friction boundary condition and the magnetic field $B$ is tangential to the boundary. Hence $(u,p,B)$ satisfies
\begin{equation}\label{MHD}
\left\{
\begin{aligned}
&\p_t u+u\cdot\nabla u-B\cdot\nabla B-\Delta u +\nabla p=0,\quad &&\text{ in }(0,T)\times\Omega,\\
&\p_t B+u\cdot\nabla B-B\cdot\nabla u=0,\quad &&\text{ in }(0,T)\times\Omega,\\
&\dive u=\dive B=0,\quad &&\text{ in }(0,T)\times\Omega,\\
&u\cdot\n=0,\quad \mathcal{N}(u)=0,\quad B\cdot \n=0\quad &&\text{ on }(0,T)\times\p\Omega\setminus\Gamma,\\
&(u,B)|_{t=0}=(u_0,B_0),\quad&&\text{ in }\Omega,
\end{aligned}
\right.
\end{equation}
where we assume the viscosity coefficient to be one for simplicity.

Let $\cO$ be a bounded smooth extension of the initial domain $\Omega$ such that $\Gamma\subset \cO,\, \p\Omega\setminus\Gamma\subset\p\cO$ and $\p\cO$ is simply connected (see Figure \ref{figure1} for a simple case).
We denote $\n$ to be the outward unit normal vector field on $\p\cO$.
We define the Navier boundary operator $\mathcal{N}(f)$ for a vector function $f$ by
\begin{equation*}
\label{defNf}\mathcal{N}(f):=(D(f)\cdot\n+Mf)_{\tan}, \quad f_{\tan}:=f-(f\cdot\n) \n,\quad  D(f):=\frac{1}{2}(\nabla f+\nabla^T f),
\end{equation*}
where $M$ is a given smooth matrix-valued function defined near the boundary $\p\cO$, which describes the friction.

\subsection{Main results}

\subsubsection{Assumptions on the initial data}

For simplicity, we assume that the initial data $(u_0,B_0)$ satisfies
\begin{subequations} \label{S1eq9}
\begin{gather}
\label{ini1}
\begin{split}
u_0 \text{ is a restriction on $\Omega$ of }u_a\in H^{24}(\cO) \text{ with } \dive u_a=0\text{ in } \Omega,\\
u_a\cdot\n=0\text{ on }\p\cO \text{ and } \CN(u_a)=0 \text{ on }\p\Omega\setminus\Gamma,
\end{split}
\\
\label{ini2}B_0 \text{ is a restriction on $\Omega$ of }B_a\in H^8(\cO) \text{ with } \dive B_a=0 \text{ in } \cO \text{ and } B_a\cdot\n=0 \text{ on }\p\cO.
\end{gather}
\end{subequations}

\begin{figure}
  \centering
  \includegraphics{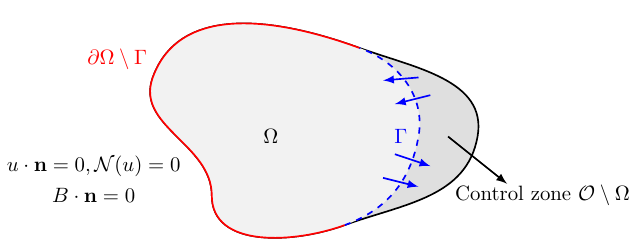}
  \caption{Extension of the physical domain $\Omega\subset\mathcal{O}$ }\label{figure1}
\end{figure}

Instead of imposing the flux condition on the initial data as in \cite{KNV} and \cite{L} to extend the initial data into a larger domain and maintain the divergence free condition and tangent to the boundary, we directly require the initial data $(u_0,B_0)$   being a restriction of $(u_a,B_a)$ defined in a larger domain $\cO$ satisfying \eqref{ini1} and \eqref{ini2}.

\subsubsection{Small-time global approximate null controllability}

Our first main goal is the small-time global approximate null controllability of this system, that is, for a given final time $T>0$ and initial data $(u_0,B_0)$ in some Sobolev spaces, we aim at finding a solution $(u,B)$ of \eqref{MHD}, for some  controls for both $u$ and $B$ on $\Gamma$ (these controls will be implicit in \eqref{MHD}), which
is approximately equal to zero at  time $T$.


The main results state as follows:

\begin{theorem}\label{mainth}
{\sl
Let $T>0$ and we assume that the initial data $(u_0,B_0)$
satisfies \eqref{ini1} and \eqref{ini2}.
Then for any $\epsilon>0$, there exists a solution $(u,p,B)$ of \eqref{MHD} with
$$u\in C([0,T];H^1(\Omega))\cap L^2((0,T);H^2(\Omega)), B\in C([0,T];H^1(\Omega)),$$
which satisfies
\begin{equation} \label{petitesse}
\|(u(T,\cdot),B(T,\cdot))\|_{H^1(\Omega)}+\|(u(T,\cdot),B(T,\cdot))\|_{L^{\infty}(\Omega)}<\epsilon.
\end{equation}
}
\end{theorem}

Here and below, we always  use for simplicity the following notation
 for any Banach space $X$ and $f_1,f_2\in X$,
\begin{equation*}\label{f1f2}
\|(f_1,f_2)\|_X:=\big(\|f_1\|^2_X+\|f_2\|_X^2\big)^{\frac{1}{2}}.
\end{equation*}

\begin{remark}
{\sl
We focus on proving Theorem \ref{mainth} in three dimensional case, the two dimensional case can be obtained similarly.
}
\end{remark}


\subsubsection{Small-time global approximate Lagrangian  controllability}
Our second result is the following small-time global approximate Lagrangian controllability theorem.

\begin{theorem}\label{LAC}
{\sl
Let $\gamma_0,\gamma_1$ be two Jordan surfaces included in $\Omega$ such that $\gamma_0$ and $\gamma_1$ are isotopic and surround the same volume.
Let $T>0,$ we assume that the initial data $(u_0,B_0)$ satisfies \eqref{ini1} and \eqref{ini2}.
Then for any $\epsilon >0$, there exist a time $T_*\in (0,T)$ so that the system \eqref{MHD} has a solution $(u,p,B)$ on $[0,T_\ast]$ with $u\in C([0,T_*];H^1(\Omega))\cap L^2((0,T_*);H^2(\Omega)), B\in C([0,T_*];H^1(\Omega)),$
and
\begin{subequations} \label{S1eq10}
\begin{gather}
\label{lc1}\forall\, t\in [0,T_*], \ \phi^u(t,0,\gamma_0)\subset\Omega,\\
\label{lc2}\|\phi^u(T_*,0,\gamma_0)-\gamma_1\|_{L^{\infty}}<\epsilon,
\end{gather}
\end{subequations}
holds up to reparameterization, where $\phi^u$ is the flow map associated with $u$ defined by
\begin{equation}\label{flowu}
\left\{
\begin{aligned}
\p_t\phi^u(t,s,x)&=u(t,\phi^u(t,s,x)),\qquad &&\forall\, t,s\in [0,T_*], \ x\in\Omega,\\
\phi^u(s,s,x)&=x,\qquad &&\forall s\in [0,T_*], \ x\in \Omega.
\end{aligned}
\right.
\end{equation}

Moreover, if we impose more regularity on the initial data, say $u_a\in H^{24+2k}(\cO)$ and $B_a\in H^{8+k}(\cO)$ for an integral $k\in \mathbb{N}_+,$  then \eqref{lc2} can be improved to be
\begin{align}
\label{lc3}\|\phi^{u}(T_*,0,\gamma_0)-\gamma_1\|_{C^k}<\epsilon.
\end{align}
}
\end{theorem}

 The proof of Theorem \ref{LAC} will be given in Section \ref{sec-lagr}.

\subsection{Related literature}\label{relli}
\subsubsection{On the well-posedness of the non-resistive MHD system.} For the viscous, non-resistive MHD equations,
Fefferman {\it et al} \cite{FMR} established the local existence and uniqueness of solutions in $\mathbb{R}^d,$ $ d=2,3,$ for initial data in $H^s(\mathbb{R}^d)$ with $s>d/2$. These results have been extended in \cite{FMRR} to the case when the initial data $B_0\in H^s(\mathbb{R}^d)$ and $u_0\in H^{s-1+\e}(\mathbb{R}^d)$ for $s>d/2$ and any $0<\e<1$.  Chemin {\it et al } proved in \cite{CMRR} the local existence of weak solution when the initial data is in some Besov spaces, that is $B_0\in B^{d/2}_{2,1}(\mathbb{R}^d)$ and $u_0\in B^{d/2-1}_{2,1}(\mathbb{R}^d),$ and the uniqueness of such solution in two space dimensions
was proved by Wan in \cite{Wan}. Many papers have established the existence of global-in-time solutions for initial data sufficiently close to certain equilibrium solutions,
cf. \cite{LXZ, Zhang} in $\mathbb{R}^2$ and \cite{LZ,AZ5} in $\mathbb{R}^3$.

\subsubsection{On the Navier-Stokes equations}
The boundary controllability problem for Navier-Stokes equations was first raised by J. L. Lions in \cite{Lions}. In Lions' original question, the control is acted in a small subdomain (this situation is similar to controlling only part of the boundary) and the boundary condition of the uncontrolled part is Dirichlet boundary condition. 

When the initial data is close to the final state, Fern\'{a}ndez-Cara {\it et al} proved in \cite{FGIP} the small time local controllability to the trajectories of Navier-Stokes equations by using some parabolic Carleman estimates.
Guerrero \cite{Guerrero} has extended this result to the case of the Navier boundary condition.

For the case of the Navier boundary condition on the uncontrolled part of the boundary,  Coron {\it et al} proved in \cite{CMS} the small time global controllability to the trajectories of Navier-Stokes equations by using Coron's return method and ``well-prepared dissipation" method in a $L^2$ framework. We extend this result to the class of strong solutions in \cite{LSZ1} by constructing higher order expansions of the approximate solutions. As a byproduct, we proved the Lagrangian approximate controllability.

For Dirichlet boundary condition, Coron {\it et al} proved in \cite{CMSZ} the small time global exact controllability of the Navier-Stokes equations in a rectangle at the expense of a small phantom force inside the domain. We got a similar theorem in a cylinder in \cite{LSZ2}.

\subsubsection{On the MHD equations}
For the controllability of MHD equations, Badra \cite{Bad} proved the small time local controllability to the trajectories by using Carleman's estimates. However, what he imposed are Dirichlet boundary condition for the velocity and Navier boundary condition for the magnetic field, besides here we don't have magnetic diffusion, thus we cannot use Badra's result to get exact null controllability.

For ideal MHD equations, Rissel and Wang proved in \cite{RW1} the null-controllability by adding an additional magnetic force. Kukavica {\it et al} provide a sufficient and necessary condition such that the MHD equations is exact controllable in a rectangle. Recently, Kucavica and O\.za\'nski proved in \cite{KO} the exact boundary controllability for $3D$ incompressible ideal MHD equations in general domains.

For MHD equations with viscosity and magnetic diffusion, Rissel and Wang proved in \cite{RW} the approximate controllability to a given trajectory in the weak sense. The boundary conditions are coupled Navier boundary condition for both velocity and magnetic field. They only constructed the main boundary layer of order $\cO(\sqrt{\e})$ and the null controllability for the linearized equations. So, their results only hold in a weak framework. We construct higher-order boundary layers and adapt the approximate results in both $H^1$ and $L^{\infty}$.

Compared to \cite{LSZ1}, since the equation for the magnetic field $B$ is a transport equation, we no longer can gain any regularity through the equation. Therefore, we wish to minimize the regularity for the initial data.
For this purpose, we use a new method: the continuity method to estimate the $L^{\infty}$ norm of the normal derivatives of the remainder. It is more efficient and powerful than the method we used in \cite{LSZ1}, which was more based on a maximum principle.
In addition, thanks to a more accurate version of \cite[Proposition 5.10]{LSZ1}, that is, Proposition \ref{2inf}, we manage to close the estimates by simply expanding $(u^{\e},B^{\e})$ to order $\cO(\e^{\frac{3}{2}})$ instead of $\cO(\e^2)$ and reducing the minimum regularity required for the initial data to $H^{24}(\cO)\times H^8(\cO)$. For more explanations, see Remark \ref{ueaid} and Remark \ref{indexpl}.

\subsubsection{Status of the main results here VS the literature}
Theorem \ref{mainth}  thus extends the result obtained in   \cite{LSZ1} for  the incompressible Navier-Stokes equations to the
case of the viscous non-resistive MHD system, both with Navier type boundary conditions for a domain with a  arbitrary shape and in the class of strong solutions.
The extra-difficulty  is that the system at stake here  is  mixed hyperbolic-parabolic, without any regularizing effect on the magnetic field.
 Albeit the strategy makes use of a quite precise asymptotic expansion, we succeed to cover the case where the initial velocity belongs to the Sobolev space  $H^{24}$ and the initial magnetic field belongs to  the Sobolev space  $H^8$.

Theorem \ref{LAC} extends the Lagrangian controllability result to the case of MHD equations from the results \cite{GH1,GH2,HK} in the case of the Euler equations, the result \cite{GH3} in the case of the steady Stokes equations and the result \cite{LSZ1} in the case of the Navier-Stokes equations with Navier boundary conditions.

\section{Scheme of the proof of Theorem \ref{mainth}}\label{Sect2}


Firstly, we use the time scale \eqref{tsca1} to stretch the time interval from $[0,T]$ to $[0, {T}/{\e}]$. After that, we construct a series of approximate solutions $(u^{\e}_{app},B^{\e}_{app})$ in Section \ref{cexp} that have a nice decay rate in time. The last step is to estimate the remainder  parts $(r^{\e},R^{\e})$ in Section \ref{estrem} and to make sure that they are globally small.

As in \cite{CMS,CMSZ,L,LSZ1,LSZ2,Marbach,RW}, we use the ``well-prepared dissipation" method to construct approximate solutions. This method was proposed by Marbach \cite{Marbach} to prove the small-time global exact controllability for the Burgers equation.  It consists of two stages.
\begin{itemize}
    \item
During the first stage, we use Coron's return method to construct a base solution $u^0$ of Euler equations which satisfies the flushing condition \eqref{flush}. We can use this flushing property of $u^0$ to construct a solution $u^2$ of the linearized Euler equation which drives the initial data $u_0$ to zero at time $T$.

Since the boundary condition of MHD equations does not coincide with the ones of the Euler equation, some boundary layers appear. Due to the Navier boundary condition on  the velocity, the main boundary layer $v^1$ is of the order $\cO(\sqrt{\e})$
and the equation for $v^1$ is a linear equation.
   \item
The second stage is to use the flushing property of $u^0$ to construct some control functions over time $t\in [0,T]$ such that the new data $v^1(T,\cdot,\cdot)$ satisfies certain conditions which entail a sufficiently nice decay for $v^1$. For $v^1$, this condition for $v^1(T,\cdot,\cdot)$ is the vanishing moments condition, see \cite{CMS}. In order to get higher-order regularities of the remainder, we have to construct higher order expansions. When $t\geq T$, the equation for higher-order boundary layers $v^i,i\geq 2,$ is a Poisson equation with non-zero source terms. In this case, we give the corresponding condition in \cite{LSZ1} which ensures the fast decay for $v^i,i\geq 2.$
\end{itemize}

For the estimates of the remainders, we use the conormal Sobolev spaces and the continuity method.

\subsection{Time scaling}\label{scal}
We introduce a smooth function $\varphi:\mathbb{R}^3\rightarrow\mathbb{R}$ such that $\varphi=0$ on $\p\cO$, $\varphi>0$ in $\cO$ and $\varphi<0$ outside $\overline{\cO}$. Moreover, we can assume that $\varphi(x)=\text{dist}(x,\p\cO)$ in a small neighborhood of $\p\cO$. Hence we can extend $\n$ by $-\nabla \varphi$ to $\cO.$ We define
\begin{equation*}
\cV_{\delta}:=\bigl\{\ x\in \cO |\ 0\leq \varphi(x)<\delta\ \bigr\}.
\end{equation*}
Thus there exists a $\delta_0>0$ such that $|\n|=1$ in $\cV_{\delta_0}$.

For any $\epsilon>0$, we want to find some regular control functions $\xi, \Xi,\sigma$ which are supported in $\overline{\cO\setminus\Omega}$, such that the associated solution $(u,p,B)$ of
\begin{equation}\label{MHDO}
\left\{
\begin{aligned}
&\p_t u+u\cdot\nabla u-B\cdot\nabla B-\Delta u +\nabla p=\xi,\quad  && \text{ in }(0,T)\times\cO,\\
&\p_t B+u\cdot\nabla B-B\cdot\nabla u+\sigma B=\Xi,\quad  && \text{ in }(0,T)\times\cO,\\
&\dive u=\sigma,\quad \dive B=0,\quad &&\text{ in }(0,T)\times\cO,\\
&u\cdot\n=0,\quad \mathcal{N}(u)=0,\quad B\cdot \n=0,\quad && \text{ on }(0,T)\times\p\cO,\\
&(u,B)|_{t=0}=(u_a,B_a),\quad  && \text{ in }\cO,
\end{aligned}
\right.
\end{equation}
satisfies
\begin{equation*}
\|(u(T,\cdot),B(T,\cdot))\|_{H^1(\cO)}+\|(u(T,\cdot),B(T,\cdot))\|_{L^{\infty}(\cO)}<\epsilon.
\end{equation*}

The reason why we add the term $\sigma B$ in the equation of the magnetic field is that we want to maintain the divergence free property for $B$.
From the construction of $\Xi$ below in Section \ref{mfe}, it is a divergence free function.
By taking divergence and make use of $\dive u=\sigma$, $\dive B$ satisfies
\begin{align*}
\p_t \dive B+u\cdot\nabla \dive B+\sigma \dive B=\dive \Xi=0.
\end{align*}
 Hence, the divergence-free property of $B$ can be propagated.
Since the control functions $\sigma, \xi$ and $\Xi$ are supported in $\overline{\cO\setminus \Omega}$, the system \eqref{MHDO} reduces to the system \eqref{MHD} in the original space domain $\Omega$.

We remark that, we can indeed view $\tilde{\Xi}=\Xi-\sigma B$ as the total control force of $B$. While if we don't add the term $\sigma B,$ in order to maintain $\dive B=0,$  the control force $\tilde{\Xi}$ is not divergence free and cannot  be easily constructed as below.

We perform the time scaling
\beq \label{tsca1}
\begin{aligned}
u^{\e} (t,x)&:=\e u(\e t,x),\quad      & p^{\e}(t,x) &:=\e^2p(\e t,x),\quad  &B^{\e}(t,x)&:=\e B(\e t,x),\\
\xi^{\e}(t,x)&:=\e^2 \xi(\e t,x),\quad &\Xi^{\e}(t,x)&:=\e^2\Xi(\e t,x),\quad& \sigma^{\e}(t,x)&:=\e\sigma(\e t,x).
\end{aligned}\eeq

Then $(u^{\e},p^{\e},B^{\e})$ satisfies
\begin{equation}\label{MHDe}
\left\{
\begin{aligned}
&\p_t u^{\e}+u^{\e}\cdot\nabla u^{\e}-B^{\e}\cdot\nabla B^{\e}-\e\Delta u^{\e} +\nabla p^{\e}=\xi^{\e},\quad &&\text{ in }(0,{T}/{\e})\times\cO,\\
&\p_t B^{\e}+u^{\e}\cdot\nabla B^{\e}-B^{\e}\cdot\nabla u^{\e}+\sigma^{\e}B^{\e}=\Xi^{\e},\quad &&\text{ in }(0,{T}/{\e})\times\cO,\\
&\dive u^{\e}=\sigma^{\e},\quad \dive B^{\e}=0,\quad&&\text{ in }(0,{T}/{\e})\times\cO,\\
&u^{\e}\cdot\n=0,\quad \mathcal{N}(u^{\e})=0,\quad B^{\e}\cdot \n=0\quad &&\text{ on }(0,{T}/{\e})\times\p\cO,\\
&(u^{\e},B^{\e})|_{t=0}=\e(u_a,B_a),\quad&&\text{ in }\cO.
\end{aligned}
\right.
\end{equation}
Through time scaling, we stretch the time interval from $[0,T]$ to $[0,{T}/{\e}]$.
For some regular control functions $\xi^{\e}, \Xi^{\e}, \sigma^{\e}$ and some regular initial data $(u_a,B_a)$, it follows from well-known works such as \cite{JN,FMR,FMRR}, for each $\e>0,$ there is a unique corresponding local strong solution.

To prove \eqref{petitesse} and end the proof of Theorem \ref{mainth},
it is sufficient to find some regular control functions $\sigma^{\e},\ \xi^{\e},\ \Xi^{\e},$  which are supported in $\overline{\cO\setminus\Omega}$, such that
\begin{equation}\label{uBeT}
\big\|\big(u^{\e}({T}/{\e},\cdot),B^{\e}({T}/{\e},\cdot)\big)\big\|_{H^1(\cO)}
+\big\|\big(u^{\e}({T}/{\e},\cdot),B^{\e}({T}/{\e},\cdot)\big)\big\|_{L^{\infty}(\cO)}=o(\e).
\end{equation}

\subsection{Return method}

\begin{lemma}\label{lmu0}
{\sl There exists a solution $(u^0,p^0,\nu^0,\sigma^0 )\in C^{\infty}([0,T]\times\overline{\cO}; \mathbb{R}^3\times\mathbb{R}\times\mathbb{R}^3\times\mathbb{R})$ to the system:
\begin{equation}\label{euler0}
\begin{aligned}
\partial_t u^0+u^0\cdot\nabla u^0+\nabla p^0=&\ \nu^0, \quad & &\mbox{ in }(0,T)\times\cO,\\
\dive u^0=&\ \sigma^0,  \quad & &\mbox{ in }(0,T)\times\cO,\\
u^0\cdot \mathbf{n}=&\ 0, & &\mbox{ on }(0,T)\times\partial\cO,\\
u^0(0,\cdot)=&\ 0, & & \mbox{ in }\cO,\\
u^0(T,\cdot)=&\ 0,  & &\mbox{ in }\cO,
\end{aligned}
\end{equation}
such that the flow $\phi^{u^0}$ defined by $\partial_t \phi^{u^0}(t,s,x)=u^0(t,\phi^{u^0}(t,s,x))$ for $t,s\in [0,T]$ and $\phi^{u^0}(s,s,x)=x$ satisfies
\begin{equation} \label{flush}
\forall\, x\,\in \overline{\cO},\,\exists\, t_x\in (0,T),\quad\phi^{u^0}(t_x,0,x)\in \overline{\cO}\setminus\overline{\Omega}.
\end{equation}
Moreover, $u^0$ can be chosen such that:
  \begin{equation*}
\nabla\times u^0=0\quad\mbox{ in }[0,T]\times \overline{\cO}.
\end{equation*}
In addition, $\nu^0$ and $\sigma^0 $ are supported in $\overline{\cO}\setminus\overline{\Omega}$, $(u^0,p^0,\nu^0,\sigma^0 )$ are compactly supported in $(0,T)$. In the sequel, we shall implicitly extend them by zero after T when it is necessary.}
\end{lemma}

This lemma is the key argument for many papers concerning the small time global exact controllability of Euler equations, cf. ~\cite{Cor1} for $2D$ simply connected domains, ~\cite{Cor2} for general $2D$ domains when $\Gamma$ intersects all connected components of $\partial\Omega$,
\cite{Gla2}  for $3D$ simply connected domains, ~\cite{Gla1} for general $3D$ domains  when $\Gamma$ intersects all connected components of $\partial\Omega$.
 Let us also refer to \cite{Gla3} and some applications to the small time global exact or approximate controllability of Navier-Stokes equations, cf. \cite[Lemma 2]{CMS}, \cite[Lemma 2.9]{LSZ1}, and MHD equations \cite[Lemma 3.2]{RW}.

\subsection{Boundary layers and multi-scale asymptotic expansion }\label{expan}

We seek for a solution $(u^{\e},p^{\e},B^{\e})$ to \eqref{MHDe} with expansion of the form
\begin{subequations} \label{S2eq9}
\begin{align}
\begin{split}
\label{expu}
u^{\e}(t,x)=&\ u^0(t,x)+\sqrt{\e}v^1\bigl(t,x,\frac{\varphi(x)}{\sqrt{\e}}\bigr)+\e U^2\bigl(t,x,\frac{\varphi(x)}{\sqrt{\e}}\bigr)\\ &+\e^{\frac{3}{2}}U^3\bigl(t,x,\frac{\varphi(x)}{\sqrt{\e}}\bigr)
+\e^{\frac{3}{2}}r^{\e}(t,x),
\end{split}\\
\label{expB}B^{\e}(t,x)=&\ \e B^2(t,x)+\e^{\frac{3}{2}}R^{\e}(t,x),
\end{align}
\end{subequations}
where
\begin{align*}
U^2(t,x,z)=&\ u^2(t,x)+\nabla\phi^2(t,x)+v^{2}(t,x,z)+\n(x) w^2(t,x,z),\\
U^3(t,x,z)=&\ \nabla\phi^3(t,x)+v^3(t,x,z)+\n(x) w^3(t,x,z).
\end{align*}
The function $u^0$ is determined by Lemma \ref{lmu0} and the functions $u^2, B^2$ are the solutions of linearized equations. All are supported in $[0,T]$. The boundary layer profiles $v^i,1\leq i\leq 3$, decay fast to zero when $z$ goes to infinity. The $w^j,j=2,3,$ are normal part of the boundary layer profiles and decay fast  to zero as $|z|\to \infty$. The $\nabla \phi^2, \nabla\phi^3$ are backflow terms due to the boundary condition \eqref{MHDe}$_4$.

\subsection{Null control of the linearized equations}

We introduce a cut-off function $\beta\in C^{\infty}(\mathbb{R}_+;[0,1])$ such that $\beta(t)=1$ when $t\leq\frac{T}{3}$ and $\beta(t)=0$ when $t\geq \frac{2T}{3}$.

\begin{proposition}\label{lnel}
{\sl
For an initial data $u_a\in H^{24}(\cO)$ with $\dive u_a=0$ in $\cO$ and $u_a\cdot\n=0$ on $\p\cO$, there exist control functions $\nu\in C^7([0,T];H^{16}(\cO)), \sigma^2\in C^{\infty}([0,T];H^{23}(\cO))$ supported in $\overline{\cO\setminus\Omega}$ and a solution $u\in C^7([0,T];H^{17}(\cO))$ to the following linearized equations of Euler system:
\begin{equation}\label{eqle}
\begin{cases}
\partial_t u+u^0\cdot\nabla u+u\cdot\nabla u^0+\nabla p=\nu, \qquad \text{ in }[0,T]\times\cO,\\
\dive u=\sigma^2, \qquad \text{ in }[0,T]\times\cO,\\
u\cdot \mathbf{n}=0, \qquad \ \ \text{ on }[0,T]\times\p\cO,\\
u(0,\cdot)=u_a, \quad \  \ \text{ in }\cO,
\end{cases}
\end{equation}
such that $u(T,\cdot)=0$ in $\cO.$}

\end{proposition}

\begin{proposition}\label{prB2}
{\sl
For an initial data $B_a\in H^8(\cO)$ with $\dive B_a=0$ in $\cO$ and $B_a\cdot\n=0$ on $\p\cO$, there exists a control function $\mu^2\in C([0,T];H^{8}(\cO))$, supported in $\overline{\cO\setminus\Omega}$, and a solution $B^2\in C([0,T];H^{8}(\cO))$ to the following
linearized equation of the magnetic equation in \eqref{MHDe}:
\begin{equation}\label{eqB2}
\begin{cases}
\p_t B^2+u^0\cdot\nabla B^2-B^2\cdot\nabla u^0+\sigma^0 B^2=\mu^2,\qquad\text{ in }[0,T]\times\cO,\\
\dive B^2=0,\qquad\text{ in }[0,T]\times\cO,\\
B^2\cdot \n =0,\qquad\ \text{ on }[0,T]\times\p\cO,\\
B^2(0,\cdot)=B_a,\quad\text{ in }\cO,
\end{cases}
\end{equation}
such that $B^2(T,\cdot)=0$ in $\cO.$

}
\end{proposition}

Both Proposition \ref{lnel} and Proposition \ref{prB2} will be proved in Section \ref{cexp}.

\subsection{Well-prepared dissipation method}
We set
\begin{equation*}
u^0_{\flat}(t,x) :=\frac{u^0(t,x)\cdot \mathbf{n}(x)}{\varphi(x)}\qquad\mbox{ in }\mathbb{R}_+\times\cO,
\end{equation*}
where  $u^0$  is a smooth function determined  by Lemma \ref{lmu0}
and we observe  that $u^0_{\flat}$ is smooth in $\overline{\cO}$ since $u^0\cdot\n=0$ on the boundary. We refer to \cite{iftimie} for more details.
Let $\mathcal{A}^0 =\mathcal{A}^0 (t,x)$  be a smooth field of $3 \times 3$ matrices such that for any $v$ in $\mathbb{R}^3$,
 \begin{equation*} \label{defA}
 \mathcal{A}^0 v:=v\cdot\nabla u^0+(u^0\cdot \nabla \mathbf{n}\cdot v) \mathbf{n}-(v\cdot\nabla u^0\cdot \mathbf{n})\mathbf{n} .
\end{equation*}
The key property associated with $\mathcal{A}^0$ is that for a regular vector field $v(t,x)$,
 \begin{equation*} \label{defB-k}
(u^0\cdot\nabla v+\mathcal{A}^0 v )\cdot \mathbf{n} = u^0\cdot\nabla  ( v \cdot \mathbf{n}) \quad \mbox{ in }\mathcal{V}_{\delta_0}.
\end{equation*}

Let us introduce the following weighted Sobolev spaces.
\begin{definition}
{\sl
For $z\in \mathbb{R}$, we denote $\langle z\rangle:=\sqrt{1+z^2}$
and for $s$ and $q\in \mathbb{N},$ we set
\begin{equation*}
H^{s}_{q}(\mathbb{R_+}):=\Bigl\{\ f\in H^s(\mathbb{R_+}): \ \sum_{0\leq j\leq s}\int_{\mathbb{R_+}}\langle z\rangle^{2q}|\partial_z^jf(z)|^2dz<+\infty\ \Bigr\},
\end{equation*}
endowed with its natural associated norm.
 Similarly we can also define the functional spaces $H^{s}_{q}(\mathbb{R})$ and its norm:
\begin{equation*}
 \|f\|_{H^{s}_{q}(\mathbb{R})} := \Bigl(\sum_{0\leq j\leq s} \int_{\mathbb{R}}\langle z\rangle^{2q}|\partial_z^jf(z)|^2dz \Bigr)^\frac12 .
\end{equation*}
}
\end{definition}

Let $\mathcal{S}(\mathbb{R})$  the Schwartz space of smooth functions on $\mathbb{R}$ whose derivatives are rapidly decreasing.
Let us denote by $\mathcal{S}(\mathbb{R}_+)$ the set of restrictions on $\mathbb{R}_+$
of the functions of $\mathcal{S}(\mathbb{R})$.
\begin{definition}\label{DC}
{\sl
Let $k\in\mathbb{N},\gamma>0$ and $X$ be a Banach space with norm $\|\cdot\|_X$.
We define the space $C^k_{\gamma}(\mathbb{R}_+;X)$ to be the set of the functions  $f\in C^k(\mathbb{R}_+;X)$ such that
$$\|f\|_{C^k_{\gamma}(\mathbb{R}_+;X)} :=\sup_{t\geq 0,0\leq j\leq k}\bigl(\|\partial_t^jf(t)\|_X\t^{\gamma}\bigr)<+\infty,$$ where
$\t=\sqrt{1+t^2}$ and
$$C^k(\mathbb{R}_+;X) :=\bigl\{\ f:\partial_t^jf\in C(\mathbb{R}_+;X), \ 0\leq j\leq k\ \bigr\}.$$
}\end{definition}

Below we denote $[x]$ to be the floor integer part  of a real number $x$.
\begin{proposition}[Proposition 3.3 of \cite{LSZ1}]\label{propv}
{\sl Let $\gamma>0$ and $s,q,k,p \in\mathbb{N}$  with $k\geq 1$. Set
\begin{equation}  \label{G1}
\begin{aligned}
n &:=[q/2+\gamma],\quad   &\tilde{\gamma}&:= 2n+3, \quad   &\tilde{s}&:=  s+2k+2n, \qquad   \tilde{q}:=   2n+3,\\
k' &:=[(s+1)/2]+k+n ,\quad   &  \tilde{k}&:= k+k'-1, \quad   &\tilde{p}&:=  p+k'+1.
\end{aligned}
\end{equation}
Assume that
$$f\in C^{\tilde{k}}_{\tilde{\gamma}}(\mathbb{R}_+; H^{\tilde{p} } (\cO ; H^{\tilde{s}}_{\tilde{q}}  (\mathbb{R}_+) ))
   \quad\text{ and }\quad
g\in C^{\tilde{k}}_{\tilde{\gamma}}(\mathbb{R}_+;H^{\tilde{p}}(\cO)) ,$$
satisfy
\begin{equation*}
f(t,x,z)\cdot \mathbf{n}(x)=g(t,x)\cdot \mathbf{n}(x)=0, \qquad \forall t\geq 0,x\in \cO, z\in \mathbb{R}_+,
\end{equation*}
and $f(t,x,z)$ and $ g(t,x) $ are supported in $ \mathcal{V}_{\delta} $ as a function of  $ x $.

Then there exists an initial data
\begin{equation*}
v_0\in H^{p+2}(\cO;C_0^{\infty}(\overline{\mathbb{R}_+}))
\end{equation*}
which is compatible to $f$ and $g$, a control function
\begin{equation*}
\xi\in C^{k-1}(\mathbb{R}_+;H^p (\cO;\mathcal{S}(\mathbb{R}_+)))
\end{equation*}
and an associate layer function
\begin{equation*}
v \in  C^{k}_{\gamma}(\mathbb{R}_+;H^p (\cO; H^s_q (\mathbb{R}_+)))
\end{equation*}
satisfies
\begin{equation*}\label{eqv}
\begin{cases}
\partial_t v+u^0\cdot\nabla v+\mathcal{A}^0 v-u^0_{\flat}z\partial_z v-\partial_z^2 v=\xi+f, \quad t\geq 0, x\in \cO, z\in \mathbb{R}_+,\\
\partial_z v|_{z=0}=g,\quad t\geq 0, x\in \cO,\\
v\cdot \mathbf{n}=0,\qquad \ t\geq 0, x\in \cO, z\in \mathbb{R}_+,\\
v|_{t=0}=v_0,\quad\ \ x\in \cO, z\in \mathbb{R}_+.
\end{cases}
\end{equation*}

Moreover  there is a continuous function ${S}:\mathbb{R}_+\rightarrow\mathbb{R}_+$ satisfying $\delta\leq {S}(\delta)$ for any positive $\delta$, and
 $\xi$
  is supported in $(\overline{\cO}\setminus\overline{\Omega})\cap\mathcal{V}_{{S}(\delta)}$ as a function of  $x$ and is compactly supported in $(0,T)$ as a function of  time $t$, and satisfies
 $\xi(t,x,z)\cdot \mathbf{n}(x)=0$, for all $ t \in (0,T)$, $x\in (\overline{\cO}\setminus\overline{\Omega})\cap\mathcal{V}_{{S}(\delta)} $ and $z\in \mathbb{R}_+$,
  and  $v$ is supported in $\mathcal{V}_{{S}(\delta)}$ as a function of  $x$ .

  If $f$ and $g$ are both supported away from $t=0$ as a function of time $t$, then so does $v$ and $v_0=0$.
 }
\end{proposition}

\subsection{Conormal Sobolev spaces}\label{css}
Let us define the conormal Sobolev spaces and we refer to \cite{Gamblin} and references therein  for more details.
We introduce a cut-off function $\chi\in C^{\infty}(\overline{\cO})$ such that $\chi(x)=0$ when $\varphi(x)\geq \delta_0$ and $\chi(x)=1$ when $\varphi(x)\leq \frac{\delta_0}{2}$, where $\delta_0$ is selected in Section \ref{scal}, and the vector fields set by
\begin{align*}
\mathfrak{W} :=  \Bigl\{
&\boldsymbol{\tau}^0:=\varphi \mathbf{n},\boldsymbol{\tau}^1:=\bigl(0,-\partial_3\varphi,\partial_2\varphi\bigr)^\top,\
\boldsymbol{\tau}^2:=\bigl(\partial_3\varphi,0,-\partial_1\varphi\bigr)^\top,\\
&\boldsymbol{\tau}^3:=\bigl(-\partial_2\varphi,\partial_1\varphi,0\bigr)^\top,
\begin{aligned}[t]
\boldsymbol{\tau}^4:
&=\bigl(\partial_3(x_3(1-\chi)),0,-\partial_1(x_3(1-\chi))\bigr)^\top,\\
\boldsymbol{\tau}^5:&=\bigl(\partial_2(x_1(1-\chi)),-\partial_1(x_1(1-\chi)),0\bigr)^\top \Bigr\}.
\end{aligned}
\end{align*}
It is easy to observe that vectors $\boldsymbol{\tau}^j$ are tangential to $\partial\cO$, $0\leq j\leq 5$. Moreover, for $1\leq j\leq 5$, $\boldsymbol{\tau}^j\cdot \mathbf{n}=0$ in $\mathcal{V}_{\delta_0/2}$ and $\dive \boldsymbol{\tau}^j=0$ in $\cO$. Now we define the tangential derivatives
\begin{equation*}\label{defZ}
Z_j :=  \boldsymbol{\tau}^j\cdot\nabla\ \ \mbox{for}\ 0\leq j\leq 5 \andf Z^{\alpha} :=  Z_0^{\alpha_0}\cdots Z_5^{\alpha_5} \ \mbox{for}\ \alpha=(\alpha_0,\cdots,\alpha_5).
\end{equation*}

We define the conormal Sobolev spaces
  \begin{equation*}
H^m_{\rm co}(\cO) :=\left\{ u\in L^2(\cO):Z^{\alpha}u\in L^2(\cO),|\alpha|\leq m \right\}
\end{equation*}
with norm
  \begin{equation*}\label{scn}
\|u\|_m :=\Bigl(\sum_{|\alpha|\leq m}\|Z^{\alpha}u\|^2_{L^2}\Bigr)^{\frac{1}{2}}.
\end{equation*}
In the same way, we set
  \begin{equation*}
\|u\|_{k,\infty} :=\sum_{|\alpha|\leq k}\|Z^{\alpha}u\|_{L^{\infty}}
\end{equation*}
and we say $u\in W_{co}^{k,\infty}$ if $\|u\|_{k,\infty}$ is finite.

\subsection{Remainder estimates}

For the remainder $(r^{\e},R^{\e})$ defined by \eqref{expu} and \eqref{expB}, we shall prove in Section \ref{thpf} that, there exists a constant $C>0$ such that
\begin{equation}\label{S2eqq}
\e^{\frac{3}{2}}\|(r^{\e},R^{\e})(t)\|_{H^1(\cO)}\leq C \e^{\frac{5}{4}},\quad
\|(r^{\e},R^{\e})(t)\|_{L^{\infty}(\cO)}\leq C,\quad \forall t\in [0, {T}/{\e}].
\end{equation}
By combining \eqref{S2eqq} with the nice decay properties of the approximate solutions, we arrive at \eqref{uBeT}.
The details will be presented in Section  \ref{thpf}.


\section{Construction of the approximate solutions}\label{cexp}

In this section, we shall construct the multi-scale asymptotic expansion mentioned in Section \ref{expan}. That is to construct layer profile functions: $v^i,w^j,1\leq i\leq 3,j=2,3,$ which are supported near the boundary and have nice decay properties for time $t$, and regular functions: $u^2,B^2,\nabla\phi^j,j=2,3,$ which are supported in $[0,T]$.
We will start by a lemma dealing with the multiplication of two-layer profile functions or a regular function times a layer profile function.

\begin{lemma}[Lemma 4.1 of \cite{LSZ1}]\label{UV}
{\sl
Let $\gamma>0, k,p,s,q \in \mathbb{N}_+$ with $p\geq 4$ and $s\geq 2.$
Let $U\in C^{k}_{\gamma}(\mathbb{R}_+;H^p(\cO))$ and $V,\tilde{V}\in C^{k}_{\gamma}(\mathbb{R}_+;H^p(\mathcal{O};H^s_q(\mathbb{R}_+)))$ be scalar functions.
 Then, one has
\begin{gather*}\label{eq-UV}
UV,\,\tilde{V}V\in C^{k}_{\gamma}(\mathbb{R}_+;H^p(\mathcal{O};H^s_q(\mathbb{R}_+))).
\end{gather*}
}
\end{lemma}

\subsection{Velocity expansion}\label{cexp1}
The process to construct the approximate solutions is similar to the one that in \cite{LSZ1}.
We define some indices according to Proposition \ref{propv}. Let $\gamma>1, k,p,s,q\in \mathbb{N}_+$ be some fixed constants with $p\geq 8,s\geq 7,q\geq 2$.
We define the mapping $\frak{\Gamma}:$ $\frak{\Gamma}(\gamma,k,p,s,q)=(\tilde{\gamma},\tilde{k},\tilde{p},\tilde{s},\tilde{q})$, where $\tilde{\gamma},\tilde{k},\tilde{p},\tilde{s},\tilde{q}$ are given respectively by \eqref{G1}. We denote
\begin{subequations} \label{S3eq1}
\begin{align}
(\gamma_3,k_3,p_3,s_3,q_3):&=(\gamma,k,p,s,q),\\
\label{index2}(\gamma_2,k_2,p_2-1,s_2-1,q_2-2):&=\frak{\Gamma}(\gamma_3,k_3,p_3,s_3,q_3),\\
(\gamma_1,k_1,p_1-1,s_1-1,q_1-2):&=\frak{\Gamma}(\gamma_2,k_2,p_2,s_2,q_2).
\end{align}
\end{subequations}
In order to close the estimates of the remainder in the next section, we take
\begin{align}\label{inp}
(\gamma,k,p,s,q)=\bigl(\frac{5}{4},1,8,7,2\bigr).
\end{align}
For more explanations, one may check Remark \ref{indexpl} below. By construction, we have
\begin{align*}
(\gamma_2,k_2,p_2,s_2,q_2)=(7,7,17,14,9)\quad\text{ and }\quad (\gamma_1,k_1,p_1,s_1,q_1)=(25,31,44,51,27).
\end{align*}

By Proposition \ref{propv}, $S:\mathbb{R}_+\rightarrow \mathbb{R}_+$ is a continuous function satisfying $S(0)=0$ and $S(\delta)\geq \delta$ for any $\delta>0$. We can choose and fix a small $\delta>0$ such that
\begin{align*}
\delta_1:=S(\delta),\quad \delta_2=S(\delta_1),\quad \text{and }\quad \delta_3=S(\delta_2),
\end{align*}
which satisfy $0< \delta\leq \delta_1\leq \delta_2\leq \delta_3\leq \delta_0,$ for $\delta_0$ being defined in Section \ref{css}.

We assume that the initial data $u_a$ satisfies
\begin{align}\label{inua}
u_a\in H^{24}(\cO).
\end{align}
The choice of the regularity index for $u_a$ is due to Proposition \ref{lnel}, see Remark \ref{indexpl} for more explanation.

\noindent$\bullet$ \underline{Order $\mathcal{O}(\sqrt{\e}).$} \vspace{0.2cm}

\noindent\underline{Construction of $v^1$.} Let $\chi_1$ be a cut-off function such that $\chi_1(x)=1$ for $x\in \mathcal{V}_{\delta/2}$ and $\chi_1(x)=0$ for $x\in \cO\setminus\mathcal{V}_{\delta}$. Set
\begin{equation*}
g^1(t,x):=2\mathcal{N}(u^0)(t,x)\chi_1(x).
\end{equation*}
Since $u^0$ is a smooth function compactly supported in $(0,T)$, hence $g^1$ is also smooth and compactly supported in $(0,T)$. By the definition of $\mathcal{N}$ and $\chi_1$, function $g^1$ is supported in $\mathcal{V}_{\delta_1}$ and  satisfies $g^1\cdot\n=0$.
By Proposition \ref{propv}, we can find a control function
\begin{equation*}
\xi^1\in C^{k_1-1}(\mathbb{R}_+;H^{p_1}(\mathcal{O};\mathcal{S}(\mathbb{R}_+)))
\end{equation*}
and a layer profile function
\begin{equation}\label{S3eqv1}
v^1\in C^{k_1}_{\gamma_1}(\mathbb{R}_+;H^{p_1}(\mathcal{O};H^{s_1}_{q_1}(\mathbb{R}_+)))
\end{equation}
satisfy
\begin{equation}\label{defv1}
\begin{cases}
\p_t v^1+u^0\cdot\nabla v^1+\mathcal{A}^0v^1-u^0_{\flat}z\p_z v^1-\p_z^2 v^1=\xi^1,\quad \text{ in } \mathbb{R}_+\times\cO\times\mathbb{R}_+,\\
\p_z v^1=g^1,\quad \text{ on }\mathbb{R}_+\times\cO\times\{0\},\\
v^1\cdot\n=0,\quad \text{ in }\mathbb{R}_+\times\cO\times\mathbb{R}_+,\\
v^1=0,\qquad \ \text{ on }\{0\}\times\cO\times \mathbb{R}_+.\\
\end{cases}
\end{equation}
Moreover the control $\xi^1$ is supported in $\overline{\mathcal{O}}\setminus \overline{\Omega}\cap \mathcal{V}_{\delta_1}$,
compactly supported in $(0,T)$, the
layer profile function $v^1$ is supported in $\mathcal{V}_{\delta_1}$ and is supported away from $t=0$,  and they satisfy $\xi^1\cdot\n= v^1\cdot\n=0$ for any $t\in\mathbb{R}_+,x\in\cO,z\in\mathbb{R}_+$.

\vspace{0.2cm}
\noindent$\bullet$ \underline{Order $\mathcal{O}(\e)$.}\vspace{0.2cm}

\begin{lemma}[see Appendix A of \cite{CMS}]\label{lmomg}
{\sl
In $3D$ case, for $p\in \mathbb{N}_+$ and  an initial data $u_a\in H^p(\cO), $ there exists a control function $\nu\in \cap_{k=0}^{p-1} C^k([0,T];H^{p-1-k}(\cO))$ supported in $\overline{\cO\setminus\Omega}$ and a solution $\omega\in \cap_{k=0}^{p-1} C^k([0,T];H^{p-1-k}(\cO))$ to
the following system:
\begin{equation}\label{eqomg}
\left\{
\begin{aligned}
&\p_t\omega+\nabla\times(\omega\times u^0)=\nabla\times\nu,\qquad  &&\text{ in }(0,T)\times\cO,\\
&\dive \omega=0,\qquad                                             &&\text{ in }(0,T)\times\cO,\\
&\omega(0,\cdot)=\nabla\times u_a,\qquad                           &&\text{ in }\cO,
\end{aligned}
\right.
\end{equation}
such that $\omega(T,\cdot)=0.$
}
\end{lemma}
\begin{remark}
{\sl
For $2D$ case, the equation \eqref{eqomg} shall be replaced by
\begin{equation*}
\left\{
\begin{aligned}
&\p_t \omega+u^0\cdot\nabla \omega+(\dive u^0)\omega={\rm curl}\,\nu,\qquad  &&\text{ in }(0,T)\times\cO,\\
&\dive \omega=0,\qquad                                                       &&\text{ in }(0,T)\times\cO,\\
&\omega(0,\cdot)={\rm curl}\, u_a,\qquad                                     &&\text{ in }\cO.
\end{aligned}
\right.
\end{equation*}
}
\end{remark}

\begin{proof}[\bf Proof of Proposition \ref{lnel}]We focus on proving the $3D$ case, the $2D$ case can be proved similarly. For initial data $u_a\in H^p(\cO)$, by virtue of Lemma \ref{lmomg}, we can find a control function $\nu\in\cap_{k=0}^{p-1} C^k([0,T];H^{p-1-k}(\cO)),$ which is supported in $\overline{\cO\setminus\Omega},$  and a solution $\omega\in \cap_{k=0}^{p-1}C^k([0,T];H^{p-1-k}(\cO)),$  which satisfies \eqref{eqomg}. Let $\sigma^2:=\beta\dive u_a\in C^{\infty}[0,T],H^{p-1}(\cO)$, then we can find a solution $u\in \cap_{k=0}^{p}C^k([0,T];H^{p-k}(\cO))$ to
the system:
\begin{equation*}
\left\{
\begin{aligned}
&\dive u=\sigma^2,\qquad &&\text{ in }(0,T)\times\cO,\\
&\nabla\times u=\omega,\qquad &&\text{ in }(0,T)\times\cO,\\
&u\cdot\n=0,\qquad     && \text{ on }(0,T)\times\p\cO,
\end{aligned}
\right.
\end{equation*}
We refer to \cite[Theorem \uppercase\expandafter{\romannumeral 4}.4.13]{Boyer} for its proof. Evaluating at time $t=T,$ we find $\omega(T,\cdot)=0$ and $\beta=0$. Thus, $u(T,\cdot)=0$. Specifically, when $u_a\in H^{24}(\cO)$ with  $\dive u_a=0$ in $\Omega$, $u_a\cdot\n=0$ on $\p\cO$, we can find control functions $\nu\in C^7([0,T];H^{16}(\cO)),$ $\sigma^2=\beta \dive u_a$ supported in $\overline{\cO\setminus\Omega}$ and a solution $u\in C^7([0,T];H^{17}(\cO))$ to \eqref{eqle} such that $u(T,\cdot)=0.$
\end{proof}

\vspace{0.2cm}
\noindent\underline{Construction of $u^2$.}
By Lemma \ref{lmu0} we have $\nabla\times u^0=0$ in $\cO$, thus $-\Delta u^0+\nabla\dive u^0=\nabla\times(\nabla\times u^0)=0$, so that $\Delta u^0=\nabla \sigma^0$ is smooth and is supported in $\overline{\cO}\setminus\overline{\Omega}$ and can be absorbed into control terms. Recall that $\beta$ is a cut-off function with $\beta(t)=1$ when $t\leq \frac{T}{3}$ and $\beta(t)=0$ when $t\geq\frac{2T}{3}.$
By using Proposition \ref{lnel}, we can find control functions $\nu^2\in C^{k_2}(\mathbb{R}_+;H^{p_2-1}(\cO)),\sigma^2\in C^{\infty}([0,T];H^{23}(\cO))$ supported in $\overline{\cO\setminus\Omega}$, a solution $u^2\in C^{k_2}(\mathbb{R}_+;H^{p_2}(\cO))$ such that
\begin{equation*}\label{u2}
\begin{cases}
\partial_t u^2+u^0\cdot\nabla u^2+ u^2\cdot\nabla u^0+\nabla p^2=\nu^2+\Delta u^0,\quad\text{ in } \mathbb{R}_+\times \cO,\\
\dive  u^2=\sigma^2,\quad\text{ in } \mathbb{R}_+\times \cO,\\
 u^2\cdot \mathbf{n}= 0 ,\quad\ \ \text{ on }\mathbb{R}_+\times \partial\cO,\\
 u^2=u_a,\qquad\ \, \text{ on } \{0\}\times \cO.
\end{cases}
\end{equation*}
Moreover, the functions $\nu^2,\sigma^2$ and $u^2$ are supported in $[0,T]$.

\vspace{0.2cm}
\noindent\underline{Construction of $w^2$.}
We set
\begin{equation*}\label{intw2}
w^2(t,x,z):=-\int_z^{+\infty}\dive_x v^1(t,x,z')\,dz'.
\end{equation*}
Then
\begin{equation}\label{S3eqw2}
w^2\in C^{k_1}_{\gamma_1}(\mathbb{R}_+;H^{p_1-1}(\mathcal{O};H^{s_1}_{q_1-2}(\mathbb{R}_+))).
\end{equation}
Indeed, since for $t\in\mathbb{R}_+,x\in \cO,$
\begin{align*}
\|\langle z\rangle^{q_1-2}w^2(t,x,z)\|_{L^2_z(\mathbb{R}_+)}
&\leq \|\langle z\rangle^{-1}\int_z^{+\infty}\langle z\rangle^{q_1-1}|\dive_x v^1(t,x,z')|\,dz'\|_{L^2_z(\mathbb{R}_+)}\\
&\leq \|\langle z\rangle^{-1}\|^2_{L^2_z(\mathbb{R}_+)}\|\langle z\rangle^{q_1}\dive_x v^1(t,x,z)\|_{L^2_z(\mathbb{R}_+)},
\end{align*}
and $\p_z w^2=\dive_x v^1$, \eqref{S3eqw2} can be easily obtained by \eqref{S3eqv1}. Moreover, $w^{2}$ is supported in $\mathcal{V}_{\delta_1}$ and is supported away from $t=0$. Similar to the proof in \cite[Section 6.1]{sueur}, we find that
\begin{equation}\label{S3eq2}
\int_{\p\mathcal{O}}w^2(t,x,0)\,dx=0.
\end{equation}

\vspace{0.2cm}
\noindent\underline{Construction of $\phi^2$.}
Let $\phi^2$ be a solution of the following Neumann problem,
\begin{equation*}
\left\{
\begin{aligned}
&\Delta \phi^2=0,\qquad                   &&\text{ in }\cO,\\
&\p_{\n}\phi^2=-w^2(t,x,0),\qquad         &&\text{ on }\p\cO.
\end{aligned}
\right.
\end{equation*}
Thanks to \eqref{S3eq2}, there exists a unique solution $\phi^2\in C^{k_1}_{\gamma_1}(\mathbb{R}_+;H^{p_1}(\cO))$ up to a constant and $\phi^2$ is supported away from $t=0$.

\vspace{0.2cm}
\noindent\underline{Construction of $\pi^2$.}
We set
\begin{equation*}
\pi^{2}(t,x,z):=-\int_z^{\infty}(-u^0\cdot\nabla\n\cdot v^1+v^1\cdot\nabla u^0\cdot \n)(t,x,z')\,dz'.
\end{equation*}
Then $\pi^{2}\in C^{k_1}_{\gamma_1}(\mathbb{R}_+;H^{p_1}(\mathcal{O};H^{s_1}_{q_1-2}(\mathbb{R}_+)))$ and satisfies $\p_z \pi^2=-u^0\cdot\nabla\n\cdot v^1+v^1\cdot\nabla u^0\cdot\n.$
Moreover, $\pi^{2}$ is supported in $\mathcal{V}_{\delta_1}$ and is supported away from $t=0$.

\vspace{0.2cm}
\noindent\underline{Construction of $v^2$.}
Let
\begin{align*}
f^2  :=\ &-\left(\nabla\pi^2+v^1\cdot\nabla v^1-w^2\partial_zv^1 +2\p_{\n}\partial_z v^1-\Delta\varphi \partial_z v^1\right)_{\tan}\\
 \ &-\left(u^0\cdot\nabla(\n w^2)+w^2\p_{\n} u^0\right)_{\tan},\\
 g^2 :=\ &2\CN(v^1|_{z=0})\chi_1(x).
\end{align*}
By Lemma \ref{UV}, we find that $f^2\in C^{k_1}_{\gamma_1}(\mathbb{R}_+;H^{p_1-1}(\mathcal{O};H^{s_1-1}_{q_1-2}(\mathbb{R}_+)))$, $g^2\in C^{k_1}_{\gamma_1}(\mathbb{R}_+;H^{p_1-1}(\cO))$ and they satisfy the conditions of Proposition \ref{propv}, that is, $f^2$ and $g^2$ are supported in $\mathcal{V}_{\delta_1}$ and are supported away from $t=0$ and satisfy $f^2(t,x,z)\cdot \mathbf{n}(x)=g^2(t,x)\cdot \mathbf{n}(x)=0$ for any $t\geq 0,x\in \mathcal{O} $ and $z\geq 0.$
Therefore there exist $\xi^2\in C^{k_2-1}_{\gamma_2}(\mathbb{R}_+;H^{p_2}(\cO;\mathcal{S}(\mathbb{R}_+)))$ and a solution $v^2\in C^{k_2}_{\gamma_2}(\mathbb{R}_+;H^{p_2}(\mathcal{O};H^{s_2}_{q_2}(\mathbb{R}_+)) )$ to
\begin{equation*}\label{u22}
\begin{cases}
\partial_t v^2+u^0\cdot\nabla v^2+\mathcal{A}^0 v^2-u^0_{\flat}z\partial_z v^2
-\partial_z^2 v^2=\xi^2+f^2,\quad \mbox{ in }\mathbb{R}_+\times\cO\times\mathbb{R}_+,\\
\partial_z v^2|_{z=0}=g^2,\quad \mbox{ on }\mathbb{R}_+\times\cO\times\{z=0\},\\
v^2|_{t=0}=0,\qquad\ \, \mbox{ on }\cO\times\mathbb{R}_+ .
\end{cases}
\end{equation*}
Moreover, $\xi^2$ is supported in $(\overline{\mathcal{O}}\setminus\overline{\Omega})\cap\mathcal{V}_{\delta_2}$ and is compactly supported in $(0,T)$, and $v^2$ is supported in $\mathcal{V}_{\delta_2}$ and is supported away from $t=0$, and $\xi^2\cdot \mathbf{n}=v^2\cdot\mathbf{n}=0$.

We denote
\begin{align*}
U^2(t,x,z):=u^2(t,x)+\nabla\phi^2(t,x)+v^2(t,x,z)+\n(x)w^2(t,x,z).
\end{align*}

\noindent$\bullet$ \underline{Order $\mathcal{O}(\e^{\frac{3}{2}})$.}\vspace{0.2cm}

\noindent\underline{Construction of $w^3$.}
We set
\begin{equation*}
\label{w3}
w^3 (t,x,z):=-\int_z^{\infty}\dive  (v^2+\n w^2)(t,x,z')dz'.
\end{equation*}
Then $\partial_z w^3=\dive (v^2+\n w^2)$ and $w^3\in C^{k_2}_{\gamma_2}(\mathbb{R}_+;H^{p_2-1}(\mathcal{O};H^{s_2}_{q_2-2}(\mathbb{R}_+))$ is supported in $\mathcal{V}_{\delta_2}$ and is supported away from $t=0$. Moreover, $w^3$ satisfies
\begin{equation}\label{iw3}
\int_{\partial\cO} w^3(t,x,0)dx=0.
\end{equation}

\vspace{0.2cm}
\noindent\underline{Construction of $\phi^3$.}
Let $\phi^3$ be a solution of the following Neumann problem:
\begin{equation*}\label{phi3}
\left\{
\begin{aligned}
&\Delta \phi^3=0,\quad                  &&\text{ in }\mathcal{O},\\
&\partial_{\mathbf{n}}\phi^3=-w^3(t,x,0),\quad &&\text{ on } \partial\mathcal{O}.
\end{aligned}
\right.
\end{equation*}
Thanks to (\ref{iw3}), there exists a unique solution $\phi^3\in C^{k_2}_{\gamma_2}(\mathbb{R}_+;H^{p_2}(\mathcal{O}))$ up to a constant and $\phi^3$ is supported away from $t=0$.

\vspace{0.2cm}
\noindent\underline{Construction of $\pi^3$.}
We set
\begin{eqnarray*}
\pi^3(t,x,z) :=-\int_z^{+\infty}\Big(\partial_t w^2+u^0\cdot \nabla (\n w^2)\cdot\n-u^0\cdot\nabla \mathbf{n}\cdot v^2+ (v^2+\n w^2)\cdot\nabla u^0\cdot \mathbf{n} \\
+v^1\cdot\nabla v^1\cdot \mathbf{n}-u^0_{\flat}z'\partial_z w^2-\partial_z^2 w^2+\partial_\mathbf{n} \pi^2\Big)(t,x,z')dz'.
\end{eqnarray*}
Then $\pi^3\in C^{k_2}_{\gamma_2}(\mathbb{R}_+;H^{p_2}(\mathcal{O};H^{s_2}_{q_2-2}(\mathbb{R}_+))$ and
\begin{equation*}
\label{pi3}
\begin{split}
\partial_z\pi^3=\ &\partial_t w^2+u^0\cdot \nabla (\n w^2)\cdot\n-u^0\cdot\nabla \mathbf{n}\cdot v^2+ (v^2+\n w^2)\cdot\nabla u^0\cdot \mathbf{n}\\
\ &+v^1\cdot\nabla v^1\cdot \mathbf{n}-u^0_{\flat}z\partial_z w^2-\partial_z^2 w^2+\partial_\mathbf{n} \pi^2.
\end{split}
\end{equation*}
Moreover, $\pi^3 $ is supported in $\mathcal{V}_{\delta_2}$ and is supported away from $t=0$.

\vspace{0.2cm}
\noindent\underline{Construction of $v^3$.}

Let
\beq\label{defg3}
\begin{split}
 f^3 :=\ &-\left(\nabla\pi^3+u^0\cdot\nabla(\n w^3)+w^3\p_{\n} u^0+v^1\cdot\nabla U^2+U^2\cdot\nabla v^1\right)_{\tan}\\
&+\left(U^2\cdot\n\p_z(v^2+\n w^2)+w^3\p_z v^1+\Delta v^1\right)_{\tan}\\
&-\left(2\p_{\n}\p_z(v^2+\n w^2)+\Delta\varphi\p_z(v^2+\n w^2)\right)_{\tan}\\
g^3 :=\ &2\CN\big(u^2+\nabla \phi^2+(v^2+w^2\mathbf{n})|_{z=0}\big)\chi_1(x).
\end{split}\eeq
From Lemma \ref{UV} and the fact that $u^2\in C^{k_2}([0,T];H^{p_2}(\cO))$, we find that
\begin{align*}
f^3\in C^{k_2}_{\gamma_2}(\mathbb{R}_+;H^{p_2-1}(\mathcal{O};,H^{s_2-1}_{q_2-2}(\mathbb{R}_+))
,\qquad g^3\in C^{k_2}_{\gamma_2}(\mathbb{R}_+;H^{p_2-1}(\cO)),
\end{align*}
and satisfy $f^3(t,x,z)\cdot \mathbf{n}(x)=g^3(t,x)\cdot \mathbf{n}(x)=0$ for any $t\geq 0,x\in \mathcal{O} $ and $z\geq 0.$
Moreover $f^3$ and $g^3$ are supported in $\mathcal{V}_{\delta_2}$. Then, by using Proposition \ref{propv}, there exist
 $\xi^3\in C^{k_3-1}_{\gamma_3}(\mathbb{R}_+;H^{p_3}(\cO;\mathcal{S}(\mathbb{R}_+)))$ , $ v^3\in C^{k_3}_{\gamma_3}(\mathbb{R}_+;H^{p_3}(\mathcal{O};H^{s_3}_{q_3}(\mathbb{R}_+))$ and $v^3_0\in H^{p_3+2}(\cO;C_0^{\infty}(\overline{\mathbb{R}_+}))$ such that
\begin{equation*}\label{u3}
\begin{cases}
\partial_t v^3+u^0\cdot\nabla v^3+\mathcal{A}^0 v^3-u^0_{\flat}z\partial_z v^3
-\partial_z^2 v^3=\xi^3+f^3\quad\mbox{ in }\ \mathbb{R}_+\times\cO\times\mathbb{R}_+,\\
\partial_z v^3|_{z=0}=g^3\quad\mbox{ on }\mathbb{R}_+\times\cO\times\{z=0\},\\
 v^3|_{t=0}=v^3_0\qquad\mbox{ on }\ \cO\times\mathbb{R}_+.
\end{cases}
\end{equation*}
Moreover, $\xi^3$ is supported in $(\overline{\mathcal{O}}\setminus\overline{\Omega})\cap\mathcal{V}_{\delta_3}$, $v^3$ is supported in $\mathcal{V}_{\delta_3}$ and $\xi^3\cdot \mathbf{n}=v^3\cdot\mathbf{n}=0$.

We denote
$$
U^3(t,x,z):=\nabla\phi^3(t,x)+v^3(t,x,z)+\n(x)w^3(t,x,z).
$$

\vspace{0.2cm}
\noindent\underline{Construction of $\pi^4$.}
We set
\begin{align*}
\pi^4(t,x,z) :=-&\int_z^{+\infty}\Big(\partial_t w^3+u^0\cdot \nabla (\n w^3)\cdot\n-u^0\cdot\nabla \mathbf{n}\cdot v^3+ (v^3+\n w^3)\cdot\nabla u^0\cdot \mathbf{n} \\
&-u^0_{\flat}z\p_zw^3-\p_zw^3+\p_{\n}\pi^3+v^1\cdot\nabla U^2\cdot\n+U^2\cdot\nabla v^1\cdot\n-U^2\cdot\n \p_z U^2\cdot\n\\
&-w^3\p_zv^1-\Delta v^1+\p_{\n}\p_z(v^2+\n w^2)-\Delta\varphi\p_z(v^2+\n w^2)\Big)(t,x,z')\,dz'.
\end{align*}
Then one has
\begin{align}\label{defpi4}
\pi^4\in C^{k_3}_{\gamma_3}(\mathbb{R}_+;H^{p_3}(\mathcal{O};H^{s_3}_{q_3-2}(\mathbb{R}_+)),
\end{align}
and
\begin{align*}
\p_z\pi^4=\ &\partial_t w^3+u^0\cdot \nabla (\n w^3)\cdot\n-u^0\cdot\nabla \mathbf{n}\cdot v^3+ (v^3+\n w^3)\cdot\nabla u^0\cdot \mathbf{n}\\
&-u^0_{\flat}z'\p_zw^3-\p_zw^3+\p_{\n}\pi^3+v^1\cdot\nabla U^2\cdot\n+U^2\cdot\nabla v^1\cdot\n-U^2\cdot\n \p_z U^2\cdot\n\\
&-w^3\p_zv^1-\Delta v^1+\p_{\n}\p_z(v^2+\n w^2)-\Delta\varphi\p_z(v^2+\n w^2)
\end{align*}
Moreover, $\pi^4 $ is supported in $\mathcal{V}_{\delta_3}$ and is supported away from $t=0$.

\begin{remark}\label{ueaid}
{\sl
Actually, for any fixed indexes $\gamma>1,k,p,s,q\in\mathbb{N}_+$, we can choose a control function
$\xi^1\in C^{k-1}_{\gamma}((0,T);H^p(\cO;\mathcal{S}(\mathbb{R}_+)))$ and a solution $v^1\in C^k_{\gamma}(\mathbb{R}_+;H^p(\cO;H^s_q(\mathbb{R}_+)))$ to system \eqref{defv1} and satisfies the above properties. Roughly speaking, we can choose $\xi^1$ such that $v^1$ is smooth enough and has any fixed decay rate. Similarly, as long as $(\gamma_1,k_1,p_1,s_1,q_1)$ is chosen large enough, the layer functions $v^2, v^3$ can be smooth enough and have any fixed decay rate. Thus $w^2,w^3,\phi^2,\phi^3,\pi^2,\pi^3,f^2,g^2$ have the same properties. While this property is not true for $f^3,g^3$ and $\pi^4$. They involve the regularity of $u^2$. Eventually, the regularity of $f^3,g^3$ and $\pi^4$ depend on $u_a$.

}
\end{remark}

\subsection{Construction of $B^2$.}\label{mfe}

\begin{proof}[\bf Proof of Proposition \ref{prB2}]
Since $B_a\in H^8(\cO), \dive B_a=0$ and $\cO$ is a simply-connected domain, there exists a $\Psi_a\in H^{9}(\cO)$ such that $B_a=\nabla\times\Psi_a$ for $3D$ case and $B_a=\nabla^{\perp}\Psi_a  =(-\p_y\Psi_a,\p_x\Psi_a)^{T}$ for $2D$ case. We divide the rest proof into 3 steps.

\vspace{0.2cm}
\noindent$\bullet$ \underline{Step 1: construction of an appropriate partition of unity.}\vspace{0.2cm}

From the flushing condition \eqref{flush} and  the continuity of $\phi^{u^0}$,
\begin{align*}
\forall x\in \overline{\cO}, \exists t_{x}\in (0,T), \exists r_x>0,\exists \epsilon_x>0,
\text{ such that }\forall x\in \mathcal{B}(x,r_x), |t-t_x|\leq \epsilon_x, \phi^{u^0}(t,0,x)\notin\overline{\Omega}.
\end{align*}
 Then by the compactness of $\overline{\cO}$, we can find finite balls $\mathcal{B}_l=\mathcal{B}(x_l,r_{x_l})$ for $1\leq l\leq L$ covering $\overline\cO$ and $\epsilon:=\min\{\epsilon_{x_l}|1\leq l\leq L\}>0$, such that
\begin{align}\label{flcom}
\forall l\in \{1,\cdots,L\},\exists t_{x_l}\in (\epsilon,T-\epsilon),\forall |t-t_{x_l}|<\epsilon, \phi^{u^0}(t,0,\mathcal{B}_l)\notin \overline{\Omega}.
\end{align}
We can choose a finite partition of unity described by smooth functions $\alpha_l$ for $1\leq l\leq L$ with $0\leq \alpha_l(x)\leq 1$ and $\sum_{l=1}^L\alpha_l=1$ on $\overline{\cO}.$ Let $\beta_{\epsilon}:\mathbb{R}\rightarrow\mathbb{R}$ be a smooth function with $\beta_{\epsilon}(t)=1$ when $t\leq -\epsilon$ and $\beta_{\epsilon}(t)=0$ when $t\geq \epsilon.$

\vspace{0.2cm}
\noindent$\bullet$ \underline{Step 2: local solvability of the uncontrolled equations.}\vspace{0.2cm}

For $1\leq l\leq L,$ let $B_l$ solve equations
\begin{equation}\label{eqBl}
\left\{
\begin{aligned}
&\p_t B_l+u^0\cdot\nabla B_l-B_l\cdot\nabla u^0+\sigma^0 B_l=0,\qquad &&\text{ in }(0,T)\times\cO,\\
&B_l|_{t=0}=\nabla\times(\alpha_l \Psi_a),\qquad  &&\text{ in }\cO.
\end{aligned}\right.
\end{equation}
The last equation will be replaced by $B_l|_{t=0}=\nabla^{\perp}(\alpha_l\Psi_a)$ when the dimension equals 2.
Since $\Psi_a\in H^9(\cO)$ and $u^0,\sigma^0$ are smooth, there exists a unique solution $B_l\in C([0,T];H^8(\cO))$ satisfying \eqref{eqBl}.
Moreover, since $\alpha_l$ is supported in $\mathcal{B}_l$, so does $B_l|_{t=0}$.

Let us show that $B_l$ satisfies
\begin{align}\label{cdBl}
\dive B_l=0\text{ in }\cO\quad \text{and}\quad B_l\cdot\n=0\text{ on } \p\cO.
\end{align}

Thanks to $\dive u^0=\sigma^0,$ by taking the divergence of the first equation of \eqref{eqBl}, we obtain
\begin{align*}
\p_t\dive B_l+u^0\cdot\nabla B_l+\sigma^0\dive B_l=0.
\end{align*}
Then we get,  by using energy methods, that the divergence condition $\dive B_l=0$ propagates.
Similarly, since $\nabla \n$ is a symmetric matrix, we find that
\begin{align*}
\p_t(B_l\cdot\n)+u^0\cdot\nabla (B_l\cdot\n)-B_l\cdot\nabla (u^0\cdot\n)+\sigma^0 B_l\cdot\n=0.
\end{align*}
It follows from Lemma \ref{optan} that
\begin{align*}
B_l\cdot\nabla =B_l\cdot \btau^1Z_1+B_l\cdot \btau^2Z_2+B_l\cdot \btau^3Z_3+B_l\cdot\n\p_{\n}.
\end{align*}
Since $Z_i$ are tangential derivatives and $u^0\cdot\n=0$ on $\p\cO,$ we find that $Z_i(u^0\cdot\n)=0$ on $\p\cO$, $1\leq i\leq 3.$ Hence
there holds
\begin{align*}
\p_t(B_l\cdot\n)+u^0\cdot\nabla (B_l\cdot\n)-B_l\cdot\n\p_{\n} (u^0\cdot\n)+\sigma^0 B_l\cdot\n=0,\qquad\text{ on }\p\cO.
\end{align*}
Then $B_l\cdot\n=0$ holds for all time on $\p\cO$ once we have proved that
\begin{align}\label{psia}
\nabla\times(\alpha_l\Psi_a)\cdot\n=0,\qquad \text{ on }\p\cO.
\end{align}

When the dimension equals 2,  since we assumed that $\cO$ is a simply connected domain, the boundary $\p\cO$ is connected. We deduce from $B_a\cdot\n=0$ that $\Psi_a$ is a  constant on the boundary, we may assume $\Psi_a=0$ on $\p\cO$ without losing generality. Then we obtain
\begin{align*}
B_l|_{t=0}\cdot\n=\alpha_lB_a\cdot\n+\n^{\perp}\cdot\nabla\alpha_l\Psi_a=0,\qquad\text{ on }\p\cO.
\end{align*}
When the dimension equals 3, we need the following lemma, which will be proved later.
\begin{lemma}\label{Sto}
{\sl
Let $p\in\mathbb{N}$,
for a three dimensional divergence free vector field $u\in H^p(\cO)$ with $u\cdot\n=0$ on $\p\cO$, we can find a vector field $\Phi\in H^{p+1}(\cO)$ with $(\nabla\times\Phi)\cdot\n=0$ and $\Phi_{\tan}=0$ on the boundary, such that $u=\nabla\times\Phi.$
}\end{lemma}

From Lemma \ref{Sto}, there exists a vector field $\Psi_a\in H^{9}(\cO)$ with $(\nabla\times\Psi_a)\cdot\n=0$ and $(\Psi_a)_{\tan}=0$ on the boundary, such that $B_a=\nabla\times\Psi_a.$ So that on the boundary, $\Psi_a$ is a normal vector field and $\Psi_a\times\n=0.$ As a consequence, on the boundary $\p\cO,$ we find
\begin{align*}
\nabla\times(\alpha_l\Psi_a)\cdot\n=\alpha_l(\nabla\times\Psi_a)\cdot\n+\nabla\alpha_l\times\Psi_a\cdot\n
=\alpha_l(\nabla\times\Psi_a)\cdot\n+\Psi_a\times\n\cdot\nabla\alpha_l=0.
\end{align*}
Therefore  the  initial condition \eqref{psia} holds true for both two and three dimension cases and we have shown that $B_l$ satisfies \eqref{cdBl}.

\vspace{0.2cm}
\noindent$\bullet$ \underline{Step 3: construction of $B^2$ and $\mu^2.$}\vspace{0.2cm}

Let
\begin{subequations} \label{S3eqR}
\begin{gather}
\label{defB2}B^2(t,x):=\sum_{1\leq l\leq L}\beta_{\epsilon}(t-t_{x_l})B_l(t,x),\\
\label{defmu2}\mu^2(t,x):=\sum_{1\leq l\leq L}\beta'_{\epsilon}(t-t_{x_l})B_l(t,x).
\end{gather}
\end{subequations}
In view of \eqref{flcom}, formula \eqref{defB2} and \eqref{defmu2} define a solution to \eqref{eqB2} with control function $\mu^2$ supported in $\overline{\cO}\setminus\overline{\Omega}$ satisfying $\dive \mu^2=0$ and $\mu^2\cdot\n=0$ on the boundary, such that
$B^2,\mu^2\in C([0,T];H^8(\cO))$ and $B^2(T,\cdot)=0$.

\end{proof}

\begin{remark}\label{B2n0}
{\sl
We observe from the proof of the above Proposition that the support of $B^2\cdot\n$ is transported by the flow of $u^0$.  This is the reason why we impose the condition \eqref{ini2} for $B_0$:  that $B_a\cdot\n$ vanishes on the entire boundary $\p\cO$ ensures that $B^2\cdot\n=0$ remains on $\p\Omega\setminus\Gamma$ for all time.
}
\end{remark}

\begin{proof}[\bf Proof of Lemma \ref{Sto}] Since $u$ is divergence free, we can find a function $\overline{\Psi}\in H^{p+1}(\cO)$
such that $u=\nabla\times\overline{\Psi} $ and $\|\overline{\Psi}\|_{H^{p+1}(\cO)}\leq C\|u\|_{H^p(\cO)}$, we refer to \cite[Theorem \uppercase\expandafter{\romannumeral 4}.4.13]{Boyer}.
We want to find a scalar function $q$ defined in $\cO$ such that the tangential part of $\overline{\Psi}-\nabla q$ disappears on the boundary.
If such a function $q$ exists, let $\Psi=\overline{\Psi}-\nabla q$. Then $u=\nabla\times\Psi$, and $\Psi$ satisfies $\Psi_{\tan}=0$ and $ (\nabla\times\Psi)\cdot\n=(\nabla\times\overline{\Psi})\cdot\n=u\cdot\n=0$ on $\p\cO$.

Now we construct such a scalar function $q$. First, we define $q$ on the boundary $\p\cO.$ Since $\p\cO$ is simply connected, fix a point $x_0\in \p\cO,$ for any $x\in\p\cO,$ we can take a smooth curve $L_x$ on $\p\cO$ connecting $x_0$ and $x$. For $x\in\p\cO,$ we define $q$ by the second curvilinear integral
\begin{align}\label{defq}
q(x)=\int_{L_x}\overline{\Psi}\cdot dx=\int_{L_x}(\overline{\Psi}_1dx_1+\overline{\Psi}_2dx_2+\overline{\Psi}_3dx_3).
\end{align}
We shall prove that $q$ is actually independent of curve $L_x$. Take another curve $L_x'$ connecting $x_0$ and $x$, let $\Sigma$ be the subset of $\p\cO$ whose boundary is $L_{x}\cup L'_{x}.$ Since $\p\cO$ is simply connected, by Stokes' formula,
\begin{align*}
\int_{L_x}\overline{\Psi}\cdot dx-\int_{L'_x}\overline{\Psi}\cdot dx=\int_{\Sigma}(\n\times\nabla)\cdot\overline{\Psi} d\sigma=0,
\end{align*}
thanks to $(\n\times\nabla)\cdot\overline{\Psi}=(\nabla\times\overline{\Psi})\cdot\n=u\cdot\n=0$ on $\p\cO$. Therefore, the function $q$ is well defined by \eqref{defq} in $\p\cO$. By  the trace theorem, $\overline{\Psi}\in H^{p+\frac{1}{2}}(\p\cO)$ and $q\in H^{p+\frac{3}{2}}(\p\cO)$. Moreover, by definition \eqref{defq}, for tangential derivatives $Z_i=\btau^i\cdot\nabla,1\leq i\leq 3,$
\begin{align}\label{Ziq}
Z_iq(x)=(\btau^i\cdot\overline{\Psi})(x),\quad \text{ for }x\in\p\cO.
\end{align}

For $x\in \mathcal{V}_{\delta_0}$, we write $x=\sigma-s\n(\sigma)$ for $\sigma\in \p\cO,s>0$, and $\sigma,s$ is determined uniquely by $x\in \mathcal{V}_{\delta_0}$. Recall that $\chi\in C^{\infty}(\overline{\cO})$ is a cutoff function supported in $\mathcal{V}_{\delta_0}$ such that $\chi(x)=1$ when $x\in \mathcal{V}_{\delta_0/2}$, see Section \ref{css}. Define
\begin{equation*}
q(x)=\left\{
\begin{aligned}
&q(\sigma)\chi(x),\quad &&\text{ for }x=\sigma-s\n(\sigma)\in \mathcal{V}_{\delta_0}, \sigma\in\p\cO,s>0,\\
&0,\quad &&\text{ for }x\in\cO\setminus\mathcal{V}_{\delta_0}.
\end{aligned}\right.
\end{equation*}
Then $q\in H^{p+2}(\cO), \Psi:=\overline{\Psi}-\nabla q\in H^{p+1}(\cO)$. We deduce from \eqref{Ziq} that $\btau^i\cdot \Psi=0$ on the boundary, for $1\leq i\leq 3$, which implies that $\Psi_{\tan}=0$ on $\p\cO$ and finishes the proof.

\end{proof}

\subsection{Consistency of the full expansion}\label{cstexp}
For a profile $f(t,x,z)$, we define
\begin{equation*}
\{f\}_{\e}(t,x):=f\bigl(t,x,\frac{\varphi(x)}{\sqrt{\e}}\bigr).
\end{equation*}
We define the approximate solution by
\beq\label{dfueapp}
\begin{split}
u^{\e}_{app}:=\ &u^0+\sqrt{\e}\{v^1\}_{\e}+\e \{U^2\}_{\e}+\e^{\frac{3}{2}}\{U^3\}_{\e},\\
p^{\e}_{app}:=\ &p^0+\e(p^2-\p_t\phi^2-u^0\cdot\nabla \phi^2+\{\pi^2\}_{\e})\\
 &+\e^{\frac{3}{2}}(p^3-\p_t\phi^3-u^0\cdot\nabla \phi^3+\{\pi^3\}_{\e})+\e^2\{\pi^4\}_{\e},\\
B^{\e}_{app}:=\ & \e B^2,\quad \sigma^{\e}:= \sigma^0+\e\sigma^2,\quad \Xi^{\e}:=\ \e\mu^2\\
\xi^{\e}:=\ &\nu^0+\sqrt{\e}\{\xi\}_{\e}+\e(\nu^2+\{\xi^2\}_{\e})+\e^{\frac{3}{2}}(\nu^3+\{\xi^3\}_{\e}).
\end{split}\eeq

By the above construction, we find that $(u^{\e}_{app},p^{\e}_{app},B^{\e}_{app})$ satisfies,
\begin{equation}\label{MHDapp}
\left\{
\begin{aligned}
&\p_t u^{\e}_{app}+u^{\e}_{app}\cdot\nabla u^{\e}_{app}-B^{\e}_{app}\cdot\nabla B^{\e}_{app}-\e \Delta u^{\e}_{app}+\nabla p^{\e}_{app}
=\xi^{\e}-\e^{\frac{3}{2}}{f}^{\e},\quad                                      &&\text{ in }\cO,\\
&\p_t B^{\e}_{app}+u^{\e}_{app}\cdot\nabla B^{\e}_{app}-B^{\e}_{app}\cdot\nabla u^{\e}_{app}+\sigma^{\e} B^{\e}_{app}
=\Xi^{\e}-\e^{\frac{3}{2}}F^{\e},\quad                                        &&\text{ in }\cO,\\
&\dive u^{\e}_{app}=\sigma^0+\e\sigma^2-\e^{\frac{3}{2}}d^{\e},\quad \dive B^{\e}_{app}=0,\quad       &&\text{ in }\cO,\\
&u^{\e}_{app}\cdot\n=0,\quad \CN(u^{\e}_{app})=-\e^{\frac{3}{2}}N^{\e},\quad B^{\e}_{app}\cdot\n=0,\quad      &&\text{ on }\p\cO,\\
&(u^{\e}_{app},B^{\e}_{app})|_{t=0}=(\e u_a-\e^{\frac{3}{2}}r_0^{\e},\e B_a),\quad             &&\text{ on }\p\cO.
\end{aligned}\right.
\end{equation}
where
\begin{align*}
f^{\e}=\ &-\e^{\frac{1}{2}}\left\{v^1\cdot\nabla U^3+U^2\cdot\nabla U^2-U^2\cdot\n\p_z U^3+U^3\cdot\nabla v^1-U^3\cdot\n\p_z v^1\right\}_{\e}\\
\nonumber&-\e\left\{U^2\cdot\nabla U^3+U^3\cdot\nabla U^2-U^3\cdot\n\p_z U^3\right\}_{\e}-\e^{\frac{3}{2}}\left\{U^3\cdot\nabla U^3\right\}_{\e},\\
\nonumber&+\e^{\frac{1}{2}}\left\{\Delta U^2-2\p_{\n}\p_zU^3+\Delta\varphi\p_zU^3\right\}_{\e}+\e\{\Delta U^3\}_{\e}
-\e^{\frac{1}{2}}\{\nabla\pi^4\}_{\e}+\e^{\frac{1}{2}}B^2\cdot\nabla B^2,\\
F^{\e}=\ &-\{v^1+\sqrt{\e}U^2+\e U^3\}_{\e}\cdot\nabla B^2+B^2\cdot\nabla\{v^1+\sqrt{\e}U^2+\e U^3\}_{\e},\\
d^{\e}=\ &-\left\{\dive_x(v^3+\n w^3)\right\}_{\e},\\
N^{\e}=\ &-\CN(\nabla\phi^3+v^3|_{z=0}+\n w^3|_{z=0}),\qquad r^{\e}_0(x)= -v^3_0(x,\frac{\varphi(x)}{\sqrt{\e}}).
\end{align*}

Since $B^2$ is supported in $[0,T]$ by construction, then $F^{\e}$ is also supported in $[0,T]$.

\begin{lemma}[Lemma 3 of \cite{iftimie}]\label{IS}
{\sl
There exists a constant $C$ independent of $\e$ such that for all $h=h(x,z)$ in $H^1(\cO;L^2(\mathbb{R}_+))$ which supported in $\mathcal{V}_{\delta_0},$
\begin{align*}
\|h(x,\frac{\varphi(x)}{\sqrt{\e}})\|_{L^2(\cO)}\leq C\e^{\frac{1}{4}}\|h\|_{H^1(\cO;L^2(\mathbb{R}_+))}.
\end{align*}

}
\end{lemma}

\begin{lemma}\label{appes}
{\sl
Let $\t:=\sqrt{1+t^2}$, assume that \eqref{inp} and \eqref{inua} is satisfied, there exists a constant $C>0$ such that the approximate solution $(u^{\e}_{app},B^{\e}_{app})$ and $f^{\e}, F^{\e}, d^{\e}, N^{\e},r^{\e}_0$ satisfy
\begin{subequations} \label{S3eq29}
\begin{gather}
\label{app1}\|u^{\e}_{app}\|_{6}+\|\nabla u^{\e}_{app}\|_{5}+\sqrt{\e}\|\nabla^2 u^{\e}_{app}\|_{4}\leq C\t^{-\gamma},\\
\label{app2}\|u^{\e}_{app}\|_{6,\infty}+\|\nabla u^{\e}_{app}\|_{5,\infty}+\sqrt{\e}\|\nabla^2 u^{\e}_{app}\|_{4,\infty}\leq C\t^{-\gamma},\\
\label{app3}\|u^{\e}_{app}-u^0\|_{5,\infty}\leq \e^{\frac{1}{2}}\t^{-\gamma},\quad \|\nabla(u^{\e}_{app}-u^0)\|_{5,\infty}\leq C\e^{\frac{1}{2}}\t^{-\gamma},\\
\label{appB1}\|B^{\e}_{app}\|_{5}+\|\nabla B^{\e}_{app}\|_{5}\leq C\e\chi_{[0,T]}\\
\label{appB2}\|B^{\e}_{app}\|_{5,\infty}+\|\nabla B^{\e}_{app}\|_{5,\infty}+\|\nabla^2 B^{\e}_{app}\|_{4,\infty}\leq C\e\chi_{[0,T]},\\
\label{app4}\|{f}^{\e}\|_{5}\leq \e^{\frac{1}{2}}\t^{-\gamma},\quad\sqrt{\e}\|\nabla {f}^{\e}\|_{4}\leq C\e^{\frac{3}{4}}\t^{-\gamma},\\
\label{appF1}\|F^{\e}\|_{5}+\sqrt{\e}\|\nabla F^{\e}\|_{4}\leq \e^{\frac{1}{4}}\chi_{[0,T]},\quad
\|\nabla F^{\e}\|_{L^{\infty}(\cO)}\leq C\e^{-\f12}\chi_{[0,T]},\\
\label{app5}\|d^{\e}\|_{5}+\e\|d^{\e}\|_{H^2(\cO)}\leq C\e^{\frac{1}{4}}\t^{-\gamma},\quad
\|d^{\e}\|_{5,\infty}+\|d^{\e}\|_{H^4(\p\cO)}\leq C\t^{-\gamma},\\
\label{app6}\|N^{\e}\|_{H^{7}(\cO)}+\|\p_t N^{\e}\|_{H^{4}(\cO)}+\|N^{\e}\|_{H^4(\p\cO)}\leq C\t^{-\gamma},\\
\label{app7}\|r_0^{\e}\|_5+\sqrt{\e}\|\nabla r^{\e}_0\|_{4}\leq C\e^{\frac{1}{4}}.
\end{gather}
\end{subequations}
}
\end{lemma}
\begin{proof}
By the construction of $u^{\e}_{app}$ and $B^{\e}_{app}$, for $\gamma=\frac{5}{4},\, 1\leq i\leq 3, j=2,3,$ we have
\begin{subequations} \label{S3eqp}
\begin{gather}
\label{ufi23}u^2\in C^7([0,T];H^{17}(\cO)),\quad \phi^2, \phi^3\in C^{7}_{7}(\mathbb{R}_+;H^{17}(\cO)),\\
\label{vw23}v^i,w^j\in C^1_{5/4}(\mathbb{R}_+;H^8(\cO;H^7_2(\mathbb{R}_+))),\\
\label{B23}B^2\in C([0,T];H^8(\cO)).
\end{gather} \end{subequations}
By employing Lemma \ref{IS} along with the Sobolev embedding theorem, one can straightforwardly confirm \eqref{app1} - \eqref{appB2}. We emphasize that both the regular and layer components of $f^{\e}$ are of at least order $\cO(\e^{\frac{1}{2}})$. Utilizing \eqref{ufi23} - \eqref{B23}, \eqref{defpi4}, and Lemma \ref{UV}, and Lemma \ref{IS}, we can check term by term that
\eqref{app4} holds true.

Since the regular part of $F^{\e}$ is at least of order $\cO(\e^{\frac{1}{2}}),$ by using Lemma \ref{UV} and Lemma \ref{IS}, we obtain \eqref{appF1}.

By the definition of $d^{\e}$, \eqref{vw23} and the Sobolev embedding, we deduce \eqref{app5}.

By the definition of $N^{\e}$, we find that $N^{\e}\in C^1_{5/4}(\mathbb{R}_+;H^7(\cO))$, which implies \eqref{app6}.

By the construction of $v^{3}_0\in H^{10}(\cO;C^{\infty}_0(\overline{\mathbb{R}_+}))$,  we deduce  \eqref{app7} from  Lemma \ref{IS}.

We thus complete the proof of Lemma \ref{appes}.
\end{proof}

\begin{remark}{\sl
Here we expand the velocity $u^{\e}$ and the magnetic field $B^{\e}$ to order $\cO(\e^{\frac{3}{2}})$ to create a small coefficient $\e^{\frac{3}{2}}$ in the nonlinear terms of $r^{\e}$ and $R^{\e}$, so that we can bound the remainder estimates; see Remark \ref{od3} for more explanation.

Of course, we can construct higher-order expansions similarly. In general, the initial data of $u^2$ is nonzero, so is $g^3$ defined in \eqref{defg3}. Due to the compatible condition with $\p_z v^3=g^3$, the initial data $v^3_0$ of $v^3$ is nonzero in general. Although we can expand the velocity and the magnetic field to order $\e^{\frac{N}{2}}$ for a lager integer $N$ and construct the approximate solution $(u^{\e,N}_{app},B^{\e,N}_{app})$, the initial data of the remainder term $r^{\e,N}$ defined by
\begin{align*}
u^{\e}=u^{\e,N}_{app}+\e^{\frac{N}{2}}r^{\e,N}
\end{align*}
will be of order $\cO(\e^{\frac{3-N}{2}})$ due to $u^{\e}|_{t=0}=\e u_a$, at least for the $L^{\infty}$ norm. Therefore, it is not profitable to expand $(u^{\e},B^{\e})$ to higher order. Unless we put an initial layer on the initial data of $(u^{\e},B^{\e})$. Or, we make a strong assumption on $(u_a,B_a)$ such that $g^i,i\geq 3$, disappear at time $t=0$. In that case, we could choose $v^i$ to vanish initially. So that the initial data of $r^{\e,N}$ would be zero.
}
\end{remark}

\section{Estimates of the remainder $(r^{\e},R^{\e})$}\label{estrem}

In this section, we shall derive the {\it a priori} estimates, which are necessary to prove Theorem \ref{mainth}.
Let us write
\beq\label{S4eq1}
\begin{split}
u^{\e}(t,x)=\ &u^{\e}_{app}(t,x)+\e^{\frac{3}{2}}r^{\e}(t,x),\\
p^{\e}(t,x)=\ &p^{\e}_{app}(t,x)+\e^{\frac{3}{2}}\pi^{\e}(t,x),\\
B^{\e}(t,x)=\ &B^{\e}_{app}(t,x)+\e^{\frac{3}{2}}R^{\e}(t,x),
\end{split}
\eeq
so that by virtue of \eqref{MHDe} and \eqref{MHDapp},
the remainder term $(r^{\e},R^{\e})$ satisfies
\begin{equation}\label{eqrR}
\begin{cases}
\p_t r^{\e}+u^{\e}\cdot\nabla r^{\e} +r^{\e}\cdot\nabla u^{\e}_{app}-B^{\e}\cdot\nabla R^{\e}-  R^{\e}\cdot\nabla B^{\e}_{app}-\e\Delta r^{\e}+\nabla\pi^{\e}=f^{\e},\\
\p_t R^{\e}+u^{\e}\cdot\nabla R^{\e}+ r^{\e}\cdot\nabla B^{\e}_{app}-B^{\e}\cdot\nabla r^{\e}-R^{\e}\cdot\nabla u^{\e}_{app}+\sigma^{\e} R^{\e}=F^{\e},\\
\begin{aligned}
&\dive r^{\e}=d^{\e},\quad \dive R^{\e}=0,\quad &&\text{ in }\R_+\times\cO,\\
&r^{\e}\cdot\n=0,\quad \CN(r^{\e})=N^{\e},\quad R^{\e}\cdot\n=0\qquad &&\text{ on }\R_+\times\p\cO,\\
&r^{\e}|_{t=0}=r^{\e}_{0},\quad R^{\e}|_{t=0}=0,\quad &&\text{ in }\cO.
\end{aligned}
\end{cases}
\end{equation}

\subsection{Preliminary on tangential derivatives}
We recall some properties of the tangential derivatives and we refer to \cite{LSZ1} for more details. The tangential derivatives $Z_i$ satisfies
\begin{subequations} \label{S4eq1}
\begin{gather}
\label{Z1}\nabla Z_j=Z_j\nabla+ \nabla \btau^j\cdot \nabla,\\
\label{Z2}\Delta Z_j=Z_j\Delta +2\nabla \btau^j:\nabla^2+\Delta \btau^j \cdot \nabla .
\end{gather}
\end{subequations}
In general, for $|\alpha|=m\in \mathbb{N}_+,$ we get, by  using Leibniz's formula, that
\begin{equation*}
\label{Z3}[\Delta, Z^{\alpha}]=\sum_{|\beta|,|\gamma|\leq m-1}(c_{\beta}\nabla^2Z^{\beta}+c_{\gamma}\nabla Z^{\gamma}),
\end{equation*}
for some smooth functions $c_{\beta}$ and $c_{\gamma}$ depended only on the vector field $\mathfrak{W}.$ Moreover, the commutator of the tangential derivatives satisfy the following property, for $1\leq i,j\leq 5,$
\begin{equation}\label{cm}
\text{ the commutators }[\partial_\mathbf{n},Z_{i}],[Z_0,Z_i],[Z_i,Z_j]\text{ are tangential derivatives}.
\end{equation}

\begin{proposition}\label{optan}
{\sl
By using tangential derivatives and normal derivative, we can rewrite the gradient operator $\nabla$ and the Laplacian operator $\Delta$ in $\mathcal{V}_{\delta_0}$ as
\begin{align}
\label{nabf}\nabla &= \btau^1Z_1 +\btau^2Z_2+\btau^3Z_3+\n\p_{\n},\\
\label{Delf}\Delta &=Z_1^2+Z_2^2+Z_3^3+\p_{\n}^2-\Delta\varphi\p_{\n}.
\end{align}
As a consequence, for a vector field $f$ supported in $\mathcal{V}_{\delta_0}$,
\begin{align}
\label{divef}\dive f=Z_1f\cdot \btau^1+Z_2f\cdot \btau^2+Z_3f\cdot \btau^3+\p_{\n} f\cdot \n.
\end{align}
}
\end{proposition}
\begin{proof}
By the definition of $Z_i, \p_{\n}$ and direct calculation, we find that
\begin{align*}
\btau^1Z_1+\btau^2Z_2+\btau^3Z_3+\n\p_{\n}=|\n|^2\nabla,
\end{align*}
which implies \eqref{nabf}and \eqref{divef} hold true since $|\n|=1$ in $\mathcal{V}_{\delta_0}$.

Let $\n=(n_1,n_2,n_3),$ we can find through calculation that,
\begin{align*}
Z_1^2=&n_3^2\p_2^2+n_2^2\p_3^2-2n_2n_3\p_2\p_3+(Z_1\btau^1)\cdot\nabla,\\
Z_2^2=&n_1^2\p_3^2+n_3^2\p_1^2-2n_1n_3\p_1\p_3+(Z_2\btau^2)\cdot\nabla,\\
Z_3^2=&n_1^2\p_2^2+n_2^2\p_1^2-2n_1n_2\p_1\p_2+(Z_3\btau^3)\cdot\nabla,\\
\p_{\n}^2=&n_1^2\p_1^2+n_2^2\p_2^2+n_3^2\p_3^2+(\p_{\n}\n)\cdot\nabla\\
&+2n_1n_2\p_1\p_2+2n_2n_3\p_2\p_3+2n_1n_3\p_1\p_3.\\
\end{align*}
Hence
\begin{align*}
Z_1^2+Z^2_2+Z_3^3+\p_{\n}^2&=|\n|^2\Delta+(Z_1\btau^1+Z_2\btau^2+Z_3\btau^3+\p_{\n} \n)\cdot\nabla\\
&=|\n|^2\Delta+\left(\nabla|\n|^2-(\dive\n)\n\right)\cdot\nabla,
\end{align*}
where we used $\n=-\nabla\varphi$ is symmetric, and  \eqref{Delf} follows since $|\n|=1$ in $\mathcal{V}_{\delta_0}$ and $\dive\n=-\Delta\varphi.$
\end{proof}

\subsection{$L^2$ estimates}
\begin{proposition}\label{L2et}
{\sl
Let $\gamma=\frac{5}{4},$ there exist constants $c_0, C_0>0$, such that for all $t\in [0,\frac{T}{\e}]$,
\begin{align}
\label{L2rRe}\|(r^{\e},R^{\e})(t)\|_{L^2(\cO)}^2+c_0\e\int_0^t\|r^{\e}(s)\|_{H^1(\cO)}^2ds\leq C_0\e^{\frac{1}{2}}.
\end{align}
}\end{proposition}
\begin{proof}Firstly, let us deal with the $L^2$ estimates of $R^{\e}$. We get, by taking $L^2$ inner product of
the $R^\e$ equation of \eqref{eqrR} with $R^\e$, that,
\begin{equation}\label{S4eq2}
\begin{split}
\frac{1}{2}\frac{d}{dt}\|R^{\e}\|^2_{L^2(\cO)}&+\int_{\cO}u^{\e}\cdot\nabla R^{\e}\cdot R^{\e}+\int_{\cO} r^{\e}\cdot\nabla B^{\e}_{app}\cdot R^{\e}\\
&-\int_{\cO}B^{\e}\cdot\nabla r^{\e}\cdot R^{\e}-\int_{\cO} R^{\e}\cdot\nabla u^{\e}_{app}\cdot R^{\e}+\int_{\cO}\sigma^{\e}|R^{\e}|^2=\int_{\cO}F^{\e}\cdot R^{\e}.
\end{split}
\end{equation}
Now we estimate each term in the above equation. Since $u^{\e}\cdot\n=0$ on $\p\cO$ and $\dive u^{\e}=\sigma^0+\e\sigma^2$ is supported in $[0,T],$ by using integration by parts, we find that
\begin{equation*}
\left|\int_{\cO}u^{\e}\cdot\nabla R^{\e}\cdot R^{\e}+\int_{\cO}\sigma^{\e}|R^{\e}|^2\right| =\frac{1}{2}\left|\int_{\cO}(\sigma^0+\e\sigma^2)|R^{\e}|^2\right|\leq \chi_{[0,T]}\|R^{\e}\|_{L^2(\cO)}^2.
\end{equation*}
We remark that if $u_a$ is not tangent to the whole boundary $\p\cO$, we have no reason to require $u^{\e}\cdot\n=0$ on $\p\cO$ for all time. As a result, there will be a boundary term $\|R^{\e}\|^2_{L^2(\p\cO)}$ appears on the right-hand-side of the above inequality, which would be a bad term since we cannot gain regularity through the magnetic equation.

By Lemma \ref{appes}, there is a constant $C>0$ such that
\begin{align*}
\left| \int_{\cO}F^{\e}\cdot R^{\e} \right|  \leq \|F^{\e}\|_{L^2(\cO)}\|R^{\e}\|_{L^2(\cO)}
\leq C\e^{\frac{1}{4}}\chi_{[0,T]}\|R^{\e}\|_{L^2(\cO)}\leq C\chi_{[0,T]}(\e^{\frac{1}{2}}+\|R^{\e}||_{L^2(\cO)}^2).
\end{align*}
There exists a constant $C>0$ such that $\chi_{[0,T]}(t)\leq C\t^{-\gamma}$ for all $t\geq 0$. So that we deduce from Lemma \ref{appes} that
\begin{align*}
&\left|\int_{\cO} r^{\e}\cdot\nabla B^{\e}_{app}\cdot R^{\e}\right|+\left|\int_{\cO}R^{\e}\cdot\nabla u^{\e}_{app}\cdot R^{\e}\right|\\
&\leq \left(\|\nabla B^{\e}_{app}\|_{L^{\infty}(\cO)}+\|\nabla u^{\e}_{app}\|_{L^{\infty}(\cO)}\right)\|(r^{\e},R^{\e})\|_{L^2(\cO)}^2
\leq C\t^{-\gamma}\|(r^{\e},R^{\e})\|_{L^2(\cO)}^2.
\end{align*}

By substituting the above estimates into \eqref{S4eq2}, we obtain
\begin{align}\label{L2Re}
\frac{1}{2}\frac{d}{dt}\|R^{\e}\|_{L^2(\cO)}^2-\int_{\cO}B^{\e}\cdot\nabla r^{\e}\cdot R^{\e}\leq C\e^{\frac{1}{2}}\t^{-\gamma}+C\t^{-\gamma}\|(r^{\e},R^{\e})\|_{L^2(\cO)}^2.
\end{align}
The second term will be handled later by combining the energy estimate of $r^{\e}$.

\bigskip

Next we turn to the $L^2$   estimate of $r^{\e}.$ Let $\mathbb{P}$ to be the Leray projection operator to the divergence free vector field, we decompose $r^{\e}$ into $r^{\e}=\mathbb{P} r^{\e}+\nabla \phi^{\e}$. Hence $\phi^{\e}$ satisfies
\begin{equation}\label{dfphi}
\begin{cases}
\Delta\phi^{\e}=d^{\e},\qquad\text{ in }\cO,\\
\p_{\n}\phi^{\e}=0,\qquad\text{ on }\p\cO.
\end{cases}
\end{equation}
By elliptic regularity estimates and the property of $d^{\e}$, one has
\begin{align}\label{phie}
\|r^{\e}-\mathbb{P}r^{\e}\|_{H^1(\cO)}\leq \|\phi^{\e}\|_{H^2(\cO)}\leq C\|d^{\e}\|_{L^2(\cO)}\leq C\e^{\frac{1}{4}}\t^{-\gamma}.
\end{align}
Then we get, by taking the $L^2$ inner product of \eqref{eqrR}$_{1}$ with $\mathbb{P}r^{\e}$, that
\beq\label{L2egr}
\begin{split}
\int_{\cO}\p_t r^{\e}\cdot\mathbb{P}r^{\e}&+\int_{\cO}u^{\e}\cdot\nabla r^{\e}\cdot\mathbb{P}r^{\e}+\int_{\cO}r^{\e}\cdot\nabla u^{\e}_{app}\cdot \mathbb{P}r^{\e}-\int_{\cO}B^{\e}\cdot\nabla R^{\e}\cdot\mathbb{P}r^{\e}\\
&-\int_{\cO} R^{\e}\cdot\nabla B^{\e}_{app}\cdot\mathbb{P}r^{\e}-\e\int_{\cO}\Delta r^{\e}\cdot\mathbb{P}r^{\e}+\int_{\cO}\nabla\pi^{\e}\cdot\mathbb{P}r^{\e}=\int_{\cO}f^{\e}\cdot\mathbb{P}r^{\e}.
\end{split} \eeq
\noindent$\bullet$ \underline{$\int_{\cO}\p_t r^{\e}\cdot\mathbb{P}r^{\e}$ and $\int_{\cO}\nabla\pi^{\e}\cdot\mathbb{P}r^{\e}$.}\vspace{0.2cm}

By the definition of $\mathbb{P}, \phi^{\e}$ and by using an integration by parts, we obtain that
\begin{gather*}
\label{L21}\int_{\cO}\p_t r^{\e}\cdot\mathbb{P}r^{\e}=\frac{1}{2}\frac{d}{dt}\|\mathbb{P}r^{\e}\|^2_{L^2(\cO)}+\int_{\cO}\p_t\nabla\phi^{\e}\cdot\mathbb{P} r^{\e}=\frac{1}{2}\frac{d}{dt}\|\mathbb{P}r^{\e}\|^2_{L^2(\cO)},\\
\label{L22}\int_{\cO}\nabla\pi^{\e}\cdot\mathbb{P}r^{\e}=0.
\end{gather*}

\noindent$\bullet$ \underline{$\int_{\cO}f^{\e}\cdot\mathbb{P}r^{\e}$.}\vspace{0.2cm}

By Lemma \ref{appes} and the properties of the Leray projection operator $\mathbb{P}:$  $\|\mathbb{P}r^{\e}\|_{L^2}\leq \|r^{\e}\|_{L^2(\cO)}$, there is a constant $C>0$ such that
\begin{align*}
\label{L23}
 \left| \int_{\cO}f^{\e}\cdot \mathbb{P}r^{\e}\right|
\leq \|f^{\e}\|_{L^2(\cO)}\|\mathbb{P}r^{\e}\|_{L^2(\cO)}
\leq C\e^{\frac{1}{2}}\t^{-\gamma}\|r^{\e}\|_{L^2(\cO)}\leq C\e\t^{-\gamma}+C\t^{-\gamma}\|r^{\e}||_{L^2(\cO)}^2.
\end{align*}

\noindent$\bullet$ \underline{$\int_{\cO}u^{\e}\cdot\nabla r^{\e}\cdot\mathbb{P}r^{\e}$.}\vspace{0.2cm}

By using the Helmholtz-Leray decomposition, we obtain that
\begin{align*}
 \int_{\cO}u^{\e}\cdot\nabla r^{\e}\cdot\mathbb{P}r^{\e}
=\int_{\cO}u^{\e}\cdot\nabla \mathbb{P}r^{\e}\cdot\mathbb{P}r^{\e}+\int_{\cO}u^{\e}\cdot\nabla (r^{\e}-\mathbb{P}r^{\e})\cdot\mathbb{P}r^{\e}.
\end{align*}
Since $\dive u^{\e}=\sigma^0+\e\sigma^2$ is supported in $[0,T], u^{\e}\cdot\n=0$ on $\p\cO,$ we get, by using integration bay parts, that
\begin{equation*}
\left|\int_{\cO}u^{\e}\cdot\nabla \mathbb{P}r^{\e}\cdot\mathbb{P}r^{\e}\right|=\frac{1}{2}\left|\int_{\cO}(\sigma^0+\e\sigma^2)|\mathbb{P}r^{\e}|^2\right|\leq C\chi_{[0,T]}\|r^{\e}\|_{L^2(\cO)}^2.
\end{equation*}
While it follows from \eqref{phie}, the decomposition: $u^{\e}=u^{\e}_{app}+\e^{\frac{3}{2}}r^{\e},$ and the properties of $\mathbb{P}$, that
\begin{align*}
\left|\int_{\cO}u^{\e}\cdot\nabla (r^{\e}-\mathbb{P}r^{\e})\cdot\mathbb{P}r^{\e}\right|
&\leq \|\nabla (r^{\e}-\mathbb{P}r^{\e})\|_{L^2(\cO)}(\|u^{\e}_{app}\|_{L^{\infty}(\cO)}\|r^{\e}\|_{L^2(\cO)}+\e^{\frac{3}{2}}\|r^{\e}\|^2_{L^4(\cO)})\\
\nonumber&\leq C\e^{\frac{1}{4}}\t^{-\gamma}(\t^{-\gamma}\|r^{\e}\|_{L^{2}(\cO)}+\e^{\frac{3}{2}}\|r^{\e}\|^2_{L^{4}(\cO)})
\end{align*}
By using H\"older's inequality and Sobolev embedding,
\begin{equation*}
\|r^{\e}\|_{L^4(\cO)}\leq C\|r^{\e}\|^{\frac{1}{4}}_{L^2(\cO)}\|r^{\e}\|^{\frac{3}{4}}_{L^6(\cO)}\leq C\|r^{\e}\|^{\frac{1}{4}}_{L^2(\cO)}\|\nabla r^{\e}\|_{L^2(\cO)}^{\frac{3}{4}}.
\end{equation*}
Hence
\begin{align*}
\left|\int_{\cO}u^{\e}\cdot\nabla (r^{\e}-\mathbb{P}r^{\e})\cdot\mathbb{P}r^{\e}\right|
\leq \frac{1}{4}c_0\e\|r^{\e}\|^2_{H^1(\cO)}+C\e^{\frac{1}{2}}\t^{-\gamma}+C\|r^{\e}\|^2_{L^2(\cO)}(\e+\t^{-\gamma}).
\end{align*}
Therefore
\begin{equation*}\label{l24}
\left|\int_{\cO}u^{\e}\cdot\nabla r^{\e}\cdot\mathbb{P}r^{\e}\right|
\leq \frac{1}{4}c_0\e\|r^{\e}\|^2_{H^1(\cO)}+C\e^{\frac{1}{2}}\t^{-\gamma}+C\|r^{\e}\|^2_{L^2(\cO)}(\e+\t^{-\gamma}).
\end{equation*}

\noindent$\bullet$ \underline{$\int_{\cO}r^{\e}\cdot\nabla u^{\e}_{app}\cdot \mathbb{P}r^{\e}$ and $\int_{\cO}R^{\e}\cdot\nabla B^{\e}_{app}\cdot\mathbb{P}r^{\e}$.}\vspace{0.2cm}

Thanks to Lemma \ref{appes}, we find that
\begin{align*}
\left|\int_{\cO}r^{\e}\cdot\nabla u^{\e}_{app}\cdot \mathbb{P}r^{\e}\right|\leq & C\t^{-\gamma}\|r^{\e}\|_{L^2(\cO)}^2,\\
\left|\int_{\cO}R^{\e}\cdot\nabla B^{\e}_{app}\cdot \mathbb{P}r^{\e}\right|
\leq & C\e\chi_{[0,T]}(\|r\|_{L^2(\cO)}^2+\|R^{\e}\|_{L^2(\cO)}^2).
\end{align*}

\noindent$\bullet$ \underline{$\int_{\cO}B^{\e}\cdot\nabla R^{\e}\cdot\mathbb{P}r^{\e}$.}\

We handle this term by combining it with the second term of inequality \eqref{L2Re}. Indeed by using an integration by parts and  $\dive B^{\e}=0, B^{\e}\cdot\n=0,$ we find that
\begin{align*}
\int_{\cO}B^{\e}\cdot\nabla R^{\e}\cdot\mathbb{P}r^{\e}+\int_{\cO}B^{\e}\cdot\nabla r^{\e}\cdot R^{\e}=\int_{\cO}B^{\e}\cdot\nabla (r^{\e}-\mathbb{P}r^{\e})\cdot R^{\e}.
\end{align*}
Thanks to $r^{\e}=\mathbb{P}r^{\e}+\nabla\phi^{\e}, B^{\e}=B^{\e}_{app}+\e^{\frac{3}{2}}R^{\e},$ one has
\begin{align*}
&\left|\int_{\cO}B^{\e}\cdot\nabla(r^{\e}-\mathbb{P}r^{\e})\cdot R^{\e}\right|\\
&\leq \|B^{\e}_{app}\|_{L^{\infty}(\cO)}\|\nabla^2\phi^{\e}\|_{L^2(\cO)}\|R^{\e}\|_{L^2(\cO)}
+\e^{\frac{3}{2}}\|R^{\e}\|^2_{L^2(\cO)}\|\nabla^2\phi^{\e}\|_{L^{\infty}(\cO)}.
\end{align*}
By Lemma \ref{appes}, $\|B^{\e}_{app}\|_{L^{\infty}(\cO)}\leq C\e\chi_{[0,T]}.$ By Sobolev embedding, the regularity estimates for elliptic equation \ref{dfphi} and \eqref{app5},
\begin{align*}
\|\nabla^2\phi^{\e}\|_{L^2(\cO)}&\leq C\|d^{\e}\|_{L^2(\cO)}\leq \e^{\frac{1}{4}}\t^{-\gamma},\\
\|\nabla^2\phi^{\e}\|_{L^{\infty}(\cO)}&\leq C\|\nabla^2\phi^{\e}\|_{H^2(\cO)}\leq C\|d^{\e}\|_{H^2(\cO)}\leq C\e^{-\frac{3}{4}}\t^{-\gamma},
\end{align*}
where we used the definition of $d^{\e}$ in Section \ref{cstexp} and the properties of the boundary layers.
As a result, it comes out
\begin{align*}\label{l27}
\left|\int_{\cO}B^{\e}\cdot\nabla (r^{\e}-\mathbb{P}r^{\e})\cdot R^{\e}\right|
&\leq C\e^{\frac{5}{4}}\chi_{[0,T]}\|R^{\e}\|_{L^2(\cO)}+C\e^{\frac{3}{4}}\t^{-\gamma}\|R^{\e}\|^2_{L^2(\cO)}\\
\nonumber&\leq C\e^{\frac{7}{4}}\chi_{[0,T]}+C\e^{\frac{3}{4}}\t^{-\gamma}\|R^{\e}\|^2_{L^2(\cO)}.
\end{align*}

\noindent$\bullet$ \underline{$\int_{\cO}\e\Delta r^{\e}\cdot\mathbb{P}r^{\e}$.}\vspace{0.2cm}

Due to $\dive r^{\e}=d^{\e},$ we get, by using an integration by parts, that
\beq\label{Drr}
\begin{split}
&-\int_{\cO}\e\Delta r^{\e}\cdot\mathbb{P}r^{\e}=-\int_{\cO}\e\p_{x_i}(\p_{x_i}r^{\e}_j+\p_{x_j}r^{\e}_i)\cdot\mathbb{P}r^{\e}_j
+\int_{\cO}\e\nabla d^{\e}\cdot\mathbb{P}r^{\e}\\
&=2\e\int_{\cO} D(r^{\e}):D(\mathbb{P}r^{\e})-2\e\int_{\p\cO}(D(r^{\e})\cdot \n)_{\tan}\cdot\mathbb{P}r^{\e}+\e\int_{\cO}\nabla d^{\e}\cdot\mathbb{P}r^{\e},
\end{split}\eeq
where $D(f)$ denotes the symmetric part of $\nabla f$, $A:B=\sum_{1\leq i,j\leq 3} A_{ij}B_{ij}$ denotes the trace of matrix $AB$ for two matrix $A,B\in M_{3\times3}.$
We deduce from  \eqref{phie} that
\begin{align*}
\int_{\cO}D(r^{\e}):D(\mathbb{P}r^{\e})&=\|D(r^{\e})\|_{L^2(\cO)}^2-\int_{\cO}D(r^{\e}):D(r^{\e}-\mathbb{P}r^{\e})\\
&\geq \|D(r^{\e})\|_{L^2(\cO)}^2-C\e^{\frac{1}{4}}\t^{-\gamma}\|D(r^{\e})\|_{L^2(\cO)}\\
&\geq\frac{1}{2}\|D(r^{\e})\|_{L^2(\cO)}^2-C\e^{\frac{1}{2}}\t^{-2\gamma}.
\end{align*}
Yet by Korn's inequality, there exist constants $c_0,c_1>0$ such that
\begin{equation*}
\|D(r^{\e})\|_{L^2(\cO)}^2\geq c_{0}\|r^{\e}\|_{H^1(\cO)}^2-c_1\|r^{\e}\|_{L^2(\cO)}^2,
\end{equation*}
which results in
\begin{equation*}
2\e\int_{\cO} D(r^{\e}):D(\mathbb{P}r^{\e})\geq c_0\e\|r^{\e}\|_{H^1(\cO)}^2-c_1\e\|r^{\e}\|_{L^2(\cO)}^2-C\e^{\frac{3}{2}}\t^{-2\gamma}.
\end{equation*}
Whereas it follows from the definition of $\mathcal{N}$, $\CN(r^{\e})=N^{\e}$ and Lemma \ref{appes} that
\begin{align*}
\e\left|\int_{\p\cO}(D(r^{\e})\cdot \n)_{\tan}\cdot\mathbb{P}r^{\e}\right|&=2\e\left|\int_{\p\cO}(N^{\e}-(Mr^{\e})_{\tan})\cdot\mathbb{P}r^{\e}\right|\\
\nonumber&=2\e\left|\int_{\cO}\dive\bigl(\n(N^{\e}-(Mr^{\e})_{\tan})\cdot\mathbb{P}r^{\e}\bigr)\right|\\
\nonumber&\leq C\e(\|r^{\e}\|_{H^1(\cO)}+\t^{-\gamma})\|r^{\e}\|_{L^2(\cO)}\\
\nonumber&\leq \frac{1}{4}c_0\e\|r^{\e}\|_{H^1(\cO)}^2+C\e(\|r^{\e}\|_{L^2(\cO)}^2+\t^{-2\gamma}).
\end{align*}
By the properties of Leray projection operator $\mathbb{P}$ and an integration by parts, one has
\begin{align*}
\int_{\cO}\nabla d^{\e}\cdot\mathbb{P}r^{\e}=0.
\end{align*}
By inserting the above inequalities int \eqref{Drr}, we obtain
\begin{align}\label{l28}
-\int_{\cO}\e\Delta r^{\e}\cdot\mathbb{P}r^{\e}\geq \frac{3}{4}c_0\e\|r^{\e}\|_{H^1(\cO)}^2-C\e\|r^{\e}\|^2_{L^2(\cO)}-C\e\t^{-2\gamma}.
\end{align}

By substituting the above estimates into
 \eqref{L2egr} and using the fact that $\chi_{[0,T]}\leq C\t^{-\gamma},\forall \, t\geq 0$ for some constant $C>0,$ we achieve
 \beq\label{L2re}
\begin{split}
&\frac{1}{2}\frac{d}{dt}\|\mathbb{P}r^{\e}(t)\|^2_{L^2(\cO)}+\frac{1}{2}c_0\e\|r^{\e}\|_{H^1(\cO)}^2+\int_{\cO}B^{\e}\cdot\nabla r^{\e}\cdot R^{\e}\\
 &\leq C\e^{\frac{1}{2}}\t^{-\gamma}+C\|(r^{\e},R^{\e})\|_{L^2(\cO)}^2\big(\e+\t^{-\gamma}\big).
\end{split}\eeq

By summing up  \eqref{L2Re} and \eqref{L2re} and integrating the resulting inequality over $0,t)$ for $t\leq \frac{T}{\e}$,  we find that
\beq\label{L2rR}
\begin{split}
&\|(\mathbb{P}r^{\e},R^{\e})(t)\|^2_{L^2(\cO)}+c_0\e\int_{0}^{t}\|r^{\e}(s)\|^2_{H^1(\cO)}\ ds\\
&\leq C\e^{\frac{1}{2}}+C\int_0^{t}\|(r^{\e},R^{\e})(s)\|_{L^2(\cO)}^2\big(\e+\s^{-\gamma}\big)\ ds.
\end{split}
\eeq
 Observing from \eqref{phie} that $\|r^{\e}-\mathbb{P}r^{\e}\|^2_{L^2(\cO)}=\|\phi^{\e}\|^2_{L^2(\cO)}\leq C\e^{\frac{1}{2}}\t^{-\gamma},$ which together with  \eqref{L2rR} and Gronwall's inequality ensure  \eqref{L2rRe}. We thus finish the proof of Proposition \ref{L2et}.
\end{proof}

\subsection{Tangential derivatives estimates}

We introduce two lemmas which will be frequently used later on. The first lemma is the generalized Sobolev-Gagliardo-Nirenberg-Moser inequality; we refer to \cite{gues} for its proof.

\begin{lemma}\label{SGNM}
{\sl
Let $m\in\mathbb{N}$, there exists a constant $C>0$, such that for all multi-index $\alpha_1,\alpha_2\in \mathbb{N}^6$ with $|\alpha_1|+|\alpha_2|=m$ and for all $u,v\in L^{\infty}(\cO)\cap H^m_{co}(\cO)$, we have
\begin{align}
\|Z^{\alpha_1}uZ^{\alpha_2}v\|_{L^{2}(\cO)}\leq C\left(\|u\|_{L^{\infty}(\cO)}\|v\|_m+\|v\|_{L^{\infty}(\cO)}\|u\|_m\right).
\end{align}
}
\end{lemma}

Thanks to $\dive r^{\e}=d^{\e}$ and $\mathcal{N}(r^{\e})=N^{\e}$ on $\p\cO$ with $\|(d^{\e},N^{\e})\|_{H^4(\p\cO)}\leq \t^{-\gamma}$ (by Lemma \ref{appes}), we can control one normal derivative on the boundary by tangential derivatives, that is, we have the following lemma. We refer to \cite[Lemma 5.2]{LSZ1} for the proof.

\begin{lemma}\label{norpa}
{\sl
Let $1\leq m\leq 5$, there exists a constant $C>0$ such that
\begin{align*}
\|\nabla r^{\e}\|_{H^{m-1}(\p\cO)}\leq C\bigl(\|r^{\e}\|_{H^m(\p\cO)}+\t^{-\gamma}\bigr).
\end{align*}
}
\end{lemma}

\begin{lemma}\label{unfi}
{\sl
Let $m\in\mathbb{N_+},$ there exists a constant $C>0$ such that for all $u$ satisfies $\dive u=f$
in $\cO$ and $u\cdot\n=0$ on the boundary $\p\cO,$ we have
\begin{subequations} \label{S4eq3}
\begin{align}
\|\frac{u\cdot\n}{\varphi}\|_{m-1,\infty}&\leq C\left(\|u\|_{m,\infty}+\|f\|_{m-1,\infty}\right), \label{unphi1}\\
\|\frac{u\cdot\n}{\varphi}\|_{m-1}&\leq C\left(\|u\|_{m}+\|f\|_{m-1}\right), \label{unphi2}\\
\|u\cdot\nabla v\|_{m-1}&\leq C\left(\|u\|_{1,\infty}\|v\|_{m}+\|u\|_{m}\|v\|_{1,\infty}+\|f\|_{m-1,\infty}\|v\|_m\right).\label{unabv}
\end{align}
\end{subequations}
}
\end{lemma}

\begin{proof}
First, let us assume $u$ and $f$ are supported in $\mathcal{V}_{\delta_0}$. For $x\in\mathcal{V}_{\delta},$ we can write $x=\sigma-s\n(\sigma)$ with $\sigma\in\p\cO,s>0$, and $\sigma,s$ is determined uniquely by $x\in\mathcal{V}_{\delta_0}$. Let $\widetilde{u\cdot\n}(\sigma,s)=(u\cdot\n)(x)$, we observe that
\begin{align*}
\frac{u\cdot\n}{\varphi}(x)=\frac{\widetilde{u\cdot\n}(\sigma,s)}{s}=\int_0^1\p_s(\widetilde{u\cdot\n})(\sigma,\xi s)d\xi
=-\int_0^1\widetilde{\p_{\n}(u\cdot\n)}(\sigma,\xi s)d\xi.
\end{align*}
Since $\nabla \n$ is symmetric, $\p_{\n}\n=\frac{1}{2}\nabla|\n|^2=0$ in $\mathcal{V}_{\delta_0},$ we have
\begin{align*}
\frac{u\cdot\n}{\varphi}(x)=-\int_0^1(\p_{\n}u\cdot\n)(\sigma-\xi s\n (\sigma))d\xi,
\end{align*}
from which, $\dive u=f $ and Proposition \ref{optan}, we infer
\begin{align}\label{unphi}
\frac{u\cdot\n}{\varphi}(x)=\int_0^1\left(Z_1u\cdot \btau^1+Z_2u\cdot \btau^2+Z_3u\cdot \btau^3-f\right)(\sigma-\xi s\n (\sigma))d\xi.
\end{align}
Consequently, there exists a constant $C>0$ such that
\begin{align*}
\|\frac{u\cdot\n}{\varphi}\|_{L^{\infty}(\cO)}&\leq C\left(\|f\|_{L^{\infty}(\cO)}+\|u\|_{1,\infty}\right),\\
\|\frac{u\cdot\n}{\varphi}\|_{L^{2}(\cO)}&\leq C\left(\|f\|_{L^{2}(\cO)}+\|u\|_{1}\right).
\end{align*}
By applying tangential derivatives to \eqref{unphi} and using Lemma \ref{SGNM}, we deduce \eqref{unphi1} and \eqref{unphi2} for general $m.$

While by virtue of  Proposition \ref{optan}, we write
\begin{align}\label{S4eq6}
u\cdot\nabla v=u\cdot \btau^1Z_1v+u\cdot \btau^2Z_2v+u\cdot \btau^3Z_3v+\frac{u\cdot\n}{\varphi}Z_0v,
\end{align}
which together with \eqref{unphi} and Lemma \ref{SGNM} ensures \eqref{unabv}.

For general $u$ and $f$, we write $u=\chi u+(1-\chi)u$, where $\chi$ is a cut-off function supported in $\mathcal{V}_{\delta_0}$ and $\chi=1$ in $\mathcal{V}_{\delta_0/2}$.  Hence $\dive(\chi u)=\chi f+u\cdot\nabla\chi$ is also supported in $\mathcal{V}_{\delta_0}$. When $1-\chi\neq 0$, one has $\varphi\geq \frac{\delta_0}{2}$. By the regularities of $\chi$ and $\varphi,$  we can easily derive the estimates \eqref{unphi1}, \eqref{unphi2} and \eqref{unabv} for general cases.
\end{proof}

The equations \eqref{unphi} and \eqref{S4eq6} show that $u\cdot\nabla $ is actually a linear combination of some tangential derivatives with coefficients composed of $\dive u$ or $Z^{\alpha}u$ for $|\alpha|\leq 1$.

\begin{proposition}\label{tanes}
{\sl
Let $\gamma=\frac{5}{4}$, $m\in \mathbb{N}$ with $1\leq m\leq 5$, The remainders $r^{\e}$ and $R^{\e}$ satisfy, for $t\in \bigl[0,\frac{T}{\e}\bigr]$ and for some constants $c_0>0, C>0$
\beq\label{ZarR}
\begin{split}
\|(r^{\e},R^{\e})(t)\|_m^2+c_0\e\int_0^{t}\|\nabla r^{\e}(s)\|_{m}^2ds
\leq C\e^{\frac{1}{2}}+C\sum_{|\alpha|\leq m}\left|\int_0^t\int_{\cO}Z^{\alpha}\nabla\pi^{\e}\cdot Z^{\alpha}r^{\e}dxds\right|\\
 +C\int_0^t\|(r^{\e},R^{\e})\|^2_{m}\left(\e+\s^{-\gamma}+\e^2\|(r^{\e},R^{\e})\|^2_{2,\infty}+\e^{\frac{3}{2}}\|\nabla r^{\e}\|_{L^{\infty}(\cO)}\right)\,ds.
\end{split}\eeq
}\end{proposition}
\begin{remark}\label{od3}
{\sl
We have to expand the velocity $u^{\e}$ at least to the order of $\cO(\e^{\frac{3}{2}})$, hence there is a coefficient $\e^{\frac{3}{2}}$ appears in the nonlinear terms of $r^{\e}$, which helps us to derive the control  of $\|\nabla r^{\e}\|_{L^{\infty}(\cO)}$, see \eqref{K4} of Section \ref{infes}.

}
\end{remark}

\begin{proof}[\bf Proof of Proposition \ref{tanes}] 
We divide the proof of this proposition into the following two steps:

\noindent{\bf Step 1.} The tangential derivative estimates of $r^{\e}.$

Let $|\alpha|\leq m\in\mathbb{N}_+,$ we get, by applying $Z^{\alpha}$ to \eqref{eqrR}$_1$ and then taking $L^2$ inner product of the resulting
equation with $Z^{\alpha}r^{\e},$ that
\beq\label{S4eq7}
\begin{split}
&\frac{1}{2}\frac{d}{dt}\|Z^{\alpha}r^{\e}(t)\|_{L^2(\cO)}^2
+\int_{\cO}Z^{\alpha}(u^{\e}\cdot\nabla r^{\e}) \cdot Z^{\alpha}r^{\e}
+\int_{\cO}Z^{\alpha}(r^{\e}\cdot\nabla u^{\e}_{app})\cdot Z^{\alpha}r^{\e}\\
&-\int_{\cO}Z^{\alpha}(B^{\e}\cdot\nabla R^{\e})\cdot Z^{\alpha}r^{\e}
-\int_{\cO} Z^{\alpha}(R^{\e}\cdot\nabla B^{\e}_{app})\cdot Z^{\alpha}r^{\e}
-\e\int_{\cO}\Delta Z^{\alpha} r^{\e}\cdot Z^{\alpha}r^{\e}\\
&-\e\int_{\cO}[Z^{\alpha},\Delta]r^{\e}\cdot Z^{\alpha}r^{\e}
+\int_{\cO}Z^{\alpha}\nabla\pi^{\e}\cdot Z^{\alpha}r^{\e}
=\int_{\cO}Z^{\alpha}f^{\e}\cdot Z^{\alpha}r^{\e}.
\end{split}\eeq

\noindent$\bullet$ \underline{$\int_{\cO}Z^{\alpha}(r^{\e}\cdot\nabla u^{\e}_{app})\cdot Z^{\alpha}r^{\e}$ and
$\int_{\cO} Z^{\alpha}(R^{\e}\cdot\nabla B^{\e}_{app})\cdot Z^{\alpha}r^{\e}$.}\vspace{0.2cm}

By using Leibniz's formula and Lemma \ref{appes}, we find that
\begin{align*}
\left|\int_{\cO}Z^{\alpha}(r^{\e}\cdot\nabla u^{\e}_{app})\cdot Z^{\alpha}r^{\e}\right|
&\leq C\|\nabla u^{\e}_{app}\|_{m,\infty}\|r^{\e}\|_m^2\leq C\t^{-\gamma}\|r^{\e}\|^2_m,\\
\left|\int_{\cO}Z^{\alpha}(R^{\e}\cdot\nabla B^{\e}_{app})\cdot Z^{\alpha}R^{\e}\right|
&\leq C\|\nabla B^{\e}_{app}\|_{m,\infty}\|R^{\e}\|_{m}^2\leq C\e\chi_{[0,T]}\|R^{\e}\|^2_m.
\end{align*}

\noindent$\bullet$ \underline{$\int_{\cO}Z^{\alpha}(u^{\e}\cdot\nabla r^{\e}) \cdot Z^{\alpha}r^{\e}$.}\vspace{0.2cm}

We decompose the vector field $u^{\e}\cdot\nabla$ into
\begin{equation*}
u^{\e}\cdot\nabla=u^0\cdot\nabla+(u^{\e}-u^0)\cdot\nabla ,
\end{equation*}
and write
\begin{equation}\label{dcue}
\begin{split}
\int_{\cO} Z^{\alpha}(u^{\e}\cdot\nabla r^{\e}) \cdot Z^{\alpha}r^{\e}
=&\int_{\cO}[Z^{\alpha},u^0\cdot\nabla]r^{\e}\cdot Z^{\alpha}r^{\e}
+\int_{\cO}u^0\cdot\nabla Z^{\alpha}r^{\e}\cdot Z^{\alpha}r^{\e}\\
&+\int_{\cO}Z^{\alpha}((u^{\e}-u^0)\cdot\nabla r^{\e}) \cdot Z^{\alpha}r^{\e}.
\end{split}
\end{equation}
We observe from \eqref{unphi} and \eqref{S4eq6} that
  $u^0\cdot\nabla=\sum_{i=1}^3\btau^iZ_i+u^0_{\flat}Z_0$ is a tangential derivative. In view of \eqref{cm}, $[Z^{\alpha}, u^0\cdot\nabla]$ is a tangential derivative of order $m$. Note that $u^0,u^0_{\flat}$ are smooth functions supported in $[0,T]$, we arrive at
\begin{align*}
\left|\int_{\cO}[Z^{\alpha},u^0\cdot\nabla]r^{\e}\cdot Z^{\alpha}r^{\e}\right|\leq C\chi_{[0,T]}\|r^{\e}\|^2_{m}.
\end{align*}
By using an integration by parts, one has
\begin{align*}
\left|\int_{\cO}u^0\cdot\nabla Z^{\alpha}r^{\e}\cdot Z^{\alpha}r^{\e}\right|
=\left|\frac{1}{2}\int_{\cO}\sigma^0|Z^{\alpha}r^{\e}|^2\right|
\leq C\chi_{[0,T]}\|r^{\e}\|^2_{m},
\end{align*}
For the third term in \eqref{dcue}, since $\dive (u^{\e}-u^0)=\e\sigma^2$ and $u^{\e}\cdot\n=u^0\cdot\n=0$ on $\p\cO$,
we get, by using an integration by parts, that
\beq\label{Zap}
\begin{split}
&\int_{\cO}Z^{\alpha}((u^{\e}-u^0)\cdot\nabla r^{\e})\cdot Z^{\alpha}r^{\e}
=\int_{\cO}Z^{\alpha}(\dive(r^{\e}\otimes(u^{\e}-u^0))-\e\sigma^2r^{\e})\cdot Z^{\alpha}r^{\e}\\
&=\sum_{|\alpha'|\leq m}c_{\alpha'}\int_{\cO}Z^{\alpha'}(r^{\e}\otimes (u^{\e}-u^0)):\nabla Z^{\alpha}r^{\e}
-\int_{\cO}\e Z^{\alpha}(\sigma^2r^{\e})\cdot Z^{\alpha}r^{\e},
\end{split}\eeq
for some smooth function $c_{\alpha'}$.
Since $\sigma^2=\beta \dive u_a\in C^{\infty}([0,T],H^{23}(\cO)),$ we have
\begin{equation*}
\left|\int_{\cO}\e Z^{\alpha}(\sigma^2r^{\e})\cdot Z^{\alpha}r^{\e}\right|\leq C\e\chi_{[0,T]}\|r^{\e}\|^2_{m}.
\end{equation*}
While due to $\|u^{\e}_{app}-u^0\|_{m,\infty}\leq C\e^{\frac{1}{2}}\t^{-\gamma},$ we get
\begin{align*}
\left|\int_{\cO}Z^{\alpha'}(r^{\e}\otimes (u^{\e}_{app}-u^0)):\nabla Z^{\alpha}r^{\e}\right|
&\leq C\e^{\frac{1}{2}}\t^{-\gamma}\|r^{\e}\|_{m}\|\nabla r^{\e}\|_{m}\\
&\leq \lambda \e\|\nabla r^{\e}\|^2_m+C_{\lambda}\t^{-2\gamma}\|r^{\e}\|^2_m,
\end{align*}
for any $\lambda>0$ and an associate constant $C_{\lambda}>0$.
It follows from Leibniz's formula and the generalized Sobolev-Gagliardo-Nirenberg-Moser inequality that
\begin{align*}
\e^{\frac{3}{2}}\int_{\cO}Z^{\alpha'}(r^{\e}\otimes r^{\e}):\nabla Z^{\alpha}r^{\e}
&\leq \e^{\frac{3}{2}}\|r^{\e}\|_{L^{\infty}(\cO)}\|r^{\e}\|_m\|\nabla r^{\e}\|_m\\
&\leq \lambda\e\|\nabla r^{\e}\|^2_m+C_{\lambda}\e^2\|r^{\e}\|^2_m\|r^{\e}\|^2_{L^{\infty}(\cO)},
\end{align*}
for any $\lambda>0$ and an associated constant $C_{\lambda}>0$.

By substituting the above estimates into \eqref{dcue}, we achieve
\begin{align*}
\left|\int_{\cO}Z^{\alpha}(u^{\e}\cdot\nabla r^{\e}) \cdot Z^{\alpha}r^{\e}\right|
\leq \lambda\e\|\nabla r^{\e}\|^2_{m}
+C_{\lambda}\|r^{\e}\|^2_{m}\left(\t^{-\gamma}+\e^2\|r^{\e}\|^2_{L^{\infty}(\cO)}\right),
\end{align*}
for any $\lambda>0$ and an associate constant $C_{\lambda}>0.$

\vspace{0.2cm}
\noindent$\bullet$ \underline{$\int_{\cO}Z^{\alpha}(B^{\e}\cdot\nabla R^{\e})\cdot Z^{\alpha}r^{\e}$.}\vspace{0.2cm}

Since there is no diffusion term, we cannot gain regularity for $R^{\e}$ by using the equation of \eqref{eqrR}$_2$. Therefore,  we have to transfer the extra derivative of $R^{\e}$ by using integration by parts. Indeed similar to \eqref{Zap}, we have
\begin{align*}
\int_{\cO}Z^{\alpha}(B^{\e}\cdot\nabla R^{\e})\cdot Z^{\alpha}r^{\e}
=\sum_{|\alpha'|\leq m}c_{\alpha'}\int_{\cO}Z^{\alpha'}(R^{\e}\otimes B^{\e}):\nabla Z^{\alpha}r^{\e},
\end{align*}
for some smooth function $c_{\alpha'}$.
We decompose $B^{\e}=B^{\e}_{app}+\e^{\frac{3}{2}}R^{\e}$, by using Lemma \ref{appes} and Lemma \ref{SGNM},
\begin{align*}
\left|\int_{\cO}Z^{\alpha}(B^{\e}\cdot\nabla R^{\e})\cdot Z^{\alpha}r^{\e}\right|
&\leq C\|\nabla r^{\e}\|_m\big(\|B^{\e}_{app}\|_{m,\infty}\|R^{\e}\|_m+\e^{\frac{3}{2}}\|R^{\e}\|_m\|R^{\e}\|_{L^{\infty}(\cO)} \big)\\
&\leq \lambda \e\|\nabla r^{\e}\|^2_m+C_{\lambda}\|R^{\e}\|^2_m\big(\e\chi_{[0,T]}+\e^2\|R^{\e}\|^2_{L^{\infty}(\cO)}\big),
\end{align*}
for any $\lambda>0$ and an associate constant $C_{\lambda}>0$.

\vspace{0.2cm}
\noindent$\bullet$ \underline{$\e\int_{\cO}\Delta Z^{\alpha} r^{\e}\cdot Z^{\alpha}r^{\e}$.}\vspace{0.2cm}

Similar to \eqref{l28}, by using integration by parts, Korn's inequality and the boundary condition $\mathcal{N}(r^{\e})=N^{\e}$ with $\|N^{\e}\|_{H^6(\cO)}\leq C\t^{-\gamma}$, we find that there exist constants $c_0,C>0$ such that
\begin{equation*}
-\sum_{|\alpha|=m}\int\e\Delta Z^{\alpha}r^{\e}\cdot Z^{\alpha}r^{\e}\geq c_0\e\|\nabla r^{\e}\|^2_{m}-C\e\|r^{\e}\|^2_m-C\e\t^{-2\gamma}.
\end{equation*}
For more details , we refer to \cite[Lemma 5.4]{LSZ1}.

\vspace{0.2cm}
\noindent$\bullet$ \underline{$\int_{\cO}\e[Z^{\alpha},\Delta]r^{\e}\cdot Z^{\alpha}r^{\e}$.}\vspace{0.2cm}

Thanks to \eqref{Z2}, we get, by using an integration by parts,  that
\begin{align*}
\left|\int_{\cO}[\Delta,Z^{\alpha}] r^{\e}\cdot Z^{\alpha}r^{\e}\right|
&\leq C\sum_{|\beta_1|,|\beta_2|\leq m-1}\left|\int_{\cO}(c_{\beta_1}\nabla^2Z^{\beta_1}r^{\e}+c_{\beta_2}\nabla Z^{\beta_2}r^{\e})\cdot Z^{\alpha}r^{\e}\right|\\
&\leq C(\|\nabla r^{\e}\|_m+\| r^{\e}\|_m)\|\nabla r^{\e}\|_{m-1}+C\|\nabla r^{\e}\|_{H^{m-1}(\p\cO)}\|r^{\e}\|_{H^m(\p\cO)}.
\end{align*}
While it follows from Lemma \ref{norpa} and the trace theorem that
\begin{align*}
\|\nabla r^{\e}\|_{H^{m-1}(\p\cO)}\|r^{\e}\|_{H^{m}(\p\cO)}
&\leq C\big(\|r^{\e}\|_{H^{m}(\p\cO)}^2+\t^{-2\gamma}\big)\\
&\leq C\left(\|r^{\e}\|^2_m+\|r^{\e}\|_m\|\nabla r^{\e}\|_m+\t^{-2\gamma} \right)\\
&\leq \lambda \|\nabla r^{\e}\|^2_m+C_{\lambda}\|r^{\e}\|^2_m+C\t^{-2\gamma}.
\end{align*}
As a consequence, it comes out
\begin{align*}
\left|\int_{\cO}\e[\Delta,Z^{\alpha}]r^{\e}\cdot Z^{\alpha}r^{\e}\right|\leq \lambda\e\|\nabla r^{\e}\|^2_m+C_{\lambda}\e(\|\nabla r^{\e}\|^2_{m-1}+\|r^{\e}\|^2_m+\t^{-2\gamma}).
\end{align*}

\vspace{0.2cm}
\noindent$\bullet$ \underline{$\int_{\cO}Z^{\alpha}f^{\e}\cdot Z^{\alpha}r^{\e}$.}\vspace{0.2cm}

By Lemma \ref{appes}, we find that
\begin{align*}
\left|\int_{\cO}Z^{\alpha}f^{\e}\cdot Z^{\alpha}r^{\e}\right|
\leq \|f^{\e}\|_{m} \|r^{\e}\|_{m}
\leq C\e^{\frac{1}{2}}\t^{-\gamma} \|r^{\e}\|_{m}
\leq C\t^{-\gamma}(\e+\|r^{\e}\|_{m}^2).
\end{align*}

By substituting the above inequalities into \eqref{S4eq7}, and then taking summation for all $|\alpha|\leq m$ and
 integrating the resulting inequality   over $ (0,t),$ and finally  choosing a sufficiently
 small constant $\lambda$, we deduce that there exist some constants $c_0,C>0$ such that
 \beq\label{Zar}
\begin{split}
&\|r^{\e}(t)\|^2_m+c_0\e\int_0^t\|\nabla r^{\e}(s)\|^2_mds\leq C\e^{\frac{1}{2}}
+C\sum_{|\alpha|\leq m}\left|\int_0^t\int_{\cO}Z^{\alpha}\nabla\pi^{\e}\cdot Z^{\alpha}r^{\e}dxds\right|\\
&+C\e\int_{0}^t\|\nabla r^{\e}\|^2_{m-1}\,ds
+C\int_0^t\|(r^{\e},R^{\e})\|^2_{m}\left(\e+\s^{-\gamma}+\e^2\|(r^{\e},R^{\e})\|^2_{L^{\infty}(\cO)}\right)\,ds.
\end{split}\eeq
 Then by virtue of Proposition \ref{L2et}, by taking summation of the above inequalities for index $1,2,\cdots,m$, we can eliminate the first term on the second line of \eqref{Zar}.

\noindent{\bf Step 2.} The tangential derivative estimates of $R^{\e}.$

The difficulty lies in the fact that: since there is no diffusion term in $R^{\e}$ equations, we cannot gain
additional one derivative  for $R^{\e}$ by using the equation of \eqref{eqrR}$_2$.
Let $|\alpha|\leq m\in\mathbb{N}_+,$ we get, by applying $Z^{\alpha}$ to \eqref{eqrR}$_2$ and then taking $L^2$
inner product of the resulting equation with $Z^{\alpha}R^{\e},$ that
\beq\label{S4eq8}
\begin{split}
&\frac{1}{2}\frac{d}{dt}\|Z^{\alpha}R^{\e}(t)\|^2_{L^2(\cO)}
+\int_{\cO}u^{\e}\cdot\nabla Z^{\alpha}R^{\e}\cdot Z^{\alpha}R^{\e}
+\int_{\cO}[Z^{\alpha},u^{\e}\cdot\nabla] R^{\e}\cdot Z^{\alpha} R^{\e}\\
&+\int_{\cO}Z^{\alpha}(r^{\e}\cdot\nabla B^{\e}_{app})\cdot Z^{\alpha} R^{\e}
-\int_{\cO}Z^{\alpha}(B^{\e}\cdot\nabla r^{\e})\cdot Z^{\alpha} R^{\e}
-\int_{\cO}Z^{\alpha}(R^{\e}\cdot\nabla u^{\e}_{app})\cdot Z^{\alpha} R^{\e}\\
&+\int_{\cO}Z^{\alpha}(\sigma^{\e} R^{\e})\cdot Z^{\alpha} R^{\e}
=\int_{\cO}Z^{\alpha}F^{\e}\cdot Z^{\alpha} R^{\e}.
\end{split}\eeq

\noindent$\bullet$ \underline{$\int_{\cO}u^{\e}\cdot\nabla Z^{\alpha}R^{\e}\cdot Z^{\alpha}R^{\e}$.} \vspace{0.2cm}

Since $\dive u^{\e}=\sigma^0+\e\sigma^2$ and
$ u^{\e}\cdot\n=0$ on the boundary, we get, by using  integration by parts, that
\begin{align*}
\left|\int_{\cO}u^{\e}\cdot\nabla Z^{\alpha}R^{\e}\cdot Z^{\alpha}R^{\e}\right|
=\left|\int_{\cO}\frac{1}{2}(\sigma^0+\e\sigma^2)|Z^{\alpha}R^{\e}|^2\right|
\leq C\chi_{[0,T]}\|R^{\e}\|_m^2.
\end{align*}

\noindent$\bullet$ \underline{$\int_{\cO} Z^{\alpha}(r^{\e}\cdot\nabla B^{\e}_{app})\cdot Z^{\alpha} R^{\e}$ and $\int_{\cO}Z^{\alpha}(R^{\e}\cdot\nabla u^{\e}_{app})\cdot Z^{\alpha} R^{\e}$.}\vspace{0.2cm}

It follows from  Lemma \ref{appes} that
\beq\label{ZarBR}
\begin{split}
\left|\int_{\cO}Z^{\alpha}(r^{\e}\cdot\nabla B^{\e}_{app})\cdot Z^{\alpha} R^{\e}\right|
&\leq C\|\nabla B^{\e}_{app}\|_{m,\infty}\|r^{\e}\|_m\|R^{\e}\|_m
\leq C\e\chi_{[0,T]}\|r^{\e}\|_m\|R^{\e}\|_m,\\
\left|\int_{\cO}Z^{\alpha}(R^{\e}\cdot\nabla u^{\e}_{app})\cdot Z^{\alpha} R^{\e}\right|
&\leq C\|\nabla u^{\e}_{app}\|_{m,\infty}\|R^{\e}\|^2_m
\leq C\t^{-\gamma}\|R^{\e}\|^2_m.
\end{split}\eeq

\noindent$\bullet$ \underline{$\int_{\cO}[Z^{\alpha},u^{\e}\cdot\nabla] R^{\e}\cdot Z^{\alpha} R^{\e}$.}\vspace{0.2cm}

By Leibniz's formula, there are some smooth functions $c_{\alpha_{1}}$ such that
\begin{align*}
[Z^{\alpha},u^{\e}\cdot\nabla]R^{\e}=\sum_{\alpha_1+\alpha_2=\alpha,\alpha_1\ne 0}c_{\alpha_1}Z^{\alpha_1}(u^{\e}\cdot\nabla)Z^{\alpha_2}R^{\e}.
\end{align*}
We split $u^{\e}$ into $u^{\e}=u^{\e}_{app}+\e^{\frac{3}{2}}r^{\e}.$
By virtue of Proposition \ref{optan}, we write
\begin{align}\label{ueapp1}
u^{\e}_{app}\cdot\nabla=u^{\e}_{app}\cdot \btau^1Z_1+u^{\e}_{app}\cdot \btau^2Z_2+u^{\e}_{app}\cdot \btau^3Z_3+\frac{u^{\e}_{app}\cdot\n}{\varphi}Z_0.
\end{align}
Since $u^{\e}_{app}\cdot\n=0$ on the boundary and $\dive u^{\e}_{app}=\sigma^0+\e\sigma^2-\e^{\frac{3}{2}}d^{\e}$, we deduce from Lemma \ref{unfi} that,
\begin{align}\label{ueapp2}
\|\frac{u^{\e}_{app}\cdot\n}{\varphi}\|_{m,\infty}\leq \|u^{\e}_{app}\|_{m+1,\infty}+\|\sigma^0+\e\sigma^2-\e^{\frac{3}{2}}d^{\e}\|_{m,\infty}\leq C\t^{-\gamma}.
\end{align}
Observing  that $\alpha_1\neq 0,$ by using \eqref{ueapp1}, we obtain
\begin{align}\label{Za1}
\|[Z^{\alpha},u^{\e}_{app}\cdot\nabla]R^{\e}\|_{L^2(\cO)}\leq C\bigl(\|u^{\e}_{app}\|_{m,\infty}+\|\frac{u^{\e}_{app}\cdot\n}{\varphi}\|_{m,\infty}\bigr)\|R^{\e}\|_m\leq C\t^{-\gamma}\|R^{\e}\|_m.
\end{align}
Similar to the proof of \eqref{unabv}  and  noting that $\alpha_1\neq 0, \dive r^{\e}=d^{\e},$ we find that
\begin{align*}
\|[Z^{\alpha},r^{\e}\cdot\nabla]R^{\e}\|_{L^2(\cO)}\leq C\left(\|r^{\e}\|_{2,\infty}\|R^{\e}\|_m+\|r^{\e}\|_{m+1}\|R^{\e}\|_{1,\infty}+\|d^{\e}\|_{m,\infty}\|R^{\e}\|_m\right).
\end{align*}
By summarizing the above two estimates, we achieve
\begin{align*}
\left|\int_{\cO}[Z^{\alpha}, u^{\e}\cdot\nabla]R^{\e}\cdot Z^{\alpha}R^{\e}\right|
&\leq C\bigl(\t^{-\gamma}\|R^{\e}\|_m^2+\e^{\frac{3}{2}}(\|r^{\e}\|_{2,\infty}\|R^{\e}\|_m^2+\|r^{\e}\|_{m+1}\|R^{\e}\|_{1,\infty}\|R^{\e}\|_m)\bigr)\\
&\leq \lambda\e\|r^{\e}\|_{m+1}^2
+C_{\lambda}\|R^{\e}\|^2_m\left(\t^{-\gamma}+\e^{\frac{3}{2}}\|r^{\e}\|_{2,\infty}+\e^2\|R^{\e}\|^2_{1,\infty}\right),
\end{align*}
for any $\lambda>0$ and an associated constant $C_{\lambda}>0$.

\vspace{0.2cm}
\noindent$\bullet$ \underline{$\int_{\cO}Z^{\alpha}(B^{\e}\cdot\nabla r^{\e})\cdot Z^{\alpha} R^{\e}$.}\vspace{0.2cm}

 It follows from Lemma \ref{SGNM} that
\begin{align*}
&\int_{\cO}Z^{\alpha}(B^{\e}\cdot\nabla r^{\e})\cdot Z^{\alpha} R^{\e}\\
&\leq C \|B^{\e}_{app}\|_{m,\infty}\|\nabla r^{\e}\|_m\|R^{\e}\|_m
+C\e^{\frac{3}{2}}\left(\|R^{\e}\|_{L^{\infty}(\cO)}\|\nabla r^{\e}\|_m+\|R^{\e}\|_m\|\nabla r^{\e}\|_{L^{\infty}(\cO)}\right)\|R^{\e}\|_m\\
&\leq \lambda\e\|\nabla r^{\e}\|_m^2
+C\|R^{\e}\|_m^2\left(\e\chi_{[0,T]}+\e^2\|R^{\e}\|^2_{L^{\infty}(\cO)}+\e^{\frac{3}{2}}\|\nabla r^{\e}\|_{L^{\infty}(\cO)}\right),
\end{align*}
for any $\lambda>0$ and an associated constants $C_{\lambda}>0.$

\vspace{0.2cm}
\noindent$\bullet$ \underline{$\int_{\cO}Z^{\alpha}(\sigma^{\e} R^{\e})\cdot Z^{\alpha} R^{\e}$.}\vspace{0.2cm}

Since $\sigma^{\e}=\sigma^0+\e\sigma^2$ is supported in $[0,T]$, one has
\begin{align*}
\left|\int_{\cO}Z^{\alpha}(\sigma^{\e} R^{\e})\cdot Z^{\alpha} R^{\e}\right|
\leq C\chi_{[0,T]}\|R^{\e}\|^2_m.
\end{align*}

\noindent$\bullet$ \underline{$\int_{\cO}Z^{\alpha}F^{\e}\cdot Z^{\alpha}R^{\e}$.}\vspace{0.2cm}

By applying Lemma \ref{appes}, we find
\begin{align*}
\left|\int_{\cO}Z^{\alpha}F^{\e}\cdot Z^{\alpha}R^{\e}\right|
\leq \|F^{\e}\|_{m} \|R^{\e}\|_{m}
\leq C\e^{\frac{1}{4}}\chi_{[0,T]}\|R^{\e}\|_{m}
\leq C\chi_{[0,T]}(\e^{\frac{1}{2}}+\|R^{\e}\|_{m}^2).
\end{align*}

By substituting the above estimates into \eqref{S4eq8} and then taking the sum on all $|\alpha|\leq m$ and   integrating the resulting
inequality over $(0,t)$, we conclude that for any $\lambda>0,$ there exist constants $C>0$ such that
\beq\label{ZaR}
\begin{split}
\|R^{\e}(t)\|_m^2\leq&\ C\e^{\frac{1}{2}}+\lambda\e\int_0^t\|\nabla r^{\e}(s)\|^2_mds\\
&+C_{\lambda}\int_0^t\|(r^{\e}, R^{\e})\|^2_m\left(\e+\s^{-\gamma}+\e^2\|(r^{\e},R^{\e})\|^2_{2,\infty}+\e^{\frac{3}{2}}\|\nabla r^{\e}\|_{L^{\infty}(\cO)}\right).
\end{split}\eeq

We conclude the proof of \eqref{ZarR} by combing \eqref{Zar} with \eqref{ZaR} and by  choosing  $\lambda>0$ to be sufficiently small.
\end{proof}

\subsection{Normal derivatives estimates}
In order to estimate the normal derivatives, as in \cite{LSZ1}, we introduce
\begin{align*}
\eta^{\e}:=\sqrt{\e}\chi(\mathcal{N}(r^{\e})-N^{\e}) \andf Q^{\e}:=\sqrt{\e}\p_{\n} R^{\e},
\end{align*}
where $\chi$ is a cut-off function supported in $\mathcal{V}_{\delta_0}$, $\chi(x)=1$ when $x\in \mathcal{V}_{\delta_0/2}$. By definition, $\eta^{\e}=0$ on the boundary. Moreover, we deduce from the definition of $\mathcal{N}$ and $r^{\e}\cdot\n=0$ on the boundary that
\begin{align}
\label{etaN}\eta^{\e}=\sqrt{\e}\chi\Bigl(\frac{1}{2}(\p_{\n}r^{\e})_{\tan}+(Mr^{\e}-\frac{1}{2}r^{\e}\cdot\nabla\n)_{\tan}-N^{\e}\Bigr).
\end{align}
While it follows from Proposition \ref{optan} that
\begin{align}\label{diver2}
\p_{\n}r^{\e}\cdot\n=d^{\e}-Z_1r^{\e}\cdot \btau^1-Z_2r^{\e}\cdot \btau^2-Z_3r^{\e}\cdot \btau^3.
\end{align}
By combining \eqref{etaN}, \eqref{diver2} with Lemma \ref{appes}, we can prove the following Lemma:
\begin{lemma}\label{eqvr}
{\sl
Let $m\in \mathbb{N}$ with $1\leq m\leq 5$ ,the following equivalences hold true:
\begin{align*}
\|\eta^{\e}\|_{m-1}+\sqrt{\e}(\|r^{\e}\|_m+\t^{-\gamma})
&\approx \sqrt{\e}(\|\nabla r^{\e}\|_{m-1}+\|r^{\e}\|_m+\t^{-\gamma}),\\
\|\eta^{\e}\|_{L^{\infty}(\cO)}+\sqrt{\e}(\|r^{\e}\|_{1,\infty}+\t^{-\gamma})
&\approx \sqrt{\e}(\|\nabla r^{\e}\|_{L^{\infty}(\cO)}+\|r^{\e}\|_{1,\infty}+\t^{-\gamma}).
\end{align*}
}
\end{lemma}
For more details, we refer to \cite[Lemma 5.5]{LSZ1}. Here $\eta^{\e}$ is nice substitute of $\p_{\n}r^{\e}$. 
\begin{proposition}\label{nores}
{\sl
Let $\gamma=\frac{5}{4}$, $m\in \mathbb{N}$ with $1\leq m\leq 5,$ there exists a constant $c_0,C>0$ such that for any $t\in [0,\frac{T}{\e}],$
\beq\label{S4eq10}
\begin{split}
&\|(\eta^{\e},Q^{\e})(t)\|^2_{m-1}+c_0\e\int_0^t\|\nabla \eta(s)\|^2_{m-1}ds\\
&\leq  C\e^{\frac{1}{2}}
+C\sum_{|\beta|\leq m-1} \left|\int_0^t\int_{\cO} Z^{\beta}(\sqrt{\e}\chi\mathcal{N}(\nabla\pi^{\e}))\cdot Z^{\beta}\eta^{\e}dxds\right|\\
&\quad+C\int_0^t\left(\|(r^{\e},R^{\e})\|^2_m+\|(\eta^{\e},Q^{\e})\|^2_{m-1}\right)\\
&\hspace{1.5cm}\cdot\left(\e+\s^{-\gamma}+\e^2\|(r^{\e},R^{\e})\|^2_{2,\infty}+\e^2\|\eta^{\e}\|^2_{1,\infty}
+\e^2\|Q^{\e}\|^2_{L^{\infty}(\cO)}\right)\,ds.
\end{split}\eeq
}\end{proposition}

\begin{proof} We divide the proof of this proposition into the following two steps:

\noindent{\bf Step 1.}  The estimate of $\|\eta^{\e}(t)\|^2_{m-1}.$

In view of \eqref{eqrR}, $\eta^{\e}$ satisfies
\beq\label{eqeta}
\begin{split}
\p_t\eta^{\e}-\e\Delta \eta^{\e}=\ &\sqrt{\e}\chi\mathcal{N}(f^{\e}-\nabla\pi^{\e}-u^{\e}\cdot\nabla r^{\e}-r^{\e}\cdot\nabla u^{\e}_{app}+B^{\e}\cdot\nabla R^{\e}+R^{\e}\cdot\nabla B^{\e}_{app})\\
&\quad -\e^{\frac{3}{2}}[\Delta,\chi\mathcal{N}]r^{\e}
-(\p_t-\e\Delta)(\sqrt{\e}\chi N^{\e}).
\end{split}\eeq
Let $|\beta|\leq m-1$, since $\eta^{\e}=0$ on the boundary, by applying $Z^{\beta}$ to \eqref{eqeta} and then taking
 $L^2$ inner product of the resulting equation with $Z^{\beta}\eta^{\e},$ we find that
 \beq\label{S4eq11}
\begin{split}
&\frac{1}{2}\frac{d}{dt}\|Z^{\beta}\eta^{\e}(t)\|^2_{L^2(\cO)}
+\e\|\nabla Z^{\beta}\eta^{\e}\|^2_{L^2(\cO)}\\
&=\int_{\cO}Z^{\beta}(\sqrt{\e}\chi\CN(f^{\e}-\nabla\pi^{\e}-u^{\e}\cdot\nabla r^{\e}-r^{\e}\cdot\nabla u^{\e}_{app}+B^{\e}\cdot\nabla R^{\e}+ R^{\e}\cdot\nabla B^{\e}_{app}))\cdot Z^{\beta}\eta^{\e}\\
&\quad -\int_{\cO}\e^{\frac{3}{2}}[\Delta,Z^{\beta}\chi\CN]r^{\e}\cdot Z^{\beta}\eta^{\e}
-\int_{\cO}Z^{\beta}(\p_t-\e\Delta)(\sqrt{\e}\chi N^{\e})\cdot Z^{\beta}\eta^{\e}.
\end{split}\eeq

\noindent$\bullet$ \underline{$I_1:=\int_{\cO}Z^{\beta}(\sqrt{\e}\chi\CN(f^{\e}))\cdot Z^{\beta}\eta^{\e}$.}

It is easy to observe that
\begin{align*}
|I_1|\leq \|Z^{\beta}(\sqrt{\e}\chi\CN(f^{\e}))\|_{L^2(\cO)}\|\eta^{\e}\|_{m-1}
\leq \sqrt{\e}\left(\|f^{\e}\|_{m-1}+\|\nabla f^{\e}\|_{m-1}\right)\|\eta^{\e}\|_{m-1},
\end{align*}
from which and Lemma \ref{appes}, we infer
\begin{align*}
|I_1|\leq C\e^{\frac{3}{4}}\t^{-\gamma}\|\eta^{\e}\|_{m-1}
\leq C\e^{\frac{3}{2}}\t^{-\gamma}+C\t^{-\gamma}\|\eta^{\e}\|_{m-1}^2.
\end{align*}

\noindent$\bullet$ \underline{$I_2:=\int_{\cO}Z^{\beta}(\p_t-\e\Delta)(\sqrt{\e}\chi N^{\e})\cdot Z^{\beta}\eta^{\e}$.}\vspace{0.2cm}

Note that $m\leq 5,$ by Sobolev embedding and  Lemma \ref{appes}, $\|N^{\e}\|_{5,\infty}\leq C\|N^{\e}\|_{H^7(\cO)}\leq C\t^{-\gamma}$. So that
we infer
\beq\label{ptN}
\|Z^{\beta}(\p_t-\e\Delta)(\sqrt{\e}\chi N^{\e})\|_{L^2(\cO)} \leq C\e^{\frac{1}{2}}\big(\|\p_t N^{\e}\|_{4}+\e\|N^{\e}\|_{6}\big)
\leq C\e^{\frac{1}{2}}\t^{-\gamma},
\eeq
which implies
\begin{align*}
|I_2|\leq C\e\t^{-\gamma}+C\t^{-\gamma}\|\eta^{\e}\|^2_{m-1}.
\end{align*}

\noindent$\bullet$ \underline{$I_3:=\int_{\cO}Z^{\beta}(\sqrt{\e}\chi\CN(u^{\e}\cdot\nabla r^{\e}))\cdot Z^{\beta}\eta^{\e}$.}\vspace{0.2cm}

Note that $\eta^{\e}=0$ on the boundary, we can transfer one derivative in $\mathcal{N}$ to $Z^{\beta}\eta^{\e}$ by using integration by parts.
So that we infer
\begin{align*}
|I_3|&\leq C\sqrt{\e}\|u^{\e}\cdot\nabla r^{\e}\|_{m-1}\|\nabla \eta^{\e}\|_{m-1}\\
&\leq C\sqrt{\e}\left(\|u^{\e}_{app}\cdot\nabla r^{\e}\|_{m-1}+\e^{\frac{3}{2}}\|r^{\e}\cdot\nabla r^{\e}\|_{m-1}\right)\|\nabla\eta^{\e}\|_{m-1}.
\end{align*}
It follows from \eqref{ueapp1} and \eqref{ueapp2} that
\begin{align*}
\|u^{\e}_{app}\cdot\nabla r^{\e}\|_{m-1}
\leq C\bigl(\|u^{\e}_{app}\|_{m-1,\infty}+\|\frac{u^{\e}_{app}\cdot\n}{\varphi}\|_{m-1,\infty}\bigr)\|r^{\e}\|_{m}
\leq C\t^{-\gamma}\|r^{\e}\|_{m}.
\end{align*}
Since $\dive r^{\e}=d^{\e}$ and $r^{\e}$ is tangent to the boundary, we deduce from \eqref{unabv} that
\begin{align}\label{rnabr}
\|r^{\e}\cdot\nabla r^{\e}\|_{m-1}
\leq C\|r^{\e}\|_m(\|r^{\e}\|_{1,\infty}+\|d^{\e}\|_{m-1,\infty})
\leq C\|r^{\e}\|_m(\|r^{\e}\|_{1,\infty}+\t^{-\gamma}).
\end{align}
By combining the above inequalities with Lemma \ref{eqvr}, we find
\begin{align*}
|I_3|\leq \lambda\e\|\nabla\eta^{\e}\|^2_{m-1}
+C_{\lambda}\|r^{\e}\|^2_{m}
\left(\t^{-2\gamma}+\e^3\|r^{\e}\|^2_{1,\infty}\right),
\end{align*}
for any $\lambda>0$ and an associated constant $C_{\lambda}>0$.

\vspace{0.2cm}
\noindent$\bullet$ \underline{$I_4:=\int_{\cO}Z^{\beta}(\sqrt{\e}\chi\CN(r^{\e}\cdot\nabla u^{\e}_{app}))\cdot Z^{\beta}\eta^{\e}$.}\vspace{0.2cm}

We get, by applying Lemma \ref{eqvr} and Lemma \ref{appes}, that
\beq\label{n2uapp1}
\begin{split}%
&\|Z^{\beta}(\sqrt{\e}\chi\CN(r^{\e}\cdot\nabla u^{\e}_{app})\|\\
&\leq C\sqrt{\e}(\|\nabla r^{\e}\|_{m-1}+\| r^{\e}\|_{m-1})\|\nabla u^{\e}_{app}\|_{m-1,\infty}+C\sqrt{\e}\|r^{\e}\|_{m-1}\|\nabla^2 u^{\e}_{app}\|_{m-1,\infty}\\
&\leq C\t^{-\gamma}(\|\eta^{\e}\|_{m-1}+\|r^{\e}\|_m+\sqrt{\e}\t^{-\gamma}),
\end{split}\eeq
from which, we infer
\begin{align*}
|I_4|\leq C\e\t^{-\gamma}+C\t^{-\gamma}(\|\eta^{\e}\|^2_{m-1}+\|r^{\e}\|^2_m).
\end{align*}

\noindent$\bullet$ \underline{$I_5:=\int_{\cO}Z^{\beta}(\sqrt{\e}\chi\CN(B^{\e}\cdot\nabla R^{\e}))\cdot Z^{\beta}\eta^{\e}$.}\vspace{0.2cm}

Similar to the estimate of $I_3$, by using integration by parts, one has
\begin{align*}
|I_5|&\leq C\sqrt{\e}\|B^{\e}\cdot\nabla R^{\e}\|_{m-1}\|\nabla \eta^{\e}\|_{m-1}\\
&\leq C\sqrt{\e}\left(\|B^{\e}_{app}\cdot\nabla R^{\e}\|_{m-1}
+\e^{\frac{3}{2}}\|R^{\e}\cdot\nabla R^{\e}\|_{m-1}\right)\|\nabla \eta^{\e}\|_{m-1}.
\end{align*}
Note that $\dive B^{\e}_{app}=\e\dive B^2=0, m\leq 5,$ similar to \eqref{ueapp1} and \eqref{ueapp2}, we get, by applying Lemma \ref{unfi} and Lemma \ref{appes}, that
\beq \label{BnabR}
\begin{split}
\|B^{\e}_{app}\cdot\nabla R^{\e}\|_{m-1}
&\leq C\Bigl(\|B^{\e}_{app}\|_{m-1,\infty}+\|\frac{B^{\e}_{app}\cdot\n}{\varphi}\|_{m-1,\infty}\Bigr)\|R^{\e}\|_m\\
&\leq C\|B^{\e}_{app}\|_{m,\infty}\|R^{\e}\|_m\leq C\e\chi_{[0,T]}\|R^{\e}\|_m.
\end{split}\eeq
Since $\dive R^{\e}=0$ and $R^{\e}$ is tangent to the boundary, we deduce from Lemma \ref{unfi} that
\begin{align}\label{RnabR}
\|R^{\e}\cdot\nabla R^{\e}\|_{m-1}
\leq C\|R^{\e}\|_{m}\|R^{\e}\|_{1,\infty}.
\end{align}
By summarizing the above inequalities, we obtain
\begin{align*}
|I_5|\leq \lambda\e\|\nabla\eta^{\e}\|^2_{m-1}
+C_{\lambda}\|R^{\e}\|^2_{m}
\left(\e^2\t^{-2\gamma}+\e^3\|R^{\e}\|^2_{1,\infty}\right),
\end{align*}
for any $\lambda>0$ and an associated constant $C_{\lambda}>0$.

\vspace{0.2cm}
\noindent$\bullet$ \underline{$I_6:=\int_{\cO}Z^{\beta}(\sqrt{\e}\chi\CN(R^{\e}\cdot\nabla B^{\e}_{app}))\cdot Z^{\beta}\eta^{\e}$.}\vspace{0.2cm}

Similar to the estimate of $I_3$, we get, by using integration by parts, that
\begin{align*}
|I_6| &\leq C\sqrt{\e}\|R^{\e}\cdot\nabla B^{\e}_{app}\|_{m-1}\|\nabla \eta^{\e}\|_{m-1}\\
&\leq C\sqrt{\e}\|R^{\e}\|_{m-1}\|\nabla B^{\e}_{app}\|_{m-1,\infty}\|\nabla\eta^{\e}\|_{m-1}
\leq \lambda\e\|\nabla \eta^{\e}\|_{m-1}^2+C_{\lambda}\e^2\chi_{[0,T]}\|R^{\e}\|_{m}^2,
\end{align*}
for any $\lambda>0$ and an associated constant $C_{\lambda}>0$.

\vspace{0.2cm}
\noindent$\bullet$ \underline{$I_7:=\int_{\cO}\e^{\frac{3}{2}}[\Delta,Z^{\beta}\chi\CN]r^{\e}\cdot Z^{\beta}\eta^{\e}$.}\vspace{0.2cm}

If $m=1, $ there are at most two order normal derivatives on $[\Delta,\chi\mathcal{N}]r^{\e}.$ So that one has
\begin{align*}
|I_7|\leq C\e^{\frac{3}{2}}\|\nabla^2 r^{\e}\|_{L^2(\cO)}\|\eta^{\e}\|_{L^2(\cO)}.
\end{align*}
In view of  Proposition \ref{optan}, we write
\begin{align}\label{nabre}
\nabla r^{\e}=\btau^1Z_1r^{\e}+\btau^2Z_2r^{\e}+\btau^3Z_3r^{\e}+\n\p_{\n}r^{\e}.
\end{align}
Then by combining \eqref{etaN} with \eqref{diver2}, we obtain
\begin{align*}
\sqrt{\e}\|\nabla^2 r^{\e}\|_{L^2(\cO)}&\leq C(\|\nabla \eta^{\e}\|_{L^2(\cO)}+\sqrt{\e}\|\nabla r^{\e}\|_{1}+\t^{-\gamma})\\
&\leq C(\|\nabla\eta^{\e}\|_{L^2(\cO)}+\|\eta^{\e}\|_1+\|r^{\e}\|_2+\t^{-\gamma}),
\end{align*}
which implies
\begin{align*}
|I_7|&\leq C\e(\|\nabla \eta^{\e}\|_{L^2(\cO)}+\|\eta^{\e}\|_{1}+\|r^{\e}\|_2+\t^{-\gamma})\|\eta^{\e}\|_{L^2(\cO)}\\
&\leq \lambda\e\|\nabla\eta^{\e}\|^2_{L^2(\cO)}+C_{\lambda}\e(\|r^{\e}\|^2_{2}+\|\eta^{\e}\|^2_{1})+C\e\t^{-2\gamma},
\end{align*}
for any constant $\lambda>0$ and an associated $C_{\lambda}>0.$

If $|\beta|=m-1$ with $ 2\leq m\leq 5,$  we observe from \eqref{Delf}  that there are at most three order normal derivatives in $[\Delta,Z^{\beta}\chi\mathcal{N}]r^{\e}$. Since $\eta^{\e}=0$ on $\p\cO$ implies $Z^{\beta}\eta^{\e}=0$ on $\p\cO$, by using integration by parts, we can transfer one derivative to $Z^{\beta}\eta^{\e}$, so that
\begin{align*}
|I_7|\leq C\e^{\frac{3}{2}}\|\nabla^2r^{\e}\|_{m-2}\|\nabla\eta^{\e}\|_{m-1}.
\end{align*}
Yet it follows from \eqref{nabre}, \eqref{etaN} and \eqref{diver2} that
\begin{align}\label{nab2r}
\sqrt{\e}\|\nabla^2 r^{\e}\|_{m-2}\leq C(\|\nabla\eta^{\e}\|_{m-2}+\|\eta^{\e}\|_{m-1}+\|r^{\e}\|_m+\t^{-\gamma}),
\end{align}
which implies
\begin{align*}
|I_7|\leq \lambda\e\|\nabla \eta^{\e}\|^2_{m-1}
+ C_{\lambda}\e(\|\nabla\eta^{\e}\|^2_{m-1}+\|\eta\|^2_{m-1}+\|r^{\e}\|^2_{m}+\t^{-2\gamma}),
\end{align*}
for any $\lambda>0$ and an associated constant $C_{\lambda}>0.$

Notice from Lemma \ref{appes} that the initial data $\eta^{\e}_0$ satisfies
$\|\eta^{\e}_0\|^2_{m-1}\leq C\e^{\frac{1}{2}}.$ Then we get, by first substituting the above estimates into \eqref{S4eq11}
and  taking summation for $|\beta|\leq m-1,$ and then integrating the resulting inequality  over $(0,t)$ and choosing $\lambda>0$
 to be a sufficiently small constant, that
 \beq \label{eqetae}
\begin{split}
&\|\eta^{\e}(t)\|^2_{m-1}+ c_0\e\int_0^t\|\nabla \eta^{\e}(s)\|^2_{m-1}ds\\
&\leq C\e^{\frac{1}{2}}
+\sum_{|\beta|\leq m-1}\left|\int_0^t\int_{\cO}Z^{\beta}(\sqrt{\e}\chi\mathcal{N}(\nabla\pi^{\e}))\cdot Z^{\beta}\eta^{\e}dxds\right|\\
&\quad+\int_0^t C\left(\|(r^{\e},R^{\e})\|^2_m+\|\eta^{\e}\|^2_{m-1}\right)
\left(\e+\s^{-\gamma}+\e^3\|(r^{\e},R^{\e}\|^2_{1,\infty}\right)\,ds.
\end{split}\eeq

\noindent{\bf Step 2.}  The  normal derivatives of $R^{\e}$.

Notice that $Q^{\e}=\sqrt{\e}\p_{\n} R^{\e}.$
Let $|\beta|\leq m-1$ with $m\geq 1$, by applying $\sqrt{\e}Z^{\beta}\p_{\n}$ to \eqref{eqrR}$_2$ and then taking $L^2$ inner product
of the resulting equation with $Z^{\beta}Q^{\e},$  that
\beq \label{S4eq12}
\begin{split}
\frac{1}{2}\frac{d}{dt}\|Z^{\beta}Q^{\e}\|^2_{L^2(\cO)}
&+\int_{\cO}\sqrt{\e}Z^{\beta}\p_{\n}\left(u^{\e}\cdot\nabla R^{\e}+ r^{\e}\cdot\nabla B^{\e}_{app}-B^{\e}\cdot\nabla r^{\e}\right.\\
&\left.-R^{\e}\cdot\nabla u^{\e}_{app}+\sigma^{\e} R^{\e}\right)\cdot Z^{\beta}Q^{\e}
=\int_{\cO}\sqrt{\e}Z^{\beta}\p_{\n} F^{\e}\cdot Z^{\beta}Q^{\e}.
\end{split}\eeq

\noindent $\bullet$ \underline{$J_1:=\int_{\cO}\sqrt{\e}Z^{\beta}\p_{\n} F^{\e}\cdot Z^{\beta}Q^{\e}$.}

It follows  from Lemma \ref{appes} that
\begin{align*}
|J_1|\leq \|\sqrt{\e}Z^{\beta}\p_{\n} F^{\e}\|_{L^2(\cO)}\|Q^{\e}\|_{m-1}\leq C\e^{\frac{1}{4}}\chi_{[0,T]}\|Q^{\e}\|_{m-1}
\leq C\chi_{[0,T]}(\e^{\frac{1}{2}}+\|Q^{\e}\|_{m-1}^2).
\end{align*}

\noindent $\bullet$ \underline{$J_2:=\int_{\cO}\sqrt{\e}Z^{\beta}\p_{\n} (\sigma^{\e} R^{\e})\cdot Z^{\beta}Q^{\e}$.}

Since $\sigma^{\e}=\sigma^0+\e\sigma^2$ is supported in $[0,T]$, we observe that
\begin{align*}
|J_2|\leq C\chi_{[0,T]}(t)(\|Q^{\e}\|^2_{m-1}+\|R^{\e}\|^2_{m-1}).
\end{align*}

\noindent $\bullet$ \underline{$J_3:=\int_{\cO}\sqrt{\e}Z^{\beta}\p_{\n} (r^{\e}\cdot\nabla B^{\e}_{app})\cdot Z^{\beta}Q^{\e}$.}\vspace{0.2cm}

By using Leibniz's formula, Lemma \ref{appes} and Lemma \ref{eqvr}, we find that
\begin{align*}
|J_3|&\leq C\sqrt{\e}\left(\|\p_{\n} r^{\e}\|_{m-1}\|\nabla B^{\e}_{app}\|_{m-1,\infty}
+\|r^{\e}\|_{m-1}\|\p_{\n}\nabla B^{\e}_{app}\|_{m-1,\infty}\right)\|Q^{\e}\|_{m-1}\\
&\leq C\e\chi_{[0,T]}\left(\|\eta^{\e}\|_{m-1}+\sqrt{\e}\|r^{\e}\|_{m}+\sqrt{\e}\t^{-\gamma}\right)\|Q^{\e}\|_{m-1}\\
&\leq C\e\chi_{[0,T]}\big(\e + \|Q^{\e}\|^2_{m-1}+\|\eta^{\e}\|^2_{m-1}+\e\|r^{\e}\|^2_{m}\big).
\end{align*}

\noindent $\bullet$ \underline{$J_4:=\int_{\cO}\sqrt{\e}Z^{\beta}\p_{\n} (R^{\e}\cdot\nabla u^{\e}_{app})\cdot Z^{\beta}Q^{\e}$.}\vspace{0.2cm}

Once again we get, by using Leibniz's formula and  Lemma \ref{appes}, that
\begin{align}\label{n2uapp2}
|J_4|&\leq C\left(\sqrt{\e}\|\p_{\n} R^{\e}\|_{m-1}\|\nabla u^{\e}_{app}\|_{m-1,\infty}
+\|R^{\e}\|_{m-1}\|\sqrt{\e}\p_{\n}\nabla u^{\e}_{app}\|_{m-1,\infty}\right)\|Q^{\e}\|_{m-1}\\
\nonumber&\leq C\t^{-\gamma}(\|Q^{\e}\|^2_{m-1}+\|R^{\e}\|^2_{m-1}).
\end{align}

\noindent $\bullet$ \underline{$J_5:=\int_{\cO}\sqrt{\e}Z^{\beta}\p_{\n} (u^{\e}\cdot\nabla R^{\e})\cdot Z^{\beta}Q^{\e}$.}

By Leibniz's formula, we write
\begin{align*}
J_5=\int_{\cO}Z^{\beta}(\sqrt{\e}\p_{\n}u^{\e}\cdot\nabla R^{\e}
-\sqrt{\e}u^{\e}\cdot\nabla\n\cdot\nabla R^{\e}
+u^{\e}\cdot\nabla Q^{\e})\cdot Z^{\beta}Q^{\e}.
\end{align*}
We then decompose $J_5$ into
$J_5=J_{51}+J_{52}+J_{53}+J_{54}+J_{55}+J_{56}$ with
\begin{align*}
&J_{51}:=\int_{\cO}\sqrt{\e}Z^{\beta}(\p_{\n}u^{\e}_{app}\cdot\nabla R^{\e}
-u^{\e}_{app}\cdot\nabla\n\cdot\nabla R^{\e})\cdot Z^{\beta}Q^{\e},\\
&\begin{aligned}
J_{52}&:=\int_{\cO}\e^2 Z^{\beta}(\p_{\n}r^{\e}\cdot\nabla R^{\e})\cdot Z^{\beta}Q^{\e},\qquad
&&J_{53}:=-\int_{\cO}\e^2 Z^{\beta}(r^{\e}\cdot\nabla \n\cdot \nabla R^{\e})\cdot Z^{\beta}Q^{\e}\\
J_{54}&:=\int_{\cO}u^{\e}\cdot\nabla Z^{\beta}Q^{\e}\cdot Z^{\beta}Q^{\e}, \qquad
&&J_{55}:=\int_{\cO}[Z^{\beta},u^{\e}_{app}\cdot\nabla]Q^{\e}\cdot Z^{\beta}Q^{\e},
\end{aligned}\\
&J_{56}:=\int_{\cO}\e^{\frac{3}{2}}[Z^{\beta},r^{\e}\cdot\nabla]Q^{\e}\cdot Z^{\beta}Q^{\e}
\end{align*}

By Leibniz's formula, one has 
\begin{align*}
|J_{51}|&\leq C\sqrt{\e}\left(\|\p_{\n}u^{\e}_{app}\|_{m-1,\infty}+\|u^{\e}_{app}\|_{m-1,\infty}\right)\|\nabla R^{\e}\|_{m-1}\|Q^{\e}\|_{m-1}\\
&\leq \t^{-\gamma}(\|Q^{\e}\|^2_{m-1}+\|R^{\e}\|^2_{m}).
\end{align*}

We get, by using Proposition \ref{optan} and $\dive r^{\e}=d^{\e}$, that
\begin{gather}
\label{pnrR}\p_{\n}r^{\e}\cdot\nabla R^{\e}=\sum_{1\leq i\leq 3}\p_{\n}r^{\e}\cdot \btau^iZ_iR^{\e}+(d^{\e}-\sum_{1\leq i\leq 3}Z_ir^{\e}\cdot \btau^i)\p_{\n}R^{\e}.
\end{gather}
So that by employing Lemma \ref{SGNM}, Lemma \ref{eqvr} and Lemma \ref{appes}, we obtain
\begin{align*}
&\e^2\|\p_{\n}r^{\e}\cdot\nabla R^{\e}\|_{m-1}\\
&\leq C\e^2\big(\|\p_{\n}r^{\e}\|_{L^{\infty}(\cO)}\|R^{\e}\|_{m}+\|\p_{\n}r^{\e}\|_{m-1}\|R^{\e}\|_{1,\infty}
+\|d^{\e}\|_{m-1,\infty}\|\p_{\n}R^{\e}\|_{m-1}\\
&\hspace{1.3cm}+\|r^{\e}\|_m\|\p_{\n}R^{\e}\|_{L^{\infty}(\cO)}+\|r^{\e}\|_{1,\infty}\|\p_{\n}R^{\e}\|_{m-1}\big)\\
&\leq C\e^{\frac{3}{2}}\big((\|\eta^{\e}\|_{L^{\infty}(\cO)}+\sqrt{\e}\|r^{\e}\|_{1,\infty}+\sqrt{\e}\t^{-\gamma})\|R^{\e}\|_{m}\\
&\hspace{1.3cm}+(\|\eta^{\e}\|_{m-1}+\sqrt{\e}\|r^{\e}\|_m+\sqrt{\e}\t^{-\gamma})\|R^{\e}\|_{1,\infty}\\
&\hspace{1.3cm}+\t^{-\gamma}\|Q^{\e}\|_{m-1}+\|r^{\e}\|_{m}\|Q^{\e}\|_{L^{\infty}(\cO)}+\|r^{\e}\|_{1,\infty}\|Q^{\e}\|_{m-1}\big).
\end{align*}
Hence
\begin{align*}
|J_{52}|\leq C\e^2\t^{-\gamma}+C&\left(\|(r^{\e}, R^{\e})\|^2_m+\|(\eta^{\e},Q^{\e})\|^2_{m-1}\right)\\
&\cdot\left(\e+\e^2\|(r^{\e},R^{\e})\|^2_{1,\infty}+\e^2\|(\eta^{\e},Q^{\e})\|_{L^{\infty}(\cO)}^2\right).
\end{align*}

Notice that $|\n|=1$ in $\mathcal{V}_{\delta_0},$ we get, by using Proposition \ref{optan}, that
\begin{align}\label{rpnR}
r^{\e}\cdot\nabla\n\cdot\nabla R^{\e}=\sum_{1\leq i\leq 3}r^{\e}\cdot\nabla\n\cdot \btau^i Z_i R^{\e}.
\end{align}
By using Lemma \ref{SGNM}, we infer
\begin{align*}
\|r^{\e}\cdot\nabla\n\cdot \nabla R^{\e}\|_{m-1}
\leq C\left( \|r^{\e}\|_{m}\|R^{\e}\|_{L^{\infty}(\cO)}+\|r^{\e}\|_{L^{\infty}(\cO)}\|R^{\e}\|_{m}\right),
\end{align*}
which together with Lemma \ref{eqvr} ensures that
\begin{align*}
|J_{53}|\leq C\e^2\left(\|(r^{\e}, R^{\e})\|^2_m+\|Q^{\e}\|^2_{m-1}\right)\|(r^{\e},R^{\e})\|_{L^{\infty}(\cO)}.
\end{align*}

Since $u^{\e}$ is tangential to the boundary with $\dive u^{\e}=\sigma^{\e}$, we get, by using integration by parts, that
\begin{align*}
|J_{54}|=\frac{1}{2}\left|\int_{\cO}\sigma^{\e}|Z^{\beta}Q^{\e}|^2\right|
\leq C\chi_{[0,T]}\|Q^{\e}\|^2_{m-1}.
\end{align*}

Similar to \eqref{Za1}, one has
\begin{align*}
|J_{55}|\leq C\t^{-\gamma}\|Q^{\e}\|^2_{m-1}.
\end{align*}

Similar to \eqref{unabv}, we get, by using Proposition \ref{optan}, Lemma \ref{SGNM} and  Lemma \ref{unfi},  that
\begin{align*}
|J_{56}|&\leq C \e^{\frac{3}{2}}\left(\|r^{\e}\|_{2,\infty}\|Q^{\e}\|_{m-1}+\|r^{\e}\|_{m+1}\|Q^{\e}\|_{L^{\infty}(\cO)}
+\|d^{\e}\|_{m-1,\infty}\|Q^{\e}\|_{m-1}\right)\|Q^{\e}\|_{m-1}\\
&\leq  \lambda\e\|r^{\e}\|^2_{m+1}+C_{\lambda}\|Q^{\e}\|_{m-1}^2
\left(\e+\t^{-\gamma}+\e^2\|r^{\e}\|^2_{2,\infty}+\e^2\|Q^{\e}\|^2_{L^{\infty}(\cO)}\right).
\end{align*}

By summarizing the above inequalities, we deduce that
\begin{align*}
&|J_5|\leq C\e^2\t^{-\gamma}+\lambda\e\|r^{\e}\|^2_{m+1}
+C_{\lambda}\left(\|(r^{\e},R^{\e})\|^2_{m}+\|(\eta^{\e},Q^{\e})\|^2_{m-1}\right)\\
&\hspace{3.9cm}\cdot\left(\e+\t^{-\gamma}+\e^2\|(r^{\e},R^{\e})\|^2_{2,\infty})
+\e^2\|(\eta^{\e},Q^{\e})\|^2_{L^{\infty}(\cO)}\right).
\end{align*}

\noindent $\bullet$ \underline{$J_6:=\int_{\cO}\sqrt{\e}Z^{\beta}\p_{\n} (B^{\e}\cdot\nabla r^{\e})\cdot Z^{\beta}Q^{\e}$.}

By Leibniz's formula, we write
\begin{align*}
J_6=\int_{\cO}Z^{\beta}\left(\sqrt{\e}\p_{\n}B^{\e}\cdot\nabla r^{\e}
-\sqrt{\e}B^{\e}\cdot\nabla\n\cdot\nabla r^{\e}
+B^{\e}\cdot\nabla (\sqrt{\e}\p_{\n}r^{\e})\right)\cdot Z^{\beta}Q^{\e}.
\end{align*}
We decompose $J_6=J_{61}+J_{62}+J_{63}+J_{64}$ with
\begin{align*}
&J_{61}:=\int_{\cO}\sqrt{\e}Z^{\beta}\left(\p_{\n} B^{\e}_{app}\cdot\nabla r^{\e}
-B^{\e}_{app}\cdot\nabla\n\cdot\nabla r^{\e}\right)\cdot Z^{\beta}Q^{\e},\\
&J_{62}:=\int_{\cO}\e^2Z^{\beta}\left(\p_{\n}R^{\e}\cdot\nabla r^{\e}-R^{\e}\cdot\nabla\n\cdot\nabla r^{\e}\right)\cdot Z^{\beta}Q^{\e},\\
&J_{63}:=\sqrt{\e}\int_{\cO} Z^{\beta}(B^{\e}_{app}\cdot\nabla\p_{\n}r^{\e})\cdot Z^{\beta}Q^{\e},\\
&J_{64}:=\e^2\int_{\cO}Z^{\beta}(R^{\e}\cdot\nabla\p_{\n}r^{\e})\cdot Z^{\beta}Q^{\e}.
\end{align*}

Similar to the estimate of $J_{51}$, we get, by applying  Lemma \ref{eqvr}, that
\begin{align*}
|J_{61}|&\leq C\sqrt{\e}\left(\|\p_{\n}B^{\e}_{app}\|_{m-1,\infty}+\|B^{\e}_{app}\|_{m-1}\right)\|\nabla r^{\e}\|_{m-1}\|Q^{\e}\|_{m-1}\\
&\leq C\e\chi_{[0,T]}+C\e\chi_{[0,T]}(\|r^{\e}\|^2_{m}+\|(\eta^{\e},Q^{\e})\|_{m-1}^2).
\end{align*}

Similar to \eqref{pnrR} and \eqref{rpnR}, we obtain
\begin{gather*}
\p_{\n}R^{\e}\cdot\nabla r^{\e}=\sum_{1\leq i\leq 3}\left(\p_{\n}R^{\e}\cdot\btau^iZ_ir^{\e}-Z_{i}R^{\e}\cdot \btau^i\p_{\n}r^{\e}\right),\\
R^{\e}\cdot\nabla\n\cdot\nabla r^{\e}=\sum_{1\leq i\leq 3}R^{\e}\cdot\nabla\n\cdot\btau^iZ_ir^{\e}.
\end{gather*}

Similar to the estimate of $J_{52}$ and $J_{53}$  we get, by using Lemma \ref{SGNM} and  Lemma \ref{eqvr}, that
\begin{align*}
|J_{62}|\leq C\e^2\t^{-\gamma}+C&\left(\|(r^{\e}, R^{\e})\|^2_m+\|(\eta^{\e},Q^{\e})\|^2_{m-1}\right)\\
&\cdot\left(\e+\e^2\|(r^{\e},R^{\e})\|^2_{1,\infty}+\e^2\|(\eta^{\e},Q^{\e})\|_{L^{\infty}(\cO)}^2\right).
\end{align*}

Whereas it follows from \eqref{nab2r} that
\begin{align*}
|J_{63}|&\leq C\|B^{\e}_{app}\|_{m-1,\infty}\|\sqrt{\e}\nabla\p_{\n}r^{\e}\|_{m-1}\|Q^{\e}\|_{m-1}\\
&\leq \e\chi_{[0,T]}\left(\|\nabla\eta^{\e}\|_{m-1}+\sqrt{\e}\|r^{\e}\|_{m}+\sqrt{\e}\t^{-\gamma}\right)\|Q^{\e}\|_{m-1}\\
&\leq \lambda \e\|\nabla\eta^{\e}\|^2_{m-1}+C_{\lambda}\e\chi_{[0,T]}\left(\|Q^{\e}\|^2_{m-1}+\e\|r^{\e}\|^2_m+\e\right),
\end{align*}
for any $\lambda>0$ and an associated constant $C_{\lambda}>0$.

Finally similar to the proof of \eqref{unabv} and the estimate of $J_{56}$, by applying Proposition \ref{optan}, Lemmas \ref{SGNM}, Lemma \ref{unfi} and Lemma \ref{eqvr}, we obtain
\begin{align*}
|J_{64}|&\leq C\e^{2}\left(\|R^{\e}\|_m\|\p_{\n}r^{\e}\|_{1,\infty}+\|R^{\e}\|_{1,\infty}\|\p_{\n}r^{\e}\|_{m}\right)\|Q^{\e}\|_{m-1}\\
&\leq \lambda\e\|\nabla r^{\e}\|^2_m+C_{\lambda}
(\|Q^{\e}\|^2_{m-1}+\|R^{\e}\|^2_m)(\e+\e^2\|(r^{\e},R^{\e})\|_{2,\infty}^2+\e^2\|\eta^{\e}\|^2_{1,\infty}).
\end{align*}

By summarizing the above estimates, we arrive at
\begin{align*}
|J_6|\leq \lambda\e\|\nabla r^{\e}\|^2_m+\lambda\e\|\nabla \eta^{\e}\|^2_{m-1}
+C\e^2\t^{-\gamma}+C_{\lambda}\left(\|(r^{\e},R^{\e})\|^2_m+\|(\eta^{\e},Q^{\e})\|^2_{m-1}\right)\\
\cdot\left(\e+\t^{-\gamma}+\e^2\|(r^{\e},R^{\e})\|^2_{2,\infty}+\e^2\|\eta^{\e}\|^2_{1,\infty}+\e^2\|Q^{\e}\|^2_{L^{\infty}}\right).
\end{align*}

By substituting
the above estimates into \eqref{S4eq12} and   summing up the resulting inequalities for $|\beta|\leq m-1,$ and then integrating it over $(0,t)$, we
achieve
\beq\label{eqQe}
\begin{split}
\|Q^{\e}(t)\|^2_{m-1}\leq &C\e^{\frac{1}{2}}+C\lambda\e\int_0^t\left(\|\nabla r^{\e}\|^2_{m-1}+\|\nabla\eta^{\e}\|^2_{m-1}\right)ds\\
&+\int_0^t\left(\|(r^{\e},R^{\e})\|^2_m+\|(\eta^{\e},Q^{\e})\|^2_{m-1}\right)\\
&\hspace{1cm}\cdot\left(\e+\s^{-\gamma}+\e^2\|(r^{\e},R^{\e})\|^2_{2,\infty}
+\e^2\|\eta^{\e}\|^2_{1,\infty}+\e^2\|Q^{\e}\|^2_{L^{\infty}}\right)\,ds.
\end{split}\eeq

By summing up \eqref{eqetae} and \eqref{eqQe} and taking  $\lambda$ to be sufficiently small, we achieve \eqref{S4eq10}.
This concludes the proof of Proposition \ref{nores}.
\end{proof}

\subsection{Estimates of the pressure}
 By virtue of the $r^{\e}$  equation in \eqref{eqrR}, we shall decompose the scalar pressure function $\pi^{\e}$ into five parts $\pi^{\e}=\sum_{i=1}^5\pi^{\e}_i,$  where $\pi_i^{\e}$ $ (1\leq i\leq 5)$ are determined respectively by
\begin{equation*}
\left\{
\begin{aligned}
\Delta\pi_1^{\e}&=\dive f^{\e},\quad&&\text{ in }\cO,\\
\p_{\n}\pi_1^{\e}&=f^{\e}\cdot\n,\quad&&\text{ on }\p\cO,
\end{aligned}\right.
\quad \text{ and }\quad
\left\{
\begin{aligned}
\Delta\pi_2^{\e}&=-\p_t d^{\e},\quad&&\text{ in }\cO,\\
\p_{\n}\pi_2^{\e}&=0,\quad&&\text{ on }\p\cO,
\end{aligned}
\right.
\end{equation*}
and
\begin{equation}\label{dpi3}
\left\{
\begin{aligned}
\Delta\pi_3^{\e} &=-\dive(u^{\e}\cdot\nabla r^{\e}+r^{\e}\cdot\nabla u^{\e}_{app}),   \qquad    &&\text{ in }\cO,\\
\p_{\n}\pi_3^{\e}&=-(u^{\e}\cdot\nabla r^{\e}+r^{\e}\cdot\nabla u^{\e}_{app})\cdot\n, \qquad    &&\text{ on }\p\cO,
\end{aligned}\right.
\end{equation}
and
\begin{equation*}\label{dpi4}
\left\{
\begin{aligned}
\Delta\pi_4^{\e} &=\e\Delta d^{\e},\qquad&&\text{ in }\cO,\\
\p_{\n}\pi_4^{\e}&=\e\Delta r^{\e}\cdot\n,\qquad&&\text{ on }\p\cO,
\end{aligned}\right.
\end{equation*}
and
\begin{equation*}\label{dpi5}
\left\{
\begin{aligned}
\Delta\pi^{\e}_5 &=\dive(B^{\e}\cdot\nabla R^{\e}+R^{\e}\cdot\nabla B^{\e}_{app}),\qquad&&\text{ in }\cO,\\
\p_{\n}\pi^{\e}_5&=(B^{\e}\cdot\nabla R^{\e}+ R^{\e}\cdot\nabla B^{\e}_{app})\cdot\n,\qquad&&\text{ on }\p\cO.
\end{aligned}\right.
\end{equation*}

To handle the estimates of $\nabla \pi^{\e}_i,$ we need the following lemma from \cite{LSZ1}:
\begin{lemma}[Lemma 5.7 of \cite{LSZ1}]\label{estpik}
{\sl
Let $\pi $ be determined by
\begin{equation*}
\left\{
\begin{aligned}
\Delta\pi&= \dive f,\qquad &&\text{ in }\cO,\\
\p_{\n}\pi&= f\cdot\n, \qquad&&\text{ on }\p\cO.
\end{aligned}\right.
\end{equation*}
Then for all non-negative integral $m$, there is a constant $C>0$ such that
\begin{align*}
\|\nabla\pi\|_m\leq C\|f\|_m.
\end{align*}
}\end{lemma}

\begin{proposition}\label{espi1}
{\sl
Let $\gamma=\frac{5}{4}$ and $m\in\mathbb{N}$ with $m\leq 5.$ Then for any $t\in \bigl[0,\frac{T}{\e}\bigr]$, there exists a constant $C>0$ so that there hold
\begin{subequations} \label{S4eq14}
\begin{align}
\label{estpi1}\|\nabla \pi^{\e}_1\|_m &\leq C\e^{\frac{1}{4}}\t^{-\gamma},\\
\label{estpi2}\|\nabla \pi^{\e}_2\|_m &\leq C\e^{\frac{1}{4}}\t^{-\gamma},\\
\label{estpi3}\|\nabla \pi^{\e}_3\|_m &\leq C\bigl(\e^{\frac{1}{4}}\t^{-\gamma}+\|r^{\e}\|_m\t^{-\gamma}+\e^{\frac{3}{2}}\|r^{\e}\|_{m+1}(\|r^{\e}\|_{1,\infty}+\t^{-\gamma})\bigr),\\
\label{estpi4}\|\nabla \pi^{\e}_4\|_m &\leq C\bigl(\e^{\frac{1}{4}}\t^{-\gamma}+\e\|\nabla r^{\e}\|_m\bigr),\\
\label{estpi5}\|\nabla \pi^{\e}_5\|_{m-1}&\leq C\bigl(\e\t^{-\gamma}+\e^{\frac{3}{2}}\|R^{\e}\|_{1,\infty}\bigr)\|R^{\e}\|_m,\qquad \text{ for }m \geq 1.
\end{align}
\end{subequations}
}\end{proposition}

\begin{proof}
 \eqref{estpi1} follows directly from  Lemma \ref{estpik} and \eqref{app4}. The proof of inequalities \eqref{estpi2} and \eqref{estpi4} can be found in \cite[Proposition 5.8]{LSZ1}.

Since $\dive r^{\e}=d^{\e}, \dive u^{\e}=\sigma^0+\e\sigma^2,$ we write
\begin{align*}
\dive (u^{\e}\cdot\nabla r^{\e})=\dive \bigl(r^{\e}\cdot\nabla u^{\e}+d^{\e}u^{\e}-(\sigma^0+\e\sigma^2)r^{\e}\bigr).
\end{align*}
On the other hand, notice that $u^{\e}\cdot\n =r^{\e}\cdot\n=0$ on $\p\cO$ and $\nabla\n$ is symmetric, one has
\begin{align*}
u^{\e}\cdot\nabla r^{\e}\cdot\n=-u^{\e}\cdot\nabla\n\cdot r^{\e}=-r^{\e}\cdot\nabla\n\cdot u^{\e}=r^{\e}\cdot\nabla u^{\e}\cdot\n\qquad \text{ on }\p\cO.
\end{align*}
So that in view of \eqref{dpi3}, $\pi^{\e}_3$ satisfies
\begin{equation*}
\left\{
\begin{aligned}
\Delta\pi^{\e}_3&=-\dive\bigl(r^{\e}\cdot\nabla(2u^{\e}_{app}+\e^{\frac{3}{2}}r^{\e})+d^{\e}u^{\e}-(\sigma^0+\e\sigma^2)r^{\e}\bigr),
\quad &&\text{ in }\cO,\\
\p_{\n}\pi^{\e}_3&=-r^{\e}\cdot\nabla(2u^{\e}_{app}+\e^{\frac{3}{2}}r^{\e})\cdot\n,\quad&&\text{ on }\p\cO,
\end{aligned}
\right.
\end{equation*}
from which, Lemma \ref{estpik} and \eqref{rnabr}, we infer
\begin{align*}
\|\nabla\pi^{\e}_3\|_{m}&\leq C\bigl(\|r^{\e}\cdot\nabla u^{\e}_{app}\|_m+\e^{\frac{3}{2}}\|r^{\e}\cdot\nabla r^{\e}\|_m+\|d^{\e}u^{\e}\|_m+\|(\sigma^0+\e\sigma^2) r^{\e}\|_m\bigr)\\
&\leq C\bigl(\e^{\frac{1}{4}}\t^{-\gamma}+\|r^{\e}\|_m\t^{-\gamma}+\e^{\frac{3}{2}}\|r^{\e}\|_{m+1}(\|r^{\e}\|_{1,\infty}+\t^{-\gamma})\bigr).
\end{align*}
This proves \eqref{estpi3}.

While it follows from Lemma \ref{appes}, Lemma \ref{estpik}, \eqref{BnabR} and \eqref{RnabR} that for $m\geq 1$
\begin{align*}
\|\nabla \pi^{\e}_5\|_{m-1}&\leq C\left(\|B^{\e}_{app}\cdot\nabla R^{\e}\|_{m-1}+\e^{\frac{3}{2}}\|R^{\e}\cdot\nabla R^{\e}\|_{m-1}
+\|R^{\e}\cdot\nabla B^{\e}_{app}\|_{m-1}\right)\\
&\leq C\left(\e\chi_{[0,T]}\|R^{\e}\|_m+\e^{\frac{3}{2}}\|R^{\e}\|_{m}\|R^{\e}\|_{1,\infty}\right),
\end{align*} which leads to \eqref{estpi5}. We thus complete the proof of \eqref{S4eq14}.
\end{proof}

Thanks to Proposition \ref{espi1}, we can deal with the pressure terms appearing in \eqref{ZarR} and \eqref{S4eq10}.

\begin{corollary}
{\sl
Let $\gamma=\frac{5}{4}$, $m\in\mathbb{N}$ with $1\leq m\leq 5,$ then there exist constants $c_0, C>0$ such that the solution of \eqref{eqrR} satisfies
\beq\label{esrR}
\begin{split}
&\|(r^{\e},R^{\e})(t)\|^2_{m}+\|(\eta^{\e},Q^{\e})(t)\|^2_{m-1}+c_0\e\int_0^t\left(\|\nabla r^{\e}(s)\|^2_{m}+\|\nabla\eta^{\e}(s)\|^2_{m-1}\right)ds\\
&\leq C\e^{\frac{1}{2}}
+C\int_0^t\left(\|(r^{\e},R^{\e})\|^2_m+\|(\eta^{\e},Q^{\e})\|^2_{m-1}\right)\\
&\quad\cdot\bigl(\e+\s^{-\gamma}+\e^2\|(r^{\e},R^{\e})\|^2_{2,\infty}
+\e^{\frac{3}{2}}\|\nabla r^{\e}\|_{L^{\infty}(\cO)}
+\e^2\|\eta^{\e}\|^2_{1,\infty}+\e^2\|Q^{\e}\|^2_{L^{\infty}(\cO)}\bigr)ds.
\end{split}
\eeq
}\end{corollary}
\begin{proof}
We decompose $\pi^{\e}$ into $\pi^{\e}=\sum_{i=1}^5\pi^{\e}_i.$ Then we get,  by applying Proposition \ref{espi1}, that for any $\lambda>0,$ there exist constants $C,C_{\lambda}>0$ such that
\begin{align*}
&\left|\int_{\cO}Z^{\alpha}\nabla(\pi^{\e}_1+\pi^{\e}_2+\pi^{\e}_3+\pi^{\e}_4)\cdot Z^{\alpha}r^{\e}\right|\\
&\leq C\sum_{1\leq i\leq 4}\|\nabla \pi^{\e}_i\|_m\|r^{\e}\|_m\\
&\leq C\bigl(\e^{\frac{1}{4}}\t^{-\gamma}+\e\|\nabla r^{\e}\|_m+\t^{-\gamma}\|r^{\e}\|_m+\e^{\frac{3}{2}}\|r^{\e}\|_{m+1}(\|r^{\e}\|_{1,\infty}+\t^{-\gamma})\bigr)\|r^{\e}\|_m\\
&\leq \lambda\e\|\nabla r^{\e}\|^2_m+C\e^{\frac{1}{2}}\t^{-\gamma}+C_{\lambda}\|r^{\e}\|^2_m\left(\e+\t^{-\gamma}+\e^2\|r^{\e}\|^2_{1,\infty}\right).
\end{align*}
Whereas by using integration by parts and  Proposition \ref{espi1}, one has
\begin{align*}
\left|\int_{\cO} Z^{\alpha}\nabla\pi^{\e}_5\cdot Z^{\alpha}r^{\e}\right|
&\leq \|\nabla\pi_5^{\e}\|_{m-1}\|r^{\e}\|_{m+1}\\
&\leq C\|R^{\e}\|_m\big(\e\t^{-\gamma}+\e^{\frac{3}{2}}\|R^{\e}\|_{1,\infty}\big)\|r^{\e}\|_{m+1}\\
&\leq\lambda\e\|\nabla r^{\e}\|^2_m+C_{\lambda}\|R^{\e}\|_m^2\big(\t^{-\gamma}+\e^2\|R^{\e}\|^2_{1,\infty}\big).
\end{align*}
As a consequence, we obtain
\beq \label{pret1}
\begin{split}
\left|\int_{\cO}Z^{\alpha}\nabla\pi^{\e}\cdot Z^{\alpha}r^{\e}\right|
\leq\ &\lambda\e\|\nabla r^{\e}\|^2_m+C\e^{\frac{1}{2}}\t^{-\gamma}\\
&+C_{\lambda}\|(r^{\e},R^{\e})\|^2_m\left(\e+\t^{-\gamma}+\e^2\|(r^{\e},R^{\e})\|^2_{1,\infty}\right).
\end{split}\eeq

Since $\eta^{\e}=0$ on the boundary $\p\cO,$ hence $Z^{\beta}\eta^{\e}=0$ on $\p\cO$. By using integration by parts, one has
\begin{align*}\label{pretb0}
\left|\int_{\cO} Z^{\beta}(\sqrt{\e}\chi\CN(\nabla\pi^{\e}))\cdot Z^{\beta}\eta^{\e}\right|
\leq C\sqrt{\e}\|\nabla\pi^{\e}\|_{m-1}\|\nabla\eta^{\e}\|_{m-1}.
\end{align*}
While it follows from Proposition \ref{espi1} and Lemma \ref{eqvr} that
\begin{align*}
\|\nabla\pi^{\e}\|_{m-1} &\leq C\bigl(\e^{\frac{1}{4}}\t^{-\gamma}
+\e\|\nabla r^{\e}\|_{m-1}
+\|(r^{\e},R^{\e})\|_{m}(\e^{\frac{3}{2}}\|(r^{\e},R^{\e})\|_{1,\infty}+\t^{-\gamma})\bigr)\\
\nonumber&\leq C\bigl(\e^{\frac{1}{4}}\t^{-\gamma}
+\sqrt{\e}\|\eta^{\e}\|_{m-1}+\e\|r^{\e}\|_m
+\|(r^{\e},R^{\e})\|_{m}(\e^{\frac{3}{2}}\|(r^{\e},R^{\e})\|_{1,\infty}+\t^{-\gamma})\bigr),
\end{align*}
which gives rise to
\beq\label{pret2}
\begin{split}
&\left|\int_{\cO} Z^{\beta}(\sqrt{\e}\chi\CN(\nabla\pi^{\e}))\cdot Z^{\beta}\eta^{\e}\right|
\leq \lambda\e\|\nabla \eta^{\e}\|^2_{m-1}\\
&\qquad+C_{\lambda}\left(\e^{\frac{1}{2}}\t^{-\gamma}
+\big(\|(r^{\e},R^{\e})\|^2_m+\|\eta^{\e}\|^2_{m-1}\big)\big(\e+\t^{-\gamma}+\e^{3}\|(r^{\e},R^{\e})\|^2_{1,\infty}\big)\right).
\end{split}\eeq

By combining \eqref{ZarR}, \eqref{S4eq10} with \eqref{pret1} and \eqref{pret2}, we achieve \eqref{esrR}.
\end{proof}

\subsection{Estimates of $L^{\infty}$ norms}\label{infes}
\begin{proposition}\label{2inf}
{\sl
There exists a constant $C>0$ such that
\begin{subequations} \label{S4eq15}
\begin{gather}
\label{rinf2}\sqrt{\e}\|r^{\e}\|^2_{2,\infty}\leq C\left(\|r^{\e}\|^2_{5}+\|\eta^{\e}\|^2_{4}+\e\t^{-2\gamma}\right),\\
\label{Rinf2}\sqrt{\e}\|R^{\e}\|^2_{2,\infty}\leq C\left(\|R^{\e}\|^2_{5}+\|Q^{\e}\|^2_{4}+\e\t^{-2\gamma}\right).
\end{gather}
\end{subequations}
}
\end{proposition}
\begin{proof}
We deduce from the anisotropic Sobolev embedding (see \cite[Proposition 20]{masmoudi}) that,
\begin{align}\label{ansb}
\|r^{\e}\|^2_{L^{\infty}(\cO)}\leq C\left(\|\p_{\n}r^{\e}\|_{2}\|r^{\e}\|_{2}+\|r^{\e}\|_{2}^2\right),
\end{align}
which together with Lemma \ref{eqvr} ensures that
\begin{align*}
\sqrt{\e}\|r^{\e}\|_{L^{\infty}(\cO)}^2
&\leq C\left((\|\eta^{\e}\|_{2}+\sqrt{\e}\|r^{\e}\|_{2}+\sqrt{\e}\t^{-\gamma})\|r^{\e}\|_{2}+\sqrt{\e}\|r^{\e}\|^2_{2}\right)\\
&\leq C\left(\|\eta^{\e}\|^2_{2}+\|r^{\e}\|^2_{2}+\e\t^{-2\gamma}\right).
\end{align*}
Along the same line, we can prove \eqref{rinf2} and \eqref{Rinf2}.
\end{proof}

Let
\beq\label{S4eq16}
\begin{split}
E^{\e}(t):&=\|(r^{\e},R^{\e})(t)\|^2_{5}+\|(\eta^{\e},Q^{\e})(t)\|^2_{4},\\
\mathcal{K^{\e}}(t):&=\e^{\frac{3}{2}}\|\nabla r^{\e}(t)\|_{L^{\infty}(\cO)}
+\e^2\|\eta^{\e}(t)\|^2_{1,\infty}+\e^2\|Q^{\e}(t)\|^2_{L^{\infty}(\cO)}.
\end{split}\eeq

\begin{proposition}\label{ult}
{\sl
There exist constants $\e_0,c_0,C_*>0$, such that, when $0<\e\leq \e_0,$ the remainder $r^{\e}$ and $R^{\e}$ satisfy
\begin{align}\label{esrR3}
\sup_{t\in [0,T/{\e}]}E^{\e}(t)
+c_0\e\int_0^{T/{\e}}\left(\|\nabla r^{\e}(s)\|^2_5+\|\nabla \eta^{\e}(s)\|^2_{4}\right)ds\leq C_*\e^{\frac{1}{2}}.
\end{align}
}
\end{proposition}
\begin{proof}
We shall prove \eqref{esrR3} through a continuous argument.
We first observe from Proposition \ref{2inf} that
\begin{align*}
\sqrt{\e}\|(r^{\e},R^{\e})(t)\|^2_{2,\infty}\leq C(E^{\e}(t)+\e\t^{-2\gamma}),
\end{align*}
from which, \eqref{esrR} and \eqref{S4eq16}, we infer
\beq \label{esrR2}
\begin{split}
E^{\e}(t)&+c_0\e\int_0^t\left(\|\nabla r^{\e}\|^2_5+\|\nabla \eta^{\e}\|^2_{4}\right)ds\\
&\leq C_0\e^{\frac{1}{2}}+C_0\int_0^tE^{\e}(s)\left(\e+\s^{-\gamma}+\e^{\frac{3}{2}}E^{\e}(s)
+\mathcal{K}^{\e}(s)\right)ds:=W^{\e}(t),
\end{split} \eeq
for some constants $c_0,C_0>0.$

We denote
\begin{align}\label{defT*}
T^{\e}:=\sup\bigl\{\ t\in \bigl[0,T/{\e}\bigr],\  \int_0^t\mathcal{K}^{\e}(s)ds\leq 1\ \bigr\}.
\end{align}
 Then we observe from \eqref{esrR2} that  $W^{\e}(t)$ satisfies
\begin{equation*}
\begin{cases}
\p_tW^{\e}(t)\leq C_0W^{\e}(t)\bigl(\e+\t^{-\gamma}+\e^{\frac{3}{2}}W^{\e}(t)+\mathcal{K}^{\e}(t)\bigr),\qquad\text{ for }t\in (0,T^{\e}],\\
W^{\e}(0)=C_0\e^{\frac{1}{2}},
\end{cases}
\end{equation*}
from which, we deduce that
\begin{align*}
\tilde{W}^{\e}(t)\leq \frac{C_0\e^{\frac{1}{2}}}{1-C_0^2\e^2 t}\with \tilde{W}^{\e}(t):=W^{\e}(t)e^{-\int_0^tC_0(\e+\s^{-\gamma}+\mathcal{K}^{\e}(s))ds}.
\end{align*}
We take
\begin{equation*}
\e_0=\min\{\frac{1}{2C_0^2T},\frac{1}{(2C_{\mathcal{K}})^4}\},
\end{equation*}
where the constant $C_{\mathcal{K}}$ is determined by Lemma \ref{proKe}.
Then for $0\leq t\leq T^{\e}\leq\frac{T}{\e}, 0<\e\leq \e_0 $,
\begin{align*}
\tilde{W}^{\e}(t)\leq 2C_0\e^{\frac{1}{2}},
\end{align*}
which implies
\begin{align*}
W^{\e}(t)\leq 2C_0\e^{\frac{1}{2}}e^{C_0(T+1+\int_0^{\infty}\s^{-\gamma}ds)},\qquad\text{ for }0< t\leq T^{\e}.
\end{align*}
We take
\begin{align*}
C_*:=2C_0e^{C_0(T+1+\int_0^{\infty}\s^{-\gamma}ds)},
\end{align*}
which is independent of $T^{\e},$ then it follows from \eqref{esrR2} that
\begin{align}\label{eqNm}
E^{\e}(t)+c_0\e\int_0^t\big(\|\nabla r^{\e}(s)\|^2_5+\|\nabla \eta^{\e}(s)\|^2_{4}\big)ds\leq W^{\e}(t)\leq C_*\e^{\frac{1}{2}}
\end{align}
holds for $0\leq t\leq T^{\e}\leq \frac{T}{\e}$.

The proof of Proposition \ref{ult} relies on the following lemma, the proof of which will be postponed after we finish
the proof of Proposition \ref{ult}.

\begin{lemma}\label{proKe}
{\sl
There exists a constant $C_{\mathcal{K}}>0$ which is independent of $\e$, such that, for $T^{\e}\in (0,\frac{T}{\e}]$ defined in \eqref{defT*}, $\mathcal{K}^{\e}(t)$ satisfies
\begin{align}\label{eqKe}
\int_0^{T^{\e}}\mathcal{K}^{\e}(t)dt\leq C_{\mathcal{K}}\e^{\frac{1}{4}}.
\end{align}
}
\end{lemma}

Admitting Lemma \ref{proKe}, if we
assume that $T^{\e}<\frac{T}{\e}$, then for $0<\e\leq \e_0$,
\begin{equation*}
\int_0^{T^{\e}}\mathcal{K}^{\e}(t)dt\leq C_{\mathcal{K}}\e_0^{\frac{1}{4}}\leq \frac{1}{2}.
\end{equation*}
which contradicts the definition of $T^{\e}$ by \eqref{defT*}. This in turn shows that $T^{\e}=\frac{T}{\e}$ and \eqref{esrR3} holds.
\end{proof}

Let us present the proof of Lemma \ref{proKe}

\begin{proof}[\bf{ Proof of Lemma \ref{proKe}}]
 We first get, by applying the anisotropic Sobolev embedding inequalities \eqref{ansb} to $\eta^{\e}$ and $Z_i\eta^{\e},$ $0\leq i\leq 5,$ and using \eqref{eqNm} that
\begin{align*}
\|\eta^{\e}(t)\|_{1,\infty}
&\leq C\big(\|\p_{\n}\eta^{\e}(t)\|^{\frac{1}{2}}_{3}\|\eta^{\e}(t)\|^{\frac{1}{2}}_{3}+\|\eta^{\e}(t)\|_{3}\big)\\
&\leq C(\e^{\frac{1}{8}}\|\nabla \eta^{\e}(t)\|^{\frac{1}{2}}_{3}+\e^{\frac{1}{4}}),\qquad \text{ for }t\in [0,T^{\e}].
\end{align*}
By using \eqref{eqNm} again, we obtain
\begin{equation}\label{K2}
\begin{split}
\int_0^{T^{\e}}\e\|\eta^{\e}(t)\|^2_{1,\infty}dt
&\leq C\e^{\frac{5}{4}}\int_0^{T^{\e}}\|\nabla\eta^{\e}(t)\|_{3}dt+C\e^{\frac{1}{2}}\\
&\leq C\e^{\frac{1}{4}}\Big(\int_0^{T^{\e}}\e\|\nabla\eta^{\e}(t)\|^2_{3}dt\Big)^{\frac{1}{2}}+C\e^{\frac{1}{2}}\leq C\e^{\frac{1}{2}}.
\end{split}
\end{equation}

While it follows from Proposition \ref{2inf} and \eqref{eqNm} that, for $t\in [0,T^{\e}]$,
\begin{align}\label{2inf2}
\|(r^{\e},R^{\e})(t)\|^2_{2,\infty}\leq C\e^{-\frac{1}{2}}\big({E^{\e}}(t)+\e\t^{-2\gamma}\big)\leq C,
\end{align}
 from which and Lemma \ref{eqvr}, we infer
 \beq \label{K4}
\begin{split}
\int_0^{T^{\e}}\e^{\frac{3}{2}}\|\nabla r^{\e}\|_{L^{\infty}(\cO)}dt
&\leq C\Big(\int_0^{T^{\e}}\e^2\|\nabla r^{\e}\|_{L^{\infty}(\cO)}^2dt\Big)^{\frac{1}{2}}\\
&\leq C\Big( \int_0^{T^{\e}}\e(\|\eta^{\e}\|_{L^{\infty}(\cO)}^2+\e\|r^{\e}\|_{1,\infty}^2+\e\t^{-2\gamma})dt\Big)^{\frac{1}{2}}
\leq C\e^{\frac{1}{4}}.
\end{split}\eeq
That is the reason why we  expand $u^{\e}$ to at least order of $\cO(\e^{\frac{3}{2}}).$  Otherwise, it is impossible for us
 to have uniform estimate of $\int_0^{T^{\e}}\|\nabla r^{\e}\|_{L^{\infty}(\cO)}dt.$

In view of \eqref{eqrR}, $\nabla R^{\e}$ satisfies
\begin{align} \label{reste}
\p_t\nabla R^{\e}+u^{\e}\cdot\nabla^2R^{\e}=\nabla F^{\e}-\nabla u^{\e}\cdot\nabla R^{\e}+
\nabla \left(B^{\e}\cdot\nabla r^{\e}+R^{\e}\cdot\nabla u^{\e}_{app}- r^{\e}\cdot\nabla B^{\e}_{app}\right) .
\end{align}
with initial data $\nabla R^{\e}|_{t=0}=0.$ Thanks to \eqref{K4}, we have
\begin{align*}
\int_0^{T^{\e}}\|\nabla u^{\e}(t)\|_{L^{\infty}(\cO)}dt
&\leq \int_0^{T^{\e}}\|\nabla u^{\e}_{app}(t)\|_{L^{\infty}(\cO)}dt
+\int_0^{T^{\e}}\e^{\frac{3}{2}}\|\nabla r^{\e}(t)\|_{L^{\infty}(\cO)}dt\leq C.
\end{align*}
By using the maximum principle for the equation \eqref{reste}, we find
\begin{align*}
\|\nabla R^{\e}(t)\|_{L^{\infty}(\cO)}
\leq &\int_0^t\left(\|\nabla F^{\e}(s)\|_{L^{\infty}(\cO)}+\|\nabla u^{\e}\cdot\nabla R^{\e}(s)\|_{L^{\infty}(\cO)}\right)ds\\
&+\int_0^t\|\nabla \left(B^{\e}\cdot\nabla r^{\e}+R^{\e}\cdot\nabla u^{\e}_{app}
- r^{\e}\cdot\nabla B^{\e}_{app}\right)(s)\|_{L^{\infty}(\cO)}ds.
\end{align*}
It follows from   Lemma \ref{appes} and \eqref{2inf2} that, for $t\in [0,T^{\e}]$,
\begin{align*}
\nonumber\|\nabla F^{\e}\|_{L^{\infty}(\cO)}&\leq C\e^{-\frac{1}{2}}\chi_{[0,T]},\\
\nonumber\|\nabla u^{\e}\cdot\nabla R^{\e}\|_{L^{\infty}(\cO)}
&\leq C\|\nabla R^{\e }\|_{L^{\infty}(\cO)}\big(\s^{-\gamma}+\e^{\frac{3}{2}}\|\nabla r^{\e}\|_{L^{\infty}(\cO)}\big),\\
\nonumber\|\nabla (R^{\e}\cdot\nabla u^{\e}_{app})\|_{L^{\infty}(\cO)}
&\leq \|\nabla R^{\e}\|_{L^{\infty}(\cO)}\|\nabla u^{\e}_{app}\|_{L^{\infty}(\cO)}+\|R^{\e}\|_{L^{\infty}(\cO)}\|\nabla^2 u^{\e}_{app}\|_{L^{\infty}(\cO)}\\
&\leq C\s^{-\gamma}\big(\e^{-\frac{1}{2}}+\|\nabla R^{\e}\|_{L^{\infty}(\cO)}\big),
\end{align*}
and
\begin{align*}
\|\nabla(r^{\e}\cdot\nabla B^{\e}_{app})\|_{L^{\infty(\cO)}}
&\leq \|\nabla r^{\e}\|_{L^{\infty}(\cO)}\|\nabla B^{\e}_{app}\|_{L^{\infty}(\cO)}
+\|r^{\e}\|_{L^{\infty}(\cO)}\|\nabla^2 B^{\e}_{app}\|_{L^{\infty}(\cO)}\\
&\leq C\e\chi_{[0,T]}\big(\|\nabla r^{\e}\|_{L^{\infty}(\cO)}+1\big),\\
\|\nabla (B^{\e}\cdot\nabla r^{\e})\|_{L^{\infty}(\cO)}
&\leq \|\nabla B^{\e}\|_{L^{\infty}(\cO)}\|\nabla r^{\e}\|_{L^{\infty}(\cO)}+\|B^{\e}\cdot\nabla^2 r^{\e}\|_{L^{\infty}(\cO)}\\
&\leq C(\e\chi_{[0,T]}+\e^{\frac{3}{2}}\|\nabla R^{\e}\|_{L^{\infty}(\cO)})\|\nabla r^{\e}\|_{L^{\infty}(\cO)}
+\|B^{\e}\cdot\nabla^2 r^{\e}\|_{L^{\infty}(\cO)}.
\end{align*}
Thanks to Lemma \ref{unfi}, $B^{\e}\cdot\nabla$ involves only tangential derivative, which along with \eqref{2inf2} ensures that
\begin{align*}
\|B^{\e}\cdot\nabla^2 r^{\e}\|_{L^{\infty}(\cO)}\leq C\|B^{\e}\|_{1,\infty}\|\nabla r^{\e}\|_{1,\infty}
\leq C\bigl(\e\chi_{[0,T]}+\e^{\frac{3}{2}}\bigr)\|\nabla r^{\e}\|_{1,\infty}.
\end{align*}

By summarizing the above inequalities, we obtain
\begin{align*}
\|\nabla R^{\e}(t)\|_{L^{\infty}(\cO)}
\leq\ & C\e^{-\frac{1}{2}}
+\int_0^t\|\nabla R^{\e}(s)\|_{L^{\infty}(\cO)}\bigl(\s^{-\gamma}+\e^{\frac{3}{2}}\|\nabla r^{\e}(s)\|_{L^{\infty}(\cO)}\bigr)ds\\
&+\int_0^t\bigl(\e\chi_{[0,T]}+\e^{\frac{3}{2}}\bigr)\|\nabla r^{\e}(s)\|_{1,\infty}\,ds.
\end{align*}
Yet it follows from \eqref{K4} that
\begin{align*}
\int_0^t\bigl(\e\chi_{[0,T]}+\e^{\frac{3}{2}}\bigr)\|\nabla r^{\e}(s)\|_{1,\infty}ds
\leq C\Big(\int_0^{T^{\e}}\e^2\|\nabla r^{\e}\|^2_{1,\infty}dt\Big)^{\frac{1}{2}}
\leq C\e^{\frac{1}{4}}.
\end{align*}
We get, by using \eqref{K4} again and  Gronwall's inequality, that
\begin{align*}
\|\nabla R^{\e}(t)\|_{L^{\infty}(\cO)}\leq C\e^{-\frac{1}{2}},\qquad \text{ for }t\in [0,T^{\e}],
\end{align*}
which in particular implies
\begin{align*}\label{K3}
\int_0^{T^{\e}}\e^3\|\nabla R^{\e}(t)\|^2_{L^{\infty}(\cO)}\leq C\e.
\end{align*}
This together with  \eqref{K2} and \eqref{K4} ensure \eqref{eqKe}. We thus completes the proof of Lemma \ref{proKe}.
\end{proof}

\begin{remark}\label{indexpl}
{\sl
Let us explain the reason why we choose the initial data $(u_a,B_a)$ belonging to $H^{24}\times H^8$.
 Thanks to Proposition \ref{2inf}, we require $m\geq 5$. In view of \eqref{ZarBR}, we require $\|\nabla u^{\e}_{\rm app}\|_{5,\infty}\leq C\t^{-\gamma}$ and $\|\nabla B^{\e}_{\rm app}\|_{5,\infty}\leq C\e\chi_{[0,T]}$. Whereas by virtue of  Sobolev embedding inequality and the construction of $B^2$, for such a process, the minimum regularity for $B_a$ is $H^8(\cO)$. From the construction of $u^{\e}_{app}$ and $\|\nabla u^{\e}_{app}\|_{5,\infty}\leq C\t^{-\gamma}$, the minimum index for integer $p_3$ is $8$.
In view of \eqref{n2uapp1} and \eqref{n2uapp2}, we require $\sqrt{\e}\|\nabla^2 u^{\e}_{app}\|_{4,\infty}\leq C\t^{-\gamma}$, so we need $s_3\geq 7.$
From \eqref{defpi4}, we require $q_3\geq 2.$
From \eqref{ptN} and the definition of $N^{\e}$, we require $k_3\geq 1.$ To close the estimates and get \eqref{uBeT}, we need $\gamma>1$, so we take $\gamma=\frac{5}{4}$ and require $v^3\in C^{1}_{5/4}(\mathbb{R}_+;H^8(\cO;H^7_2(\mathbb{R}_+)))$. By the definition of $k_2,p_2$ in \eqref{index2}, we need $k_2\ge 7, p_2\geq 17$. From the proof of Proposition \ref{lnel}, for such a process, the minimum regularity for $u_a$ is $H^{24}(\cO).$

}
\end{remark}

\subsection{End of proof of Theorem \ref{mainth}}\label{thpf}

Let $\e>0, \gamma=\frac{5}{4}$, we construct the multi-scale asymptotic expansion \eqref{expu} and \eqref{expB} in Section \ref{Sect2}.
Next we shall evaluate $u^{\e}$ and $B^{\e}$ at time $t=\frac{T}{\e}$ in spaces $H^1(\cO)$ and $L^{\infty}(\cO).$

In view of the construction of $B^2$ given by Proposition \ref{prB2},  $B^2$ is supported in $[0,T]$. So that by virtue of \eqref{expB}, one has
\beq\label{S4eq20}
B^{\e}\bigl({T/\e},\cdot\bigr)=\e^{\frac{3}{2}}R^{\e}\bigl({T/\e},\cdot\bigr).
\eeq
While it follows from  \eqref{expu}, the construction of $u^{\e}_{app}$ given by \eqref{dfueapp} and the fact that $u^0, u^2$ is supported in $[0,T]$,
\begin{equation}
\begin{split}\label{S4eq21}
u^{\e}({T/\e},\cdot)=&\ (\e\nabla\phi^2+\e^{\frac{3}{2}}\nabla\phi^3)(T/\e,\cdot)\\
& +\bigl\{\sqrt{\e}v^1+\e(v^2+\n w^2)+\e^{\frac{3}{2}}(v^3+\n w^3)\bigr\}_{\e}({T/\e},\cdot)
+\e^{\frac{3}{2}}r^{\e}({T/\e},\cdot).
\end{split}
\end{equation}
For $1\leq i\leq 3,j=2,3,$ $v^i,w^j\in C^1_{\gamma}(\mathbb{R}_+;H^8(\cO;H^7_2(\mathbb{R}_+))),\phi^j\in C^7_7(\mathbb{R}_+;H^{17}(\cO)).$ We get, by using Lemma \ref{IS} and Sobolev embedding, that
\begin{align*}
\e\|\nabla\phi^j(T/\e)\|_{H^1(\cO)\cap L^{\infty}(\cO)}&\leq
C\e\|\phi^j(T/\e)\|_{H^3(\cO)}\leq C\e\langle{T/\e}\rangle^{-7}\leq C\e^8,\\\
\sqrt{\e}\|(\{v^i\}_{\e},\{w^j\}_{\e})({T/\e})\|_{H^1(\cO)}&\leq C\e^{\frac{1}{4}}\|(v^i,w^j)({T/\e})\|_{H^2(\cO;H^1(\mathbb{R}_+))}
\leq C\e^{\frac{1}{4}}\langle{T/\e}\rangle^{-\gamma}\leq C\e^{\frac{3}{2}},\\
\sqrt{\e}\|(\{v^i\}_{\e},\{w^j\}_{\e})({T/\e})\|_{L^{\infty}(\cO)}&\leq C
\sqrt{\e}\|(v^i,w^j)({T/\e})\|_{H^2(\cO;H^1(\mathbb{R}_+))}\leq C\sqrt{\e}\langle{T/\e}\rangle^{-\gamma}\leq C\e^{\frac{7}{4}}.
\end{align*}
Recall that we denote $\langle{t}\rangle=\sqrt{1+t^2}\geq t$ for $t>0$.

Thanks to Lemma \ref{eqvr} and Proposition \ref{ult}, for $t\in [0,{T/\e}],$ we have
\begin{align*}
\e^{\frac{3}{2}}\|(r^{\e},R^{\e})(t)\|_{H^1(\cO)}
&\leq C\e\left(\|(\eta^{\e},Q^{\e})(t)\|_{L^2(\cO)}+\sqrt{\e}\|(r^{\e},R^{\e})(t)\|_1+\sqrt{\e}\t^{-\gamma}\right)\\
\nonumber&\leq C \e \big(E^{\e}(t)\big)^{\frac{1}{2}}\leq C\e^{\frac{5}{4}}.
\end{align*}
While it follows from Proposition \ref{2inf} and Proposition \ref{ult} that for $t\in [0,T/{\e}]$,
\begin{align*}
\|(r^{\e},R^{\e})(t)\|^2_{L^{\infty}(\cO)}\leq C\e^{-\frac{1}{2}}\big({E^{\e}}(t)+\e\t^{-2\gamma}\big)\leq C.
\end{align*}

Then in view of \eqref{S4eq20} and \eqref{S4eq21}, we get, by summarizing the above estimates, that
\begin{align*}
\|(u^{\e},B^{\e})({T/\e},\cdot)\|_{H^1(\cO)}&\leq C\e^{\frac{5}{4}},\\
\|(u^{\e},B^{\e})({T/\e},\cdot)\|_{L^{\infty}(\cO)}&\leq C\e^{\frac{3}{2}},
\end{align*}
for a constant $C>0$, which implies \eqref{uBeT} and ends the proof of Theorem \ref{mainth}.


\section{Proof of Theorem \ref{LAC}}
\label{sec-lagr}

This section is devoted to the proof of Theorem \ref{LAC}. We use the same strategy as in \cite[Section 6]{LSZ1} to establish the small-time global approximate Lagrangian controllability for the MHD equations. After time scaling \eqref{tsca1}, we shall construct a series of solutions $(u^{\e},p^{\e},B^{\e})$ to equations \eqref{MHDe} with expansions \eqref{expu} and \eqref{expB}.
Since Theorem \ref{LAC} only requires that inequalities \eqref{lc1} and \eqref{lc2} hold for a positive time, we can choose $T_*=\e T$ to be arbitrarily small. Therefore, we do not need the well-prepared dissipation property of the boundary layers. Again, we use a rapid and violent control so that the flow of a solution $u^0$ to a Euler system can transport $\gamma_0$ approximately to $\gamma_1$ at time $T$, where $\gamma_0$ and $\gamma_1$ are isotopic in $\Omega$ and surround the same volume. We need the following proposition due to Krygin.
\begin{proposition}[\cite{krygin}]
{\sl
Assuming that $\gamma_0$ and $\gamma_1$ are as above, there exists a volume-preserving diffeotopy $h\in C^{\infty}([0,1]\times\Omega;\Omega)$ such that $\p_t h$ is compactly supported in $(0,1)\times\Omega$, $h(0,\gamma_0)=\gamma_0$ and $h(1,\gamma_0)=\gamma_1$.
}
\end{proposition}

As a direct consequence, for $h^{-1}(t,x)$ being the inverse mapping of $h(t,x)$ with respect to the space variables, the smooth vector field $X(t,x):=\p_th(t,h^{-1}(t,x))$ is compactly supported in $(0,1)\times\Omega$ and satisfies $\phi^X(t,0,\gamma_0)\subset\Omega,$ $\forall t\in [0,1], \phi^{X}(1,0,\gamma_0)=\gamma_1$ and $\dive X=0$ in $[0,1]\times\Omega$,
where $\phi^X$ is the flow of $X$ defined similar to \eqref{flowu}. Then we can apply the following proposition due to Glass and Horsin.

\begin{proposition}[Proposition 1 of \cite{GH2}]
{\sl
Let $\gamma_0$ be a $C^{\infty}$ contractible two-sphere embedded in $\Omega$ and consider $X\in C^0([0,1];C^{\infty}(\overline{\Omega};\mathbb{R}^3))$ a solenoidal vector field such that
\begin{equation*}
\forall \,t\in [0,1],\quad \phi^{X}(t,0,\gamma_0)\subset\Omega,
\end{equation*}
and let $\gamma_1:=\phi^{X}(1,0,\gamma_0)$. For any $\mu>0$ and $k\in \mathbb{N}$, there exists $\theta\in C_0^{\infty}((0,1)\times\overline{\Omega};\mathbb{R})$ such that
\begin{subequations} \label{S5eq1}
\begin{align}
\label{tta1}\Delta_x\theta=0, \quad&\text{ in } [0,1]\times\Omega,\\
\label{tta2}\p_{\n}\theta=0,  \quad&\text{ on }[0,1]\times(\partial\Omega\setminus\Gamma),\\
\label{tta3}\forall\,t\in [0,1], \quad &\ \phi^{\nabla\theta}(t,0,\gamma_0)\subset\Omega,\\
\label{tta4}\|\phi^{\nabla\theta}(1,0,\gamma_0)&-\gamma_1\|_{C^{k}(\mathbb{S}^2)}< \mu,
\end{align}
\end{subequations}
up to reparameterization.
}
\end{proposition}

Now we need to extend $\nabla\theta$ to $\cO$ and remain tangent to the whole domain $\p\cO.$ We shall prove the following proposition, as a Lagrangian version of Lemma \ref{lmu0}.

\begin{proposition}\label{lagctr}
{\sl There exist a solution $(\tilde{u}^0,\tilde{p}^0,\tilde{\sigma}^0)\in C^{\infty}([0,T]\times\cO; \mathbb{R}^3\times \mathbb{R}\times \mathbb{R})$ to the system
\begin{equation*}
\left\{
    \begin{aligned}
        &\p_t\tilde{u}^0+\tilde{u}^0\cdot\nabla\tilde{u}^0+\nabla \tilde{p}^0=0,\quad &&\text{ in } (0,T)\times\cO,\\
        &\dive \tilde{u}^0=\tilde{\sigma}^0,\quad &&\text{ in } (0,T)\times\cO,\\
        &\tilde{u}^0\cdot\n=0 \quad &&\text{ on } (0,T)\times\p\cO,\\
        &\tilde{u}(0,\cdot)=\tilde{u}(T,\cdot)=0,\quad&&\text{ in }\cO.
    \end{aligned}
\right.
\end{equation*}
Moreover $\tilde{\sigma}^0$ is supported in $\overline{\cO\setminus\Omega},$ $\tilde{u}^0,\tilde{p}^0$ and $\tilde{\sigma}^0$ are compactly supported in $(0,T)$.

}
\end{proposition}

To prove the above proposition, we need the following Runge-type harmonic approximation theorem (Runge's Theorem for $2D$ case and Walsh's Theorem for $3D$ case, see \cite[Section 1.1]{Gar} )
\begin{theorem}\label{harapp}
{\sl
Let $K$ be a compact set in $\mathbb{R}^n$ such that $\mathbb{R}^n\setminus K$ is connected. Then for each function $u$ which is harmonic on an open set containing $K$ and for each positive number $\mu$, there is a harmonic polynomial $v$ such that $|u-v|<\mu$ on $K$.
}
\end{theorem}

\begin{proof}[\bf Proof of Proposition \ref{lagctr}]
Due to the smoothness of $\theta$ and \eqref{tta3}, the flow of $\nabla \theta$ starting from $\gamma_0$ would never reach $\p\Omega$. Therefore there exists a small constant $\delta_*>0$ such that
\begin{equation*}\label{pgm0}
\text{dist}(\phi^{\nabla\theta}(t,0,\gamma_0),\p\Omega)> \delta_*>0,\quad \forall \, t\in [0,1].
\end{equation*}
We take a compact set $K:=K_1\cup K_2$ (see Figure \ref{figure2} for a simple case) with two components defined by
\begin{equation*}\label{K1}
\begin{split}
&K_1:=\bigl\{\ x\in\Omega | \ \text{dist}(x,\p\Omega)\geq\delta_*\ \bigr\},\\
&K_2:=\bigl\{\ x\in \mathbb{R}^3 |\ \text{dist}(x,\p\Omega\setminus\Gamma)\leq\frac{\delta_*}{2}\ \bigr\}.
\end{split}
\end{equation*}
\begin{figure}
  \centering
  \includegraphics{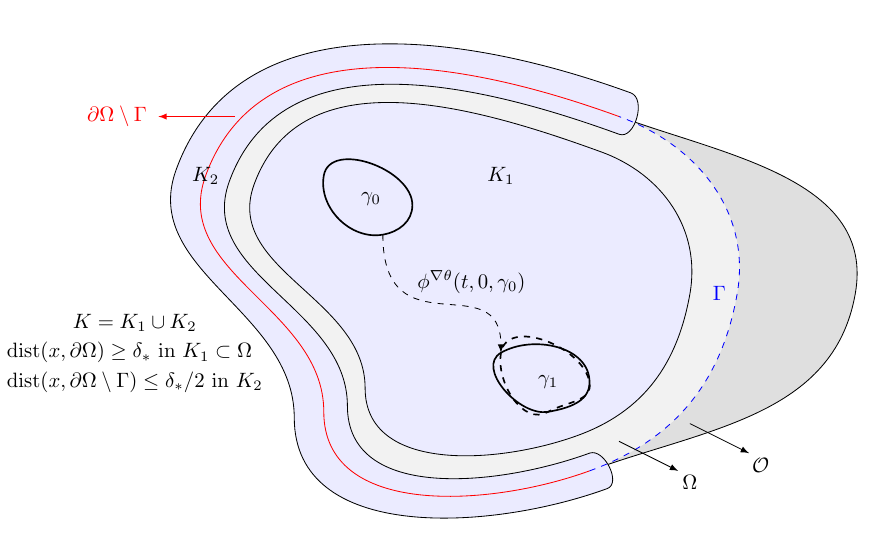}
  \caption{Domains for the Runge-type theorem}\label{figure2}
\end{figure}

We define a harmonic function $u_K$ in a small neighborhood of $K$, such that $u_K=\theta$ in a small neighborhood of $K_1$ and $u_K=0$ in a small neighborhood of $K_2.$ Then for any $\mu>0$ and $k\in\mathbb{N}$, by the compactness of $[0,1]\times K$, there exists a integer $N>0$ such that
\begin{equation}\label{ukts}
    \|u_K(t,x)-u_K(s,x)\|_{C^{k+2}(K)}<\mu,\qquad\text{ for }|t-s|\leq \frac{1}{N}.
\end{equation}
Then by using the above Theorem \ref{harapp}, we can find a harmonic polynomial $v_i (0\leq i\leq N)$ in $\cO$ such that $\|u_K(\frac{i}{N},\cdot)-v_i(\cdot)\|_{C(K)}<\mu$ in $K$. For any fixed integer $k\in\mathbb{N}$, since $u_K(\frac{i}{N},\cdot),v_i(\cdot)$ are harmonic functions, we have
\begin{equation}\label{Ckapp}
\bigl\|u_K(\frac{i}{N},\cdot)-v_i(\cdot)\bigr\|_{C^{k+2}(K)}<C_k\mu,
\end{equation}
for a constant $C_k$ and with a slight shrinking of the set $K,$  which we still denote by $K.$ Taking a partition of unity $\alpha_i$ associated to the covering of $[0,1]$ by $(\frac{i-1}{N},\frac{i+1}{N}),0\leq i\leq N,$ we define
\begin{equation*}
    v_K(t,x):=\sum_{i=0}^N \alpha_i(t)v_i(x).
\end{equation*}
Together with \eqref{ukts} and \eqref{Ckapp}, we get
\begin{equation*}
    \sup_{0\leq t\leq 1}\|u_K(t,\cdot)-v_K(t,\cdot)\|_{C^k(K)}\leq C_k\mu,
\end{equation*}

We introduce a smooth cut-off function $\chi_K$ in $\cO$ such that $\chi_K(x)=1$ when $x\in \Omega$ or $\text{dist}(x,\Gamma)\leq \frac{\delta_*}{16}$, and $\chi_K(x)=0$ when $x\in\cO\setminus\Omega$ and $\text{dist}(x,\Gamma)\geq \frac{\delta_*}{8}$.
Then we define a harmonic function $\Psi_K$ to compensate the normal derivatives of $\chi_K v_K$ by
\begin{equation*}
\left\{
    \begin{aligned}
        \Delta \Psi_K&=0,\quad&&\text{ in }\cO,\\
        \p_{\n}\Psi_K&=\p_{\n}(\chi_K v_K),\quad&&\text{ on }\p\cO.
    \end{aligned}\right.
\end{equation*}
Since $u=0$ in $K_2$,  by combining \eqref{Ckapp} and using classical regularity estimates for elliptic equations wo obtain $\|\nabla \Psi_K\|_{C^k(\cO)}<C\mu$ for some constant $C$ depending on $k$ and $\chi_K$.
We define
\begin{equation*}
\tilde{\theta}=\chi_Kv_K-\Psi_K,
\end{equation*}
then $\tilde{\theta}$ satisfies
\begin{align}
\nonumber    \Delta \tilde{\theta}&=0,\quad \text{ in }\Omega,\\
\nonumber    \p_{\n}\tilde{\theta}&=0,\quad \text{ on }\p\cO,\\
\label{tak}   \|\tilde{\theta}-\theta&\|_{C^{k+1}(K)}<C_k\mu,
\end{align}
for a constant $C_k.$ From the definition of the flow $\phi^{\nabla\tilde{\theta}},\phi^{\nabla\theta}$, \eqref{tta4} and \eqref{tak}, we can easily get
\begin{equation}
    \|\phi^{\nabla\tilde{\theta}}(1,0,\gamma_0)-\gamma_1\|_{C^k(\mathbb{S}^2)}<C_k\mu,
\end{equation}
for a constant $C_k$. By taking $\mu$ small enough, we obtain that
\begin{equation*}
\phi^{\nabla\tilde{\theta}}(t,0,\gamma_0)\subset\Omega.
\end{equation*}
Since $\theta$ is compactly supported in $(0,1)$ so does $v_K,\Psi_K$ and $\tilde{\theta}$.

We define
\beq\label{nu0}
\begin{split}
\tilde{u}^0(t,x)&=\frac{1}{T}(\nabla\tilde{\theta})(\frac{t}{T},x),\\
\tilde{p}^0(t,x)&=\frac{1}{T}(\p_t\tilde{\theta}+\frac{1}{2}|\nabla\tilde{\theta}|^2)(\frac{t}{T},x),\\
\tilde{\sigma}^0(t,x)&=\frac{1}{T}(v_K\Delta\chi_K+2\nabla\chi_K\cdot\nabla v_K)(\frac{t}{T},x).
\end{split}\eeq
Then $(\tilde{u}^0,\tilde{p},\sigma^0)$ meets all the requirements in Proposition \ref{lagctr}.
\end{proof}

\begin{proof}[\bf Proof of theorem \ref{LAC}]

We perform the time scaling \eqref{tsca1}. We can construct the  multi-scale asymptotic expansions \eqref{expu} and \eqref{expB} to system \eqref{MHDe} in $(0,T)\times\Omega$. The only difference is that we replace the first term $u^0$ from the solution of Euler equation \eqref{euler0} given by Lemma \ref{lmu0} to be $\tilde{u}^0$ given by Proposition \ref{lagctr}.

Since $\tilde{u}^0$ satisfies all the conditions of Lemma \ref{lmu0} except that $\tilde{u}^0$ doesn't satisfies the flushing property \eqref{flush}, which is unnecessary. We only need to solve system \eqref{MHDe} in $(0,T)\times\cO$, there is no need for the flushing condition \eqref{flush} to guarantee the decay properties of the boundary profiles.
Therefore we can construct multi-scale asymptotic expansions \eqref{expu} and \eqref{expB} with approximate solution $(u^{\e}_{app},B^{\e}_{app})$ satisfies Lemma \ref{appes} for $t\in [0,T]$.
Then the remainder term $(r^{\e},R^{\e})$ satisfies a similar version of \eqref{esrR3}, which is
\begin{align*}
\sup_{t\in [0,T]}E^{\e}(t)+c_0\e\int_0^T(\|\nabla r^{\e}\|^2_5+\|\nabla \eta^{\e}\|^2_4)ds\leq C_*\e^{\frac{1}{2}},
\end{align*}
for some constants $c_0,C_*,\e_0$ and for all $0\leq \e\leq \e_0.$

The rest of this Section is devoted to proving \eqref{lc1} and \eqref{lc2}.
For $t\in (0,T),x\in \cO,$ by the definition of flow $\phi^{u^{\e}}$ and $\phi^{\tilde{u}^0},$
\begin{align*}
&\p_t\big(\phi^{u^{\e}}(t,0,x)-\phi^{\tilde{u}^0}(t,0,x)\big)=u^{\e}(t,\phi^{u^{\e}}(t,0,x))-\tilde{u}^0(t,\phi^{\tilde{u}^0}(t,0,x))\\
&=\tilde{u}^{0}(t,\phi^{u^{\e}}(t,0,x))-\tilde{u}^0(t,\phi^{\tilde{u}^0}(t,0,x))+(u^{\e}-\tilde{u}^0)(t,\phi^{u^{\e}}(t,0,x)),
\end{align*}
from which, we infer
\beq \label{pueu0}
\begin{split}
\|(\phi^{u^{\e}}-\phi^{\tilde{u}^0})(t,0,\cdot)\|_{L^{\infty}(\cO)}
\leq &\ C\int_0^t \|\nabla \tilde{u}^0(s)\|_{L^{\infty}(\cO)}\|(\phi^{u^{\e}}-\phi^{\tilde{u}^0})(s,0,\cdot)\|_{L^{\infty}(\cO)\,ds}\\
&+\int_0^tC\|(u^{\e}-\tilde{u}^0)(s)\|_{L^{\infty}(\cO)}\,ds.
\end{split}\eeq
It follows from Lemma \ref{appes} that there is a constant $C>0$ such that
\begin{align}\label{ueinf}
\|(u^{\e}_{app}-\tilde{u}^0)(t)\|_{L^{\infty}(\cO)}\leq C\sqrt{\e},\quad\forall\,t\in [0,T].
\end{align}
From Proposition \ref{2inf} and Proposition \ref{ult}, there is a constant $C>0$ such that
\begin{align}\label{rRinfl}
\|r^{\e}(t)\|_{L^{\infty}(\cO)}\leq C,\quad\forall\, t\in [0,T].
\end{align}
Therefore there is a constant $C>0$ such that
\begin{align}\label{ueu0}
C\int_0^t\|(u^{\e}-\tilde{u}^0)(s)\|_{L^{\infty}(\cO)}ds\leq CT\sqrt{\e},\quad\forall\,t\in[0,T].
\end{align}
From the smoothness of $\theta$ and the definition of $\tilde{u}^0$, $\|\nabla \tilde{u}^0\|_{L^{\infty}(\cO)}\leq C/T$ for a constant $C>0$ and for any $t\in [0,T].$ By combining \eqref{pueu0} with \eqref{ueu0}, and by using Gronwall's inequality we find that
\begin{align}
\|(\phi^{u^{\e}}-\phi^{\tilde{u}^0})(t,0,x)\|_{L^{\infty}(\cO)}\leq CTe^{C}\sqrt{\e},\quad\forall\,t\in [0,T],
\end{align}
for a constant $C>0$.

By the definition of the flow $\phi^{u^{\e}}$ and time scaling \eqref{tsca1},
\begin{align*}\label{psca}
\phi^{u^{\e}}(t,0,x)=\phi^u(\e t,0,x)\quad\text{ and }\quad
\phi^{\tilde{u}^0}(t,0,x)=\phi^{\nabla\tilde{\theta}}(t/T,0,x)\quad \forall \, t\in [0,T],x\in \Omega.
\end{align*}
Together with \eqref{tta4}, we obtain
\begin{equation*}
\begin{split}
\|\phi^u(\e T,0,\gamma_0)-\gamma_1\|_{L^{\infty}}
&\leq \|\phi^{u^{\e}}-\phi^{\tilde{u}^0}\|_{L^{\infty}(\Omega)}+\|\phi^{\nabla\tilde{\theta}}(1,0,\gamma_0)-\gamma_1\|_{L^{\infty}}\\
&\leq C_*\sqrt{\e}+C_k\mu,
\end{split}
\end{equation*}
for some constant $C_*,C_k>0$.

For any $\epsilon>0$, we can choose small positive numbers $\e$ and $\mu$ such that $\max\{C_*\sqrt{\e},C_k\mu\}<\min\{\epsilon/2,\delta_*/2\}$ and take $T_*=\e T$, thus
\begin{align}\label{phiinf}
\|\phi^u(T_*,0,\gamma_0)-\gamma_1\|_{L^{\infty}}\leq \epsilon.
\end{align}
By the choice of $\varepsilon$
\begin{align*}
\phi^{u}(t,0,\gamma_0)\subset\Omega,\quad\forall\,t\in [0,T_*],
\end{align*}
which is the first part of Theorem \ref{LAC}.

Now we prove the second part of Theorem \ref{LAC}. If we impose more regularities on the initial data, say that $(u_0,B_0)\in H^{24+2k}(\Omega)\times H^{8+k}(\Omega)$ for an integral $k\in\mathbb{N}_+,$  then we can construct multi-scale asymptotic expansions \eqref{expu} and \eqref{expB} with $u^0$ being replaced by $\tilde{u}^0$ defined by \eqref{nu0}. Moreover, each terms of the expansions will be more regular. Roughly speaking, they can gain $k$ tangential space derivatives than Lemma \ref{appes} and the remainder terms can gain $k$ tangential derivatives than Proposition \ref{ult}. The indexes $(24+2k,8+k)$ is deduce from Proposition \ref{propv} and the definition of indexes $(\gamma_i,k_i,p_i,s_i,q_i)$ in Section \ref{cexp1}. Thus we can get the counterpart of  \eqref{ueinf} and \eqref{rRinfl}, that is
\begin{align*}
\|Z^{\alpha}(u^{\e}-\tilde{u}^0)(t)\|_{L^{\infty}(\Omega)}&\leq C\sqrt{\e},\\
\|Z^{\alpha}r^{\e}(t)\|_{L^{\infty}(\Omega)}&\leq C,
\end{align*}
hold for any $|\alpha|\leq k,\, t\in [0,T].$

Then we can prove \eqref{lc3} the counterpart of \eqref{phiinf} for higher order space derivatives, see \cite[Equation (23)]{koch}. This estimate is performed in a compact set $K\subset \Omega$ such that an open neighborhood of $\cup_{t\in [0,\e T]}\phi^{u}(t,0,\gamma_0)$ is contained in $K$. The existence of such a $K$ is guaranteed by \eqref{phiinf}. Since $K$ is a compact subset inside $\Omega$ and has a positive distance to $\p\Omega$, $\|(r^{\e},R^{\e})\|_{m}$ is equivalent to the usual Sobolev norm of order $m$ by the definition of conormal Sobolev spaces in Section \ref{css}. The details are left to the reader.

\end{proof}

\section*{Acknowledgement}

J. Liao is supported by National Natural Science Foundation of China under Grant No. 12301238.
P. Zhang is partially  supported by National Key R$\&$D Program of China under grant 2021YFA1000800 and by National Natural Science Foundation of China under Grants  No. 12421001, No. 12494542 and No. 12288201.

\section*{Declarations}

\subsection*{Conflict of interest} The authors declare that there are no conflicts of interest.

\subsection*{Data availability}
This article has no associated data.

\end{document}